\documentclass[11pt]{amsart}

\usepackage{xcolor}
\usepackage[draft,inline,nomargin]{fixme}
\usepackage[top=1in,bottom=1in,left=1in,right=1in]{geometry}

\usepackage{multirow}
\usepackage{array}
\usepackage{booktabs}

\usepackage{cellspace}
\fxsetup{theme=color,mode=multiuser}
\FXRegisterAuthor{mac}{amac}{\color{red}Angie}
\FXRegisterAuthor{hm}{ahm}{\color{blue}Hannah}

\usepackage{enumerate}
\usepackage{amsthm}
\usepackage{thmtools}
\usepackage{amsmath}
\usepackage{amsfonts,amssymb}
\usepackage{mathrsfs}
\usepackage{graphicx}
\usepackage{hhline}
\usepackage{blkarray}

\usepackage{hyperref}

\usepackage[all,arc,curve,color,frame]{xy}
\usepackage{colordvi}
\usepackage{vmargin}

\usepackage{pdflscape}
\usepackage{afterpage}
\usepackage{caption}

\usepackage{amscd}
\usepackage{xcolor}

\numberwithin{equation}{section}

\makeatletter
\@namedef{subjclassname@2010}{%
 \textup{2010} Mathematics Subject Classification}
\makeatother

\declaretheorem[style=plain,parent=section]{theorem}
\declaretheorem[style=plain,sibling=theorem]{corollary}
\declaretheorem[style=plain,sibling=theorem]{lemma}

\declaretheorem[style=plain,sibling=theorem]{proposition}

\declaretheorem[style=plain,sibling=theorem]{question}

\declaretheorem[style=definition,sibling=theorem]{definition}

\declaretheorem[style=definition, qed=\hfill $\diamond$, sibling=definition]{example}
\declaretheorem[style=remark,sibling=theorem]{remark}

\newcommand{\RR}{\ensuremath{\mathbb{R}}}

\newcommand{\PP}{\ensuremath{\mathbb{P}}}
\newcommand{\resK}{\ensuremath{\widetilde{K}}}
\newcommand{\CC}{\ensuremath{\mathbb{C}}}

\newcommand{\ZZ}{\ensuremath{\mathbb{Z}}}

\newcommand{\trop}{\text{trop}}
\newcommand{\Trop}{\text{Trop}}
\newcommand{\val}{\operatorname{val}}

\newcommand{\ww}{\omega}

\newcommand{\TP}{\mathbb{T}}
\newcommand{\TPr} {\ensuremath{\TP\PP}}
\newcommand{\PS}{\CC\{\!\{t
\}\!\}}
\newcommand{\PSR}{\RR\{\!\{t
\}\!\}}

\newcommand\be{\begin{equation}}
\newcommand\ee{\end{equation}}
\newcommand\bea{\begin{eqnarray}}
\newcommand\eea{\end{eqnarray}}
\newcommand\bi{\begin{itemize}}
\newcommand\ei{\end{itemize}}
\newcommand\ben{\begin{enumerate}}
\newcommand\een{\end{enumerate}}
\newcommand\ba{\begin{array}}
\newcommand\ea{\end{array}}
\newcommand\bpf{\begin{proof}}
\newcommand\epf{\end{proof}}
\newcommand\bex{\begin{exercise}}
\newcommand\eex{\end{exercise}}

\providecommand{\Sn}[1]{\ensuremath{\mathfrak{S}_{#1}}}

\newcommand{\KK}{\ensuremath{\mathbb{K}}}
\newcommand{\KKr}{\ensuremath{\KK_{\RR}}}

\newcommand{\Star}{\ensuremath{\operatorname{Star}}}
\newcommand{\len}{\ensuremath{\operatorname{len}}}
\newcommand \sage{\texttt{Sage}}

\title{Combinatorics and real lifts of bitangents to tropical quartic curves}

\author [Maria Angelica Cueto and Hannah Markwig]
{
Maria Angelica Cueto
 \and Hannah Markwig${}^{\S}$
}
\thanks{${\S}$ \emph{Corresponding author}}
\date{\today}

\begin{document}

\begin{abstract}
Smooth algebraic plane quartics over algebraically closed fields of characteristic different than two have 28 bitangent lines. Their tropical counterparts often have infinitely many bitangents. They are grouped into seven equivalence classes, one for each  linear system associated to an effective tropical theta characteristic on the tropical quartic. We show such classes determine tropically convex sets and provide a complete combinatorial classification of such objects into 41 types (up to symmetry).

  The occurrence of a given class is determined by both the combinatorial type and the metric structure of the  input tropical plane quartic. We use this result to provide explicit sign-rules to obtain real lifts for each tropical bitangent class, and  confirm that each one has either zero or exactly four real lifts, as previously conjectured by Len and the second author. Furthermore, such real lifts are always totally-real. 
\end{abstract}

\maketitle

\section{Introduction}\label{sec:introduction}

\emph{Superabundance} phenomena in tropical geometry pose a challenge to addressing enumerative geometry questions  by combinatorial means~\cite{bir.gat.sch:17,Mi03,ran:17, spe:14}. 
For example, smooth tropical plane quartic curves can have infinitely many bitangent tropical lines, as opposed to the count of 28 bitangents to smooth algebraic plane quartic curves  over algebraically closed fields of characteristic different than two~\cite{bak.len.ea:16}. For an example with exactly seven tropical bitangents, we refer to~\cite[Figure 12b]{gei.pan:21}.

Basic duality between $\RR^2$ and the space of non-degenerate tropical lines in $\RR^2$ identifies each such line with (plus or minus) its unique vertex. 
This paper gives a combinatorial characterization of these infinite sets of points and sheds light on this question over real closed valued fields, such as the real Puiseux series $\PSR$, related to Pl\"ucker's famous count of real bitangents to  real quartic curves~\cite{Plucker}. 

The existence of infinitely many tropical bitangents was first shown by  Baker \emph{et al.}~\cite{bak.len.ea:16} using the theory of divisors on tropical curves and their linear equivalences, encoded by chip-firing moves~\cite{bak.jen:16}. Combinatorially, tropical quartic curves correspond to metric graphs of genus three (depicted in~\autoref{fig:PossibleSkeletons}.) Out of all five graphs, only the first four can be realized as skeleta of smooth tropical plane curves dual  to unimodular triangulations of the $4$-dilated $2$-simplex. 
The relevant length data of these graphs  is also linearly restricted.

Even though the tropical count is infinite, the collection of tropical bitangent lines can be grouped together into seven equivalence classes, under perturbations that preserve the tangencies~\cite{bak.len.ea:16}. These \emph{bitangent classes} are polyhedral complexes in $\RR^2$. To highlight the interactions between a given tropical quartic curve $\Gamma$ and its bitangent classes, we can further refine the structure of each class using the subdivision of $\RR^2$ induced by $\Gamma$. We define the \emph{shape} of a tropical bitangent class to be this refined combinatorial structure.

The symmetric group $\Sn{3}$ acts on bitangent classes and their shapes. 
Our first main result is a complete combinatorial classification of such objects up to $\Sn{3}$-symmetry:

\begin{theorem}\label{thm:combclass}
  There are 41 shapes of bitangent classes to generic tropicalized plane quartics, up to symmetry (see~\autoref{fig:2dCells}.) All of them are  min-tropical convex sets. 
\end{theorem}
\noindent
This classification only relies on the duality between tropical smooth plane quartic curves and unimodular triangulations of the standard simplex of side length four. 

Standard duality identifying a tropical line with the negative of its vertex seems more natural, since incidence relations would be preserved~\cite{BJLR17}. With this convention, the associated bitangent classes would be max-tropical convex sets.
We choose to avoid changing the sign of the vertex since the tangency points can be easily determined from it.

All shapes in~\autoref{fig:2dCells} are color-coded to highlight their refined combinatorial structure. Black and gray cells correspond to those missing the curve, whereas red ones lie on it. Unfilled vertices correspond to vertices of $\Gamma$. 

The proof of this statement is given in~\autoref{sec:comb-class-bitang}. It involves a careful analysis of possible combinations of tangency points (up to symmetry) and the description of local moves of the vertex of the bitangent line $\Lambda$ that preserve tangencies. 
A simple inspection shows that they are all min-tropical convex sets, most of which are not finitely generated. 
The latter contrasts with the construction of complete linear systems on abstract tropical curves done by Haase-Musiker-Yu~\cite{haa.mus.yu:12}. The difference arises precisely due to the choice of embedding.

Whenever superabundance is observed, lifting questions arise naturally.  
Len and the second author~\cite{len.mar:20} proved that for generic choices, each class lifts to exactly four bitangents if the classical quartic curve is defined over a non-Archimedean algebraically closed field $\KK$. 
This fact was independently proven by Jensen and  Len~\cite{jen.len:18} (removing this mild genericity constraints) by exploiting the classical connection between bitangent lines and theta characteristics in the tropical setting~\cite{MZ08,Zh10}, and by Chan and Jiradilok~\cite{cha.jir:17} in the special case where the  underlying non-Archimedean skeleton of the tropical curve is the complete graph on four vertices.

Although each bitangent class has exactly four lifts~\cite[Theorem 4.1]{len.mar:20}, their number can be realized in various ways and by more than one member of the given class. \autoref{fig:2dCells} shows which members have  non-trivial lifting multiplicities (with values one, two or four.)

By definition, the tropical theta characteristics  of a tropical quartic are extremely sensitive to the underlying metric structure on its skeleton~\cite[Appendix]{bak.len.ea:16}. The same is true for its bitangent classes. In particular, the seven shapes occurring can vary  within a single chamber in the secondary fan of the standard simplex of side-length four.
Nonetheless, the presence of a bitangent class of a fixed shape  imposes restrictions on the Newton subdivision of the quartic curve with tropicalization $\Gamma$. Our findings are summarized in~\autoref{cor:combClass}. \autoref{fig:NP} shows the relevant cells associated to each representative shape. Edges are color-coded to emphasize the combinatorial  tangency types.

Questions involving the realness of bitangent lines to plane quartics can be traced back to Pl\"ucker~\cite{Plucker}. As shown by Zeuthen~\cite{Zeuthen}, the answer depends on the topology  of the underlying smooth real quartic curve viewed in the real projective plane~\cite[Table 1]{PSV}. The last two columns of~\autoref{tab:ex} show these numbers in four of the six possible topological types. The missing three topological types  (two nested ovals, one oval or an empty curve) admit exactly four real bitangent lines each.

The fact that the number of real bitangents of a real plane quartic and the number of complex lifts to any tropical bitangent class are always a multiple of four suggests the question whether real lifts to a given bitangent class also come in multiples of four. Using Tarski's Principle~\cite{JL89} we can count real bitangents by lifting tropical bitangents to a real closed filed $\KKr$, thus providing  a positive answer to this question:

\begin{theorem}{\cite[Conjecture 5.1]{len.mar:20}}\label{thm:reallifting}
Let $\Gamma$ be a generic tropicalization of a smooth plane quartic defined over a real closed complete non-Archimedean valued field $\KKr$. Then, a bitangent class of a given shape has either zero or exactly four lifts to real bitangents to a quartic curve when its real locus is near the tropical limit.
\end{theorem}
Rather than venturing for a tropical analog between real theta characteristics and real bitangents to plane quartics~\cite{GH81, Kra98, kum:19}, we choose an algorithmic approach which has the potential to be used in the arithmetic setting~\cite{lar.vog:19, mar.pay.sha:22}.
Our proof of~\autoref{thm:reallifting}  builds on the lifting techniques developed in~\cite{len.mar:20}, which we review in~\autoref{sec:lift-trop-bitang}.  \autoref{tab:lifting} provides precise necessary and sufficient conditions for the existence of real lifts to each bitangent class in terms of the signs of relevant vertices in the Newton subdivision. The positivity conditions in the table are due to the presence of radicands in the formulas for computing the initial terms of the coefficients of the classical line lifting a given tropical bitangent.

 Once the realness of a given bitangent is established, it is natural to ask whether the tangency points are also real, i.e., if the bitangent to the real quartic curve is totally real or not. ~\autoref{cor:totallyreal} shows that for generic plane curves defined over $\KKr$ with smooth generic tropicalizations in $\RR^2$, any real bitangent line to them is always totally real. Thus, new examples will only be captured by tropical geometry once the current  methods are extended to non-smooth tropical plane quartics~\cite{LL17} or smooth ones embedded in higher dimensional tori. We leave this task to future work.

 The rest of the paper is organized as follows. \autoref{sec:trop-plane-curv} reviews the construction of tropical bitangents to tropical plane curves and its connection to tropical theta characteristics.  We provide a combinatorial classification of local tangencies and recall the main techniques for lifting tropical bitangents from~\cite{len.mar:20} under mild genericity conditions. Both tools play a central role in the proofs of 
 {Theorems}~\ref{thm:combclass} and~\ref{thm:reallifting}.

 In~\autoref{sec:bitang-class-shap} we  define  bitangent classes and introduce their combinatorial refinements, called shapes. An analysis of local moves for points on each bitangent class, discussed in~\autoref{lm:localMoves}, allows us to conclude each class is a connected polyhedral complex. Our last result characterizes  unbounded  cells in suitable classes. \autoref{sec:comb-class-bitang} contains the proof of~\autoref{thm:combclass}.

 \autoref{sec:lift-trop-bitang} discusses local lifting multiplicities both over real closed valued fields and their algebraic closure for all bitangent classes. \autoref{thm:combinationsOfTangencies} determines all combinations of local tangency types that can arise from tropical bitangents under mild genericity conditions. 
         {Propositions}~\ref{pr:realLifts3a3c}, \ref{pr:realLift46a} and~\autoref{lm:tangency2} provide necessary and sufficient local lifting conditions over $\KKr$ for each tangency type with multiplicity two. Lifting formulas in the presence of multiplicity four tangencies are provided in~\autoref{sec:appendix1}.

 \autoref{sec:real-lift-trop} contains the proof of~\autoref{thm:reallifting}.  \autoref{sec:trop-totally-real} confirms our lifting techniques only produces totally real bitangents, which manifest the geometry behind the genericity conditions imposed on the input smooth plane quartics. We conclude with some open questions and directions to pursue in the future.

\subsection{How to use this paper}
For any given quartic polynomial $q(x,y)$ defined over the field of real Puiseux series $\PSR$,~\autoref{thm:reallifting} and~\autoref{tab:lifting} provide an easy way to decide which of the seven bitangent classes of its tropicalization $\Gamma =\Trop\,V(q)$ lift to the reals: it suffices to check the positivity of appropriate products of its coefficients.

Building on~\cite{len.mar:20}, we can even determine which member of each bitangent class lifts to a classical bitangent line to $V(q)$, and how many lifts does it have. 
Furthermore, since the signs of the coefficients in $q$ determine the topology of the real quartic curve $V_{\RR}(q)$ close to the tropical limit by means of Viro's Patchworking method~\cite{IMS09,Viro}, our methods give a way of verifying Pl\"ucker's classical Theorem in concrete examples. We refer the reader to the Polymake extension~\texttt{TropicalQuarticCurves} and the database entry \texttt{QuarticCurves} in \texttt{polyDB} recently developed by Geiger and Panizzut~\cite{gei.pan:21bis} to provide a tropical proof of Pl\"ucker and Zeuthen's count~\cite[Theorem 1]{gei.pan:21}.

We illustrate this ideas in the following example. The connection to tropical theta characteristics is discussed in~\autoref{ex:abstract212I}.

\begin{figure}[tb]
  \includegraphics[scale=0.33]{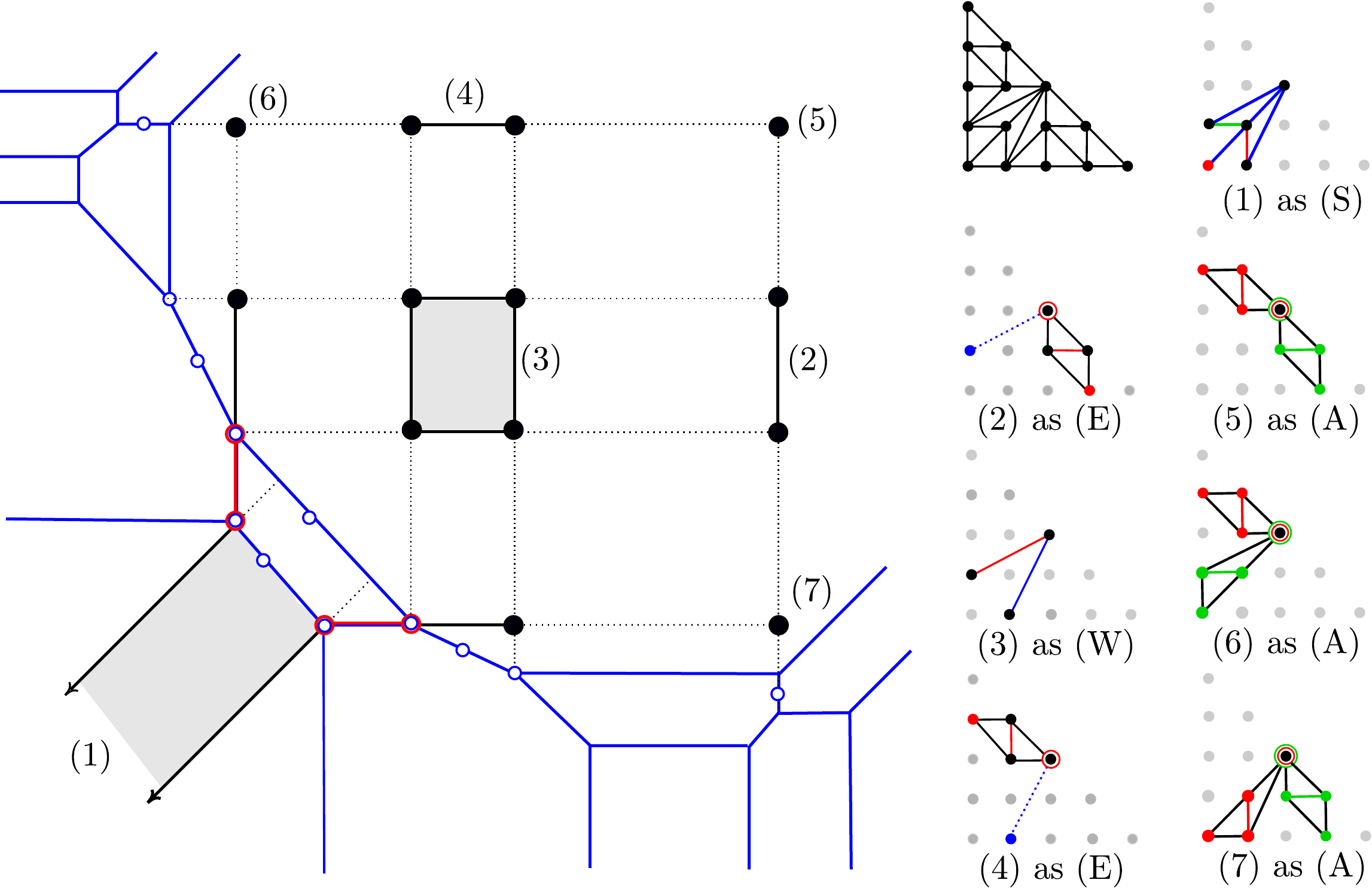}
  \caption{From left to right: a tropical quartic curve $\Gamma$ and its seven bitangent classes, together with its dual subdivision and the relevant dual cells responsible for each class. The color code matches the  $\Sn{3}$-orbit shape representative from~\autoref{fig:2dCells} (permutations are listed in~\autoref{tab:exampleLifts}). The unfilled dots on $\Gamma$ match the locations of chips from~\autoref{fig:thetaChar}.\label{fig:quartic}}
\end{figure}

\begin{example}\label{ex:212I} We consider the plane quartic curve over $\PS$ defined by 
\begin{equation}\label{eq:exampleQ}
  \begin{aligned}
    q(x,y)\!:= &  1 + 3t^{10} x - t^8y + 5t^{29}x^2 + 7t^{15} xy - 3 t^{28} y^2 + 11t^{50}x^3 - 5t^{29}x^2y -7t^{28}xy^2\\
   &   +13t^{50}y^3 +17t^{77}x^4-13t^{51}x^3y+14t^{32}x^2y^2-11t^{51}xy^3    -17 t^{74}y^4.
  \end{aligned}
\end{equation}
The tropical curve $\Gamma = \Trop\,V(q)$ and its seven tropical bitangent classes together with its dual Newton subdivision can be seen in~\autoref{fig:quartic}. The shape for classes (2), (3) and (4) is only revealed after applying a non-trivial permutation, described explicitly in~\autoref{tab:exampleLifts}. The picture shows the relevant cells in the Newton subdivision for these non-standard  bitangent classes. In this example, the shape of all seven bitangent classes is independent of the metric structure on $\Gamma$, hence it only depends on the Newton subdivision of $q$.

Next, we discuss how the signs $s_{ij}$ of the coefficients $a_{ij}$ of $q$ determine the  existence of real lifts of the bitangent classes of $\Gamma$, as predicted by~\autoref{tab:lifting} (after suitable permutations). 
By construction, all lifts from shape (3) must be real, but realness for the remaining shapes does depend on the choice of signs. \autoref{tab:exampleLifts} summarizes our findings.

\begin{table}[htb]
  \begin{tabular}{|c|c|c|c|c|}
    \hline Class & Shape & Permutation  & Sign Conditions \\
    \hline (1) & (S) & $id$ & $s_{00}\,s_{22} > 0$ \\
    \hline (2) & (E) & $\tau_{1} \circ \tau_{0}$ & $s_{21}\,s_{31}\,s_{30}\,s_{22} < 0$  \\
    \hline (3) & (W) & $\tau_0 \circ \tau_{1} \circ \tau_{0}$ & ------ \\
    \hline (4) & (E) & $\tau_0 \circ \tau_{1} \circ \tau_{0}$ & $s_{12}\,s_{13}\, s_{22}\,s_{03}<0$\\
\hline    (5) & (A) & $id$ & same as (2) and (4) \\
 \hline   (6) & (A) & $id$ & same as (1) and (4)\\
  \hline   (7) & (A) & $id$ & same as (1) and (2)\\
    \hline
  \end{tabular}
  \caption{Bitangent classes, their canonical shapes, the required permutations to bring a class into each shape,  and sign conditions for real-liftings. The maps $\tau_0$ and $\tau_1$ are given in~\autoref{tab:S3Action}.
     \label{tab:exampleLifts}}
\end{table}

The parameters featuring in all formulas in~\autoref{tab:lifting} are determined by the vertices of the Newton subdivision in the neighborhood of each tangency points. In particular, for class (5) we use the parameters $v=u=2$, $i=j=3$, for  (6) we have  $v=2$, $i=3$, $u=j=0$, whereas for (7) we take $v=i=0$, $u=2$ and $j=3$. Replacing the values for each $s_{ij}$ coming from~\eqref{eq:exampleQ} we certify that the class (1) is the only one with  real lifts. Thus, $V(q)$ has a total of eight real bitangents. Furthermore, our results place the vertex of the corresponding tropical bitangents at the black vertices of (1) and (3).

Notice that three inequalities involving eight signs in~\autoref{tab:exampleLifts} govern the realness of the algebraic lifts. Thus, we can study how the total number of real bitangents varies as we choose a different smooth plane quartic  $V(\tilde{q})$ over $\PSR$ with tropicalization  $\Gamma$.  If all inequalities hold, the quartic curve $V(\tilde{q})$ has 28 real bitangents. Violating any subset of them will lead to different numbers of classes with real lifts, namely four, two or one. Using~\autoref{thm:reallifting}, we conclude that the number of real lifts are 28, 16, 8 or 4.

\autoref{tab:ex} shows sign choices realizing each of these four cases, together with the bitangent classes admitting real lifts and the topology of a real quartic curve 
close to the tropical limit. The latter was certified with the combinatorial patchworking online tool~\cite{Viro:sage2}.
\end{example}

\begin{table}[tb]
   \begin{tabular}{|c|c|c|c|}
    \hline
    Negative signs & Real bitangent classes & Number of Real lifts & Topology \\
    \hline \hline
 ---  & (1) and (3) & 8 & 2 non-nested ovals \\\hline
 $s_{31}$  & (1), (2), (3) and (7) & 16 & 3 ovals  \\\hline
 $s_{13}$, $s_{31}$      & (1)$,\ldots,$(7) & 28 & 4 ovals \\\hline
 $s_{13}$, $s_{31}$, $s_{22}$     & (3) & 4 & 1 oval \\
 \hline  \end{tabular}
  \caption{Sample sign choices for the tropicalized quartic  $\Gamma$ in~\autoref{fig:quartic} giving all possible number of real bitangents on the algebraic lift $V(\tilde{q})$. For each row, we indicate which  bitangent classes have real lifts  and the topology of the real curve $V_{\RR}(\tilde{q})$ near the tropical limit.\label{tab:ex}}
 \end{table}

\section{Tropical curves, bitangent lines and their lifts}\label{sec:trop-plane-curv}

Throughout this paper we work with real closed, complete non-Archimedean valued fields $\KKr$ with valuation ring $(R,\mathfrak{M})$ and their algebraic closures, denoted by $\KK$. We assume the valuation map is non-trivial and admits a splitting, which we denote by $\ww \mapsto t^{\ww}$ as in~\cite[Chapter 2]{MSBook}. We set $t^{0}=1$.  Our main examples will be the Puiseux series $\PSR$ and $\PS$ in the uniformizer $t$. By the Tarski principle \cite{JL89}, $\PSR$ is equivalent to the reals, so we can count real bitangents by lifting tropical bitangents to $\PSR$. This principle has been applied to other problems in tropical geometry, see e.g.\ \cite{ABGJ, MSBook}.

Initial forms (defined below) will be central to this work:

\begin{definition}\label{def:initialForms}
  Given $a\in\KKr$,   $a\in \KK$ with $\val(a)=\alpha$ its \emph{initial form} $\bar{a}$ is defined as the class $\overline{t^{-\alpha}a}$  in the residue field $\resK := R/\mathfrak{M}$.
\end{definition}

In the sequel, a  bitangent line $\ell$ to $V(q)$ will be determined by the fixed equation
 \begin{equation}\label{eq:line}
\ell = y + m + n x ,\qquad \text{ with } m,n\in \KK^*.
 \end{equation}
Its tropicalization $\Lambda$ will always be a non-degenerate tropical line in $\RR^2$, i.e., a tripod graph with ends of direction $(-1,0)$, $(0,-1)$ and $(1,1)$.

Tropical curves will be defined following the \emph{max} convention (for more details, see e.g.~\cite{MSBook}.) The tropicalization of a plane curve $V(q)$ with $q\in \KK[x,y]$ will be determined by the Newton subdivision of $q$  (for an example, see~\autoref{fig:quartic}.) In turn, the
 tropicalization of a curve  embedded in $(\KK^*)^n$ by an ideal $I$ in the Laurent polynomial ring is  the  (Euclidean) closure in $\RR^n$ of the image of $V(I)$   under the  coordinatewise negative valuation map on $\KK^*$. Our choice of $I$ will be determined by tropical modifications or refinements of $\RR^2$ along tropical bitangent lines to tropical plane curves. For details, we refer to~\cite{Mi06}.

\subsection{Tropical bitangents} \label{sec:tropical-bitangents} Throughout this paper, we consider smooth plane quartics $V(q)$ defined over either $\KK$ or $\KKr$, where 
\begin{equation}\label{eq:q}
  q(x,y):=\sum_{\substack{0\leq i +j \leq 4\\0 \leq i,j }} a_{ij} \, x^i \, y^j.
\end{equation}
We assume the Newton polytope of $q$ is the standard $2$-simplex of side length four and we let $\Gamma$ be the associated tropical plane quartic. 

We assume $\Gamma$ is a \emph{tropical smooth plane quartic}, that is, the Newton subdivision of $q$ is a unimodular triangulation.
The skeleton $\Sigma(\Gamma)$ of $\Gamma$ is the subgraph obtained by repeatedly contracting all the edges adjacent to leaves. As we mentioned in the introduction, only four out of all five planar genus three graphs in~\autoref{fig:PossibleSkeletons} can arise as skeletons of  $\Gamma$. The edge lengths are also linearly constrained, as shown in~\cite[Theorem 5.1]{PlaneModuli2015}.

\begin{figure}[tb]
\includegraphics[scale=0.75]{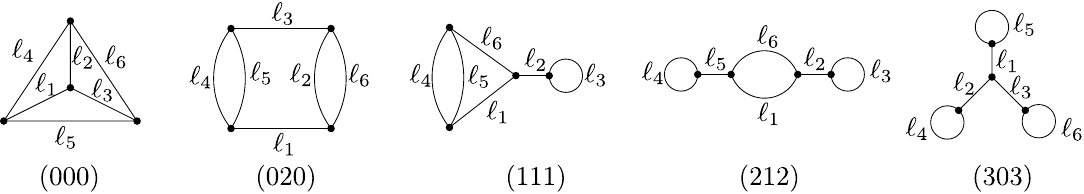}
\caption{Possible skeletons for smooth plane quartics. Following  the notation of~\cite{PlaneModuli2015}, each graph is labeled by a triple $(lbc)$ of single digit numbers, where $l$ is the number of loops, $b$ is the number of bi-edges (pairs of edges between adjacent vertices), and $c$ is the number of bridges (or cut-edges.) The graph (303) cannot be realized via tropicalization in $\RR^2$.
  \label{fig:PossibleSkeletons}}
\end{figure}

Tropical bitangent lines to $\Gamma$ were first studied by Baker \emph{et al.}~\cite{bak.len.ea:16}. They were defined using~\emph{stable intersections} between a tropical line $\Lambda$ and $\Gamma$. We recall the construction:
\begin{definition}\label{def:tropicalBitangent} A tropical line $\Lambda$ is \emph{bitangent} to $\Gamma$ if any of the following conditions holds:
  \begin{enumerate}[(i)]
  \item $\Lambda \cap \Gamma$    has two connected components, each with stable intersection multiplicity 2; or
  \item $\Lambda \cap \Gamma$ is connected and its stable intersection multiplicity is 4.
  \end{enumerate}
\end{definition}
Stable intersections in $\RR^2$ are determined by the fan displacement rule (see, e.g.~\cite{jen.yu:16}):
\[\Gamma \cap_{st} \Lambda = \lim_{\varepsilon\to 0} \Gamma \cdot (\Lambda + \varepsilon v).
\]
A generic choice of vector  $v$
ensures that $\Gamma$ and $\Lambda + \varepsilon v$ intersect properly in $\RR^2$ whenever $\varepsilon$ is small enough. If two tropical curves $C_1$ and $C_2$ intersect property and $p\in C_1\cap C_2$, the intersection multiplicity at $p$ is $(C_1\cdot C_2)_{|p} = |\det(u_1,u_2)|$ where $u_1$ and $u_2$ are the weighted directions of the edges of $C_1$ and $C_2$ containing $p$ in their relative interiors. Furthermore, $\displaystyle{C_1\cdot C_2 = \sum_{p\in C_1\cap C_2} (C_1\cdot C_2)_{|p}}$ whenever these curves intersect properly.

\begin{remark}\label{rm:actionBySn} By construction, the symmetric group $\Sn{3}$ on three letters records automorphisms of $\PP^2$  fixing the  tropical line. This action extends to  $\ZZ^2$ and the space of smooth tropical plane quartics and their bitangent lines. \autoref{tab:S3Action} shows the action of two generators of $\Sn{3}$ on both the classical and tropical worlds.
\begin{table}[htb]
  \begin{tabular}{|c|c|c|c|}
\hline    Gen. & Projective & Lattice & Tropical \\
    \hline
    $\tau_0$ & $x\longleftrightarrow y$ & $(i,j) \mapsto (j,i)$ & $X\longleftrightarrow Y$\\
        \hline $\tau_1$ & $ x\longleftrightarrow z$ & $(i,j)\mapsto (4-i-j, j)$ & $X\mapsto -X$\;\;; $Y\mapsto Y-X$\\
     \hline
      \end{tabular}
    \caption{Generators of $\Sn{3}$ as projective, polytopal and tropical isomorphisms.\label{tab:S3Action}}
\end{table}
\end{remark}

Using the theory of tropical divisors on curves (see e.g.~\cite{bak.jen:16}),  bitangent lines to $\Gamma$ with tangency points $P$ and $P'$ can be identified with tropical divisors on $\Gamma$ that are linearly equivalent to $2P+2P'$. More precisely, given the stable intersection $D:=\Gamma\cap_{st} \Lambda$, we must find  a piecewise linear function $f$ on $\Gamma$ that is linear outside $D\cup \{P, P'\}$ and such that $D + \operatorname{div}(f)$ is effective and contains $2P + 2P'$. These tropical divisors correspond to effective \emph{tropical theta characteristics}~\cite{Zh10} on the metric graph $\Sigma(\Gamma)$, as the following example illustrates.

\begin{example}\label{ex:abstract212I} 
  \autoref{fig:thetaChar} shows the skeleton $\Sigma(\Gamma)$ of the curve discussed in~\autoref{ex:212I} with its metric structure and 12 relevant points ($v_0'$ through $v_3'$) used to describe its eight  theta characteristics.  The location of the points  $v_{12}''$ and $v_{12}'''$ depends on the metric structure on $\Sigma(\Gamma)$.
  The loops $\gamma_1,\gamma_2$ and $\gamma_3$ in  $\Sigma(\Gamma)$ of lengths 17, 12 and 14, respectively, are dual to the lattice points $(1,2)$, $(1,1)$ and $(2,1)$ in the Newton subdivision.

  Using Zharkov's Algorithm~\cite{Zh10}, we write all theta characteristics of the metric graph:
\begin{equation*}\label{eq:thetaChar}
  \begin{minipage}[c]{0.3\textwidth}
\[    \begin{aligned}
    L_0  & :=  - v_0' + v_0 + v_{12} + v_3',\\
    L_{\gamma_1}& :=  v_{12} + v_3'\,,\\
\end{aligned}
\]  \end{minipage}
  \begin{minipage}[c]{0.22\textwidth}
\[    \begin{aligned}
  L_{\gamma_2}& :=  v_0' + v_3'\,,\\
    L_{\gamma_3}& :=  v_0' + v_{12}'\,,    
  \end{aligned}\]
  \end{minipage}
  \begin{minipage}[c]{0.22\textwidth}
   \[ \begin{aligned}
    L_{\gamma_{12}}  & :=  v_{01} + v_3'\,,\\
    L_{\gamma_{13}}  &  :=  v_{12}'' + v_{12}'''\,,
  \end{aligned}\]
  \end{minipage}
  \begin{minipage}[c]{0.22\textwidth}
   \[ \begin{aligned}
    L_{\gamma_{23}} & :=  v_{0}' + v_{23}\,,\\
    L_{\gamma_{123}} &:=  v_{01} + v_{23}\,.    
  \end{aligned}\]
  \end{minipage}
      \end{equation*}
where $\gamma_{I} := \sum_{i\in I}\gamma_i$ for each $I$. The seven effective theta characteristics $L_{\gamma_1},\ldots, L_{\gamma_{123}}$ correspond to the seven bitangent classes of $\Gamma$. The location of the chips on each graph of~\autoref{fig:thetaChar} indicates the pair of tangency points for some bitangent line to $\Gamma$.
  \end{example}

\begin{figure}[tb]
  \includegraphics[scale=0.15]{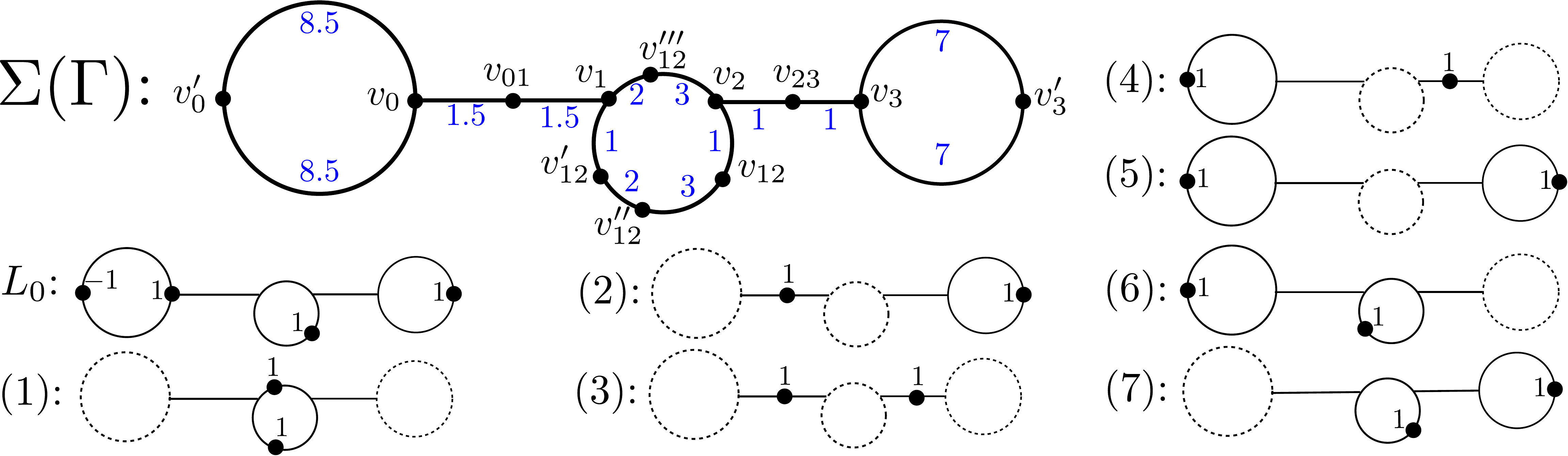}
  \caption{Metric structure on the skeleton $\Gamma$ of the tropical quartic, and its eight theta characteristics. The dashed cycles on the chip configuration of each effective theta characteristic indicate its associated non-zero element of $H_1(\Gamma,\ZZ/2\ZZ)$. Note that $L_0$ is the single non-effective theta characteristic. The labeling matches that of~\autoref{fig:quartic}.
    \label{fig:thetaChar}}
\end{figure}

In~\cite{len.mar:20}, Len and the second author provided a classification of all tangency types between the curves $\Gamma$ and $\Lambda$ into five types. Furthermore, the theta characteristics approach allowed them to determine the precise location of the tangency points within the stable intersection $\Gamma\cap_{st} \Lambda$. For non-proper intersections, the tangencies are located in the midpoint of bounded edges where overlappings occur. Overlappings along ends of $\Gamma$ reflect a tangency at the adjacent vertex of $\Gamma$.

To simplify our case analysis in~\autoref{sec:comb-class-bitang}, we provide a refined classification of local tangencies. Their $\Sn{3}$-representatives are depicted to the right of~\autoref{fig:localTangencies}. Our combinatorial classification requires precise information of the edge directions in the link of a tropical tangency. Tangencies of types (6a) and (6b) show an overlap between a horizontal end of both $\Lambda$ and $\Gamma$. The remaining two edge directions in the star of the tangency at $\Gamma$ are fixed: they are  $(1,-1)$ and $(0,1)$ for type (6a), and $(3,-1), (-2,1)$ for type (6b).  \autoref{tab:edgeOnEdgeMult2} shows all relevant directions involved in type (2) and (4) tangencies occurring at each end of  $\Lambda$. For the bitangent lines depicted in (3b), the connected component of intersection with $\Gamma$ contains both tangency points. In our local discussion here, we only refer to the tangency point in the interior of the horizontal bounded edge, as highlighted in the picture. 

\begin{table}[htb]
  \begin{tabular}{|c|c|c|}
\hline    End & Direction of $e$ & Direction of $e^\vee$ \\
    \hline
    Diagonal & {$(-1,1)$}, $(1,3)$, $(3,1)$ & $(1,1)$, $(-3,1)$, $(-1,3)$\\
        \hline Horizontal & $(1,2)$, $(-1,2)$,  $(3,2)$ &  $(-2,1)$, $(2,1)$, $(-2,3)$ \\
     \hline
    Vertical & $(2,1)$, $(2,3)$, $(2,-1)$ & $(-1,2)$, $(-3,2)$, $(1,2)$\\
    \hline
      \end{tabular}
    \caption{Data for  multiplicity two local tangencies  of types (1) and (2) for each end of $\Lambda$. Here, $e$ is the  bounded edge  of $\Gamma$ containing the tangency and $e^{\vee}$ denotes the corresponding dual edge in the Newton subdivision. For type (4) only the first direction for $e$ and $e^{\vee}$ in each column is allowed.\label{tab:edgeOnEdgeMult2}}
  \end{table}

The next lemma discusses the combinatorics of proper tangencies at vertices:

\begin{lemma}\label{lm:transverseTangencies} Let  $v$ be a vertex of both $\Gamma$ and $\Lambda$ and assume the local intersection at $v$ is transverse. There are two configurations (up to $\Sn{3}$-symmetry) for which $v$ is a tangency point. They correspond to types (5a) and (5b) in~\autoref{fig:localTangencies}. The ends of $\Star_{\Gamma}(v)$ have directions $(1,0),(0,1)$ and $(-1,-1)$ for (5a), and  $(1,0)$, $(2,-1)$, $(-3,1)$ for (5b).
\end{lemma}
\begin{proof} The min- and max-tropical lines with vertex $v$ divide $\RR^2$ into six regions, as seen in the right-most picture in~\autoref{fig:localTangencies}. We prove the statement by a direct computation, analyzing the locations of the three ends in the star $\Star_{\Gamma}(v)$ of $v$ in $\Gamma$ with respect to these six regions. First, assume $\Star_{\Gamma}(v)$ is a min-tropical line. Then, the vertex $v$ is a tangency point of local multiplicity two.

  Second, assume $\Star_{\Gamma}(v)$  contains only one of the three ends of a min-tropical line. Up to $\Sn{3}$-symmetry, we may assume it has direction $(1,0)$ and the remaining two ends are located in the relative interior of the following combined regions: (I,IV), (I,V), (I,VI) or (II,V). The latter case is not possible by the smoothness condition on $\Gamma$. In the first three cases, the same condition ensures that the bounded edge of $\Gamma$ in I equals  $(-a,1)$ with $a=2$ or $3$. Since $v$ is a tangency,  we conclude that $a=3$ and the multiplicity is four.

  Finally, we claim no tangency at $v$ can occur if the star of $\Gamma$ at $v$ shares no end with a min-tropical line. Indeed, we exploit the $\Sn{3}$-action to restrict  the location of the three ends of $\Star_{\Gamma}(v)$ to four triples: (I,I,IV), (I,II,IV), (I,III,V) or (I,III,VI). The last three cases are incompatible with the balancing condition and smoothness of $\Star_{\Gamma}(v)$. In turn, smoothness and the requirement of having intersection multiplicity two or four at $v$ reduces the possible (I,I,IV) configurations to three cases, with ends $(0,1)$, $(-1,0)$, $(1,-1)$,  $(0,1)$ $(-1,2)$ and $(1,-3)$, or $(-1,0)$, $(-2,1)$ and $(3,-1)$, violating either the transverse condition property or the restrictions on $\Star_{\gamma}(v)$. This concludes our proof.  
\end{proof}

\begin{figure}[tb]
  \includegraphics[scale=0.32]{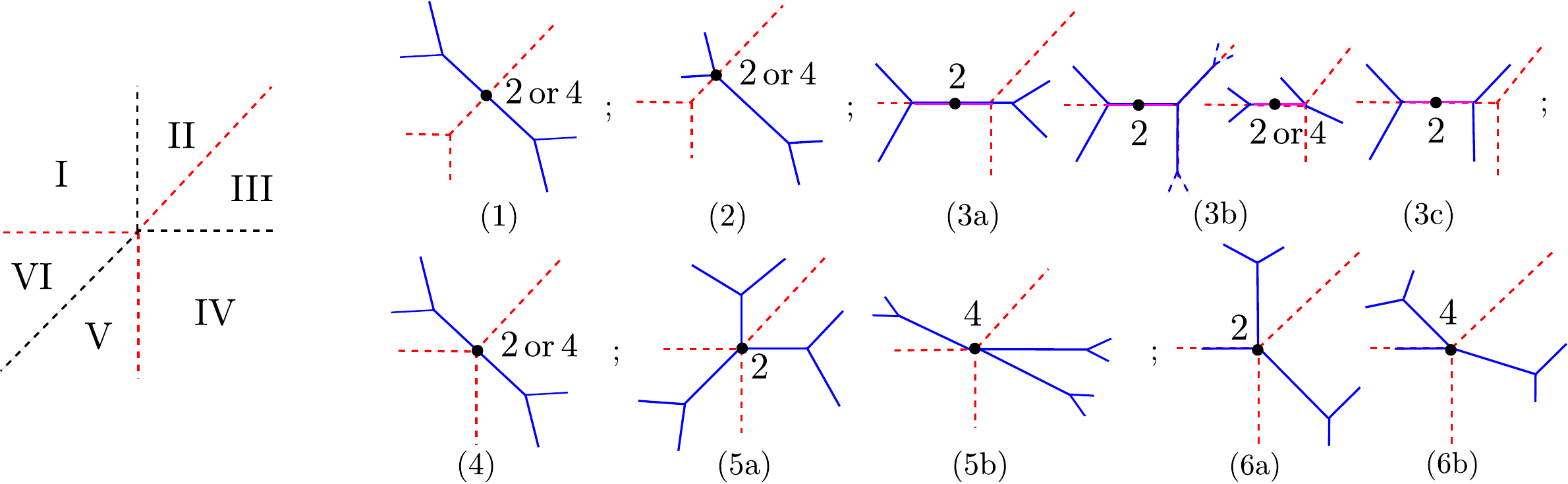}
  \caption{From left to right: six sectors of $\RR^2$ determined by max- and min- tropical lines, and six types of local tangencies at the marked point $P$. The numbers indicate the possible local intersection multiplicities for $P$. Cases (3a), (3b), (3c), (6a) and (6b) are non-transverse intersections.\label{fig:localTangencies}}
\end{figure}

\begin{remark}\label{rem:mult4}
  The action of $\Sn{3}$ allows us to restrict the local configurations corresponding to a tangency point $P$ of multiplicity 4. A simple computation of local multiplicities determines the directions of the relevant edges of $\Gamma$. For cases (1), (2), (4) and (5b), the point $P$ is adjacent to or in the relative interior of an edge $e$ with direction $(-3,1)$. For (6b), $P$ is a vertex of $\Gamma$ and its two adjacent  bounded edges have directions $(3,-1), (-2,1)$. In these five cases, we assume $P$ lies in the diagonal end of $\Lambda$.

  For (3b), we assume $P$ lies in the horizontal end of $\Lambda$. In this situation, the multiplicity four can only occur where the vertex of $\Lambda$ agrees with the right-most vertex of a bounded horizontal edge $e$ of $\Gamma$ and the stable intersection assigns weight three to $v$ and one to the other endpoint of $e$. Linear equivalence identifies this situation with a multiplicity four point on the edge. There are two possibilities for $\Star_{\Gamma}(v)$ that determine the location of the tropical tangency points via chip-firing.

  First, if  $\Star_{\Gamma}(v) = \Star_{\Lambda}(v)$, then $v$ must be adjacent to three bounded edges of $\Gamma$, with lengths  $\lambda_1\leq \lambda_2, \lambda_3$. The tangency points will be located on the two longest edges, at distance $(\lambda_3-\lambda_1)/2$ and $(\lambda_3-\lambda_2)/2$ away from the vertex $v$. 

  On the contrary, $\Star_{\Gamma}(v) = \Star_{\Lambda}(v)$ then the tangency points are located at $v$ and at the midpoint of the horizontal end of $\Gamma$ adjacent to $v$. Furthermore, the three edge directions in $\Star_{\Gamma}(v)$ come in two  possible combinations: $(-1,0)$, $(2,-1)$, and $(-1,1)$, or $(-1,0)$, $(3,1)$ and $(-2,-1)$. They are related by the map $\tau_0\circ\tau_1\circ \tau_0$ from~\autoref{tab:S3Action}.  
\end{remark}

\subsection{Lifting tropical bitangents}\label{sec:lift-bitang} In~\cite{len.mar:20}, Len and the second author developed a novel effective technique for producing bitangent lines to generic smooth plane quartics from their tropical counterparts.  
In this section, we review their construction, setting up the notation and definitions for 
{Sections}~\ref{sec:lift-trop-bitang} through~\ref{sec:trop-totally-real}, and~\autoref{sec:appendix1}.

\begin{definition}\label{def:bitangentTriple}
  Consider a degree four bivariant polynomial $q$ over $\KK$ and let $V(q)$ be the associated plane quartic curve. We say $(\ell,p,p')$ is a \emph{bitangent triple} to $V(q)$ defined over $\KK$ if $\ell$ is a bitangent line to $V(q)$ with tangency points $p$ and $p'$.
\end{definition}

 We are interested in analyzing which tropical bitangents $\Lambda$ to $\Gamma$ with tropical tangency points $P$ and $P'$ arise as tropicalizations of bitangent triples, and, most importantly, how many such liftings exist.
 
\begin{definition} \label{def:liftingBitangents}
  We say $\Lambda$ lifts over $\KK$  if there exists a bitangent triple $(\ell,p,p')$ defined over $\KK$ (also called a $\KK$-bitangent triple) whose tropicalization is $(\Lambda,P,P')$, i.e.,
  \[ \Trop(\ell) = \Lambda, \quad \Trop(p)=P\quad \text{ and } \Trop(p') =P'.
  \]The \emph{lifting multiplicity} of $\Lambda$ equals the number of such bitangent triples.
  \end{definition}

The results in~\cite[Section 3]{len.mar:20} highlight a central feature of studying bitangent lines to \emph{generic} smooth plane curves and their lifting multiplicities through the tropical lens. Indeed, they show that liftings can be determined by local systems of equations relative to each tropical tangency. Furthermore, in the quartic case, each tropical tangency provides independent and complementary information regarding the coefficients of $\ell$ in~\eqref{eq:line}. A tangency along the horizontal end of $\Lambda$ will determine $m$. The maps $\tau_0$ and $\tau_1$ from~\autoref{tab:S3Action} replace $m$ with $m/n$ and $n$ when the point lies in the vertical and  diagonal end of $\Lambda$, respectively. The situation is more delicate when the vertex of $\Lambda$ is the tangency point, as it occurs  for cases (4) and (5a) through (6b). We postpone their discussion to~\autoref{sec:lift-trop-bitang} (for multiplicity two tangencies) and~\autoref{sec:appendix1} (for multiplicity four.)

In what follows, we review the local lifting methods from~\cite{len.mar:20} when the local tangency point $P$ has multiplicity two and occurs in the relative interior of the horizontal end of $\Lambda$. After a translation and rescaling of $\Gamma$ and $q$, we may assume that  $q\in R[x,y]\smallsetminus \mathfrak{M}R[x,y]$ and $P=(0,0)$, so the vertex of $\Lambda$ equals $(\lambda,0)$ for some $\lambda>0$. In particular, $\val(m)=0$, $\val(n)=-\lambda$ and  $p\in R^2$. We write $\overline{p}=(\bar{x},\bar{y})\in (\resK^*)^2$ for the point recording the coordinatewise initial forms of $p$.

The \emph{local equations} for $\ell$ at $P$ give  a system of three equations in three unknowns $(\overline{m},\bar{x},\bar{y})$ over $\resK$ determined by the vanishing of the initial forms (with respect to $P$) of $q,\ell$ and the Wronskian $W := \det(J(q,\ell; x,y))$, that is,
\begin{equation}\label{eq:localEqns}
\overline{q}=\overline{\ell}=\overline{W}=0, \quad \text{ for }  \overline{q}:=\!\!\!\sum_{(i,j)\in P^\vee}\!\!\! \overline{a_{ij}}\, \bar{x}^i\bar{y}^j, \;\overline{\ell} := \bar{y} + \overline{m}  \text{ and } \overline{W} := \det(J(\overline{q},\overline{\ell};\bar{x},\bar{y})).
\end{equation}
Here, $P^{\vee}$ denotes the cell in the Newton subdivision of $q$ dual to $P$.
Notice that since $P=(0,0)$ and $\val(n)<0$, these initial forms are nothing but the image of the three original polynomials in $\resK[\bar{x},\bar{y},\overline{m}]$. Furthermore, the class of the Wronskian agrees with the Wronskian of $\overline{q}$ and $\overline{\ell}$.
If the system above does not have a unique solution, \emph{tropical modifications} can be used to solve for the initial forms and higher order terms (see \cite[Lemma 3.7]{len.mar:20} and references therein.) The support of the local equations has to be increased accordingly. 

Each solution $(\overline{m},\bar{x},\bar{y})\in (\resK^*)^3$ of the system from~\eqref{eq:localEqns} produces a unique pair $(m,p)$ in $R^3$ with the prescribed initial forms.  Applying a similar method to the second tropical tangency point and combining the outputs will  determined the bitangent triples lifting $\Lambda$. 
Uniqueness is essential to compute the lifting multiplicities of each $(\Lambda,P,P')$ and it follows from the well-known multivariate analog of Hensel's Lemma (see~\cite[Theorem 2.2]{len.mar:20} and references therein.)

\begin{lemma}[\textbf{Multivariate Hensel's Lemma}]\label{lm:multivariateHensel} Consider $\mathbf{f}:= (f_1,\ldots, f_n)$ with $f_i\in R[x_1,\ldots, x_n]\smallsetminus \mathfrak{M}R[x_1,\ldots, x_n]$ and let $J_{\mathbf{f}}$ be its Jacobian matrix. Fix $a=(a_1,\ldots, a_n)\in \resK^n$ satisfying $\overline{J_{\mathbf{f}}}(a)\neq 0 \in \resK$ and $\overline{f_i}(a)=0$ for all $i=1,\ldots,n$. Then, there exists a unique $b = (b_1,\ldots, b_n)\in R^n$ with $f_i(b) = 0$, $\val(b_i)=0$ and $\overline{b_i} = a_i$ for all $i=1,\ldots, n$.
\end{lemma}
 \noindent  In our case of interest, $a = (\overline{m}, \bar{x},\bar{y})$ and $b = (m,p)$. Furthermore, we can often certify the non-vanishing of the initial form of the Jacobian by computing   $J(\overline{q}, \overline{\ell},\overline{W})$ instead.

\begin{remark}[\textbf{Genericity constraints}]\label{rem:genericity} As we mentioned earlier, the tropical techniques for computing bitangents require some mild genericity conditions. In addition to the smoothness of $\Gamma$ and the non-degeneracy of a tropical bitangent lines to it,  we further assume the following conditions:
  \begin{enumerate}[(i)]
  \item if $\Gamma$ contains a vertex $v$ adjacent to three bounded edges with directions $(-1,0)$, $(0,-1)$ and $(1,1)$, then the shortest lengths of these three edges is unique;
\item $V(q)$ has no hyperflexes;
\item the coefficients of $q$ are generic enough to guarantee that if the tangencies occur in the relative interior of the same end of $\Lambda$, then 
  the local systems defined by these two points  are inconsistent.
\end{enumerate}
Condition (i) allows us to determine the position of the tropical tangencies for the line $\Lambda$ with vertex $v$, thus ensuring the validity of the lifting methods for (5b) tangencies developed in~\cite[Proposition 3.12]{len.mar:20}.  Overall, these  mild constraints determine which bitangent lines to $\Gamma$ lift to bitangent triples over $\KK$ among each bitangent class. 
\end{remark}

As a consequence of the genericity conditions imposed on $\Gamma$ and $q$ we have:

\begin{theorem}\label{thm:liftingProducts} {\cite[Theorem 3.1]{len.mar:20}} Assume $\Lambda$ is a tropical bitangent to a generic tropical smooth plane quartic $\Gamma$, with two distinct tangencies. Suppose $\Gamma\cap \Lambda$ is disconnected. Then:
  \begin{enumerate}[(i)]
    \item
    If these points lie in the relative interior of two distinct ends of $\Lambda$, the lifting multiplicity of $\Lambda$ is the
    products of the local lifting multiplicities of the two tangencies.
  \item If the tangency points lie in the same end of $\Lambda$, the lifting multiplicity of $\Lambda$ is zero.
  \end{enumerate}
\end{theorem}
In~\autoref{sec:appendix1}, we discuss lifting multiplicities in the presence of a multiplicity four tropical tangency.  In particular,~\autoref{thm:liftingMult4} combined with the formulas in~\autoref{tab:liftingMult4} allows us to conclude that multiplicity four tangencies of types (5b) and (6b) both lift with multiplicity one as seen in~\autoref{fig:2dCells}. Since we assume $V(q)$ has no hyperflexes, none of the other potential multiplicity four tropical tangencies will lift to a bitangent triple.

\autoref{tab:liftingMultiplicities} provides a summary for each relevant type, combining the statements included in~\autoref{sec:appendix1} with various results from~\cite{len.mar:20}, namely,  Propositions 3.5  (for type (1)), 3.6 (for type (2)), 3.7 (for types (3a) and (3c)), 3.10 (for type (4)), 3.11 (for type (5a)) and Remark 3.8 (for type (3b) and (6a).)

Recall that in our local discussion here, for case (3b) we only refer to the tangency point highlighted in the picture in~\autoref{fig:localTangencies}. \autoref{pr:mult4type3b} shows that, globally, the bitangent line depicted in (3b) on the right does not lift if its multiplicity is four. This happens precisely when the second tangency point is the vertex of the tropical bitangent line. \autoref{pr:shapeC} studies the global lifting behavior of a bitangent line with two type (3b) tangencies, as seen in the left of the figure. In this situation,  the two tangency points lie in the relative interior of two edges of the quartic curve.

The multiplicity formula for types (4) and (6a) involves  two vectors. First, the edge $e$ of $\Gamma$ carrying the tangency, and, second,  the end  $e'$ of $\Lambda$ where the remaining tangency occurs. Possible combinations of $(e,e')$ will be determined in~\autoref{sec:comb-class-bitang} (see~\autoref{rm:NPColorCoding}). In~\autoref{sec:lift-trop-bitang} we refine these techniques to address real liftings of tropical bitangents with two tangencies.

\begin{table}[tb]
  \begin{tabular}{|c||c|c|c|c|c|c||c|c|c|}
    \hline    type
    & (1) & (2) & (3a),  (3b) or (3c) &  (4) & (5a) & (6a) & (3b) & (5b) & (6b) \\
    \hline
    mult. & 0 & 1 & 2 &  $|\det(e,e')|$ & 2  & $|\det(e,e')|$ & 4 & 1 & 1\\
      \hline
      \end{tabular}
    \caption{Local lifting multiplicities for each tangency type, assuming multiplicity two for all types except for the three right-most columns.\label{tab:liftingMultiplicities}}
  \end{table}

\section{Bitangent classes, shapes and their local properties}\label{sec:bitang-class-shap}

As was mentioned in~\autoref{sec:introduction} we are interested in determining all tropical bitangent lines $\Lambda$ to generic tropical smooth plane quartics, denoted throughout by $\Gamma$. In particular, we assume that such lines are non-degenerate in $\TPr^2$, i.e., they have a vertex $v$ and three ends. Thus, we may identify $\Lambda$ with the location of its vertex $v$ in $\RR^2$. Results in this section are purely combinatorial: they  depend solely on the duality between $\Gamma$ and unimodular triangulations of the standard 2-simplex of side length four. The metric structure on the skeleton of $\Gamma$ plays no role, and $\Gamma$ need not be generic in the sense of~\autoref{rem:genericity}.

By~\cite[Proposition 3.6, Definition 3.8]{bak.len.ea:16}, linear equivalence of effective tropical theta characteristics correspond to  continuous translations of tropical bitangent lines that preserve the bitangency property. This leads us to  the following definition:

\begin{definition}\label{def:bitangent class}
  Given a tropical  line $\Lambda$ bitangent to a tropical smooth plane quartic $\Gamma$, we define its \emph{tropical bitangent class} as the connected components of the subset of $\RR^2$ containing the vertices of all tropical bitangent lines linearly equivalent to $\Lambda$. The \emph{shape of a tropical bitangent class} refines each class by coloring those points belonging to the tropical quartic $\Gamma$, and subdividing edges and rays of a class, accordingly.
\end{definition}

\noindent By~\cite[Theorem 3.9]{bak.len.ea:16} we know that each $\Gamma$ admits seven bitangent classes. Shapes refine the combinatorial structure of bitangent classes  using the subdivision of $\RR^2$ induced by $\Gamma$.

By~\autoref{rm:actionBySn}, the symmetric group $\Sn{3}$ acts on bitangent classes and their shapes. We aim for a complete classification of all bitangent shapes, up to $\Sn{3}$-symmetry. To achieve this, we must first discuss how to perturb a vertex while remaining in the same bitangent class. It suffices to focus on one tangency point at a time. A description of such local moves is the content of the next lemma:

\begin{lemma}\label{lm:localMoves} Let $\Lambda$ be a tropical bitangent to $\Gamma$ and let $P$ be a tangency point. The relative position of $P$ within $\Lambda$ and $\Gamma$ restricts the directions in which to move the vertex $v$ of $\Lambda$ so that the corresponding translation of $P$ remains a tangent point. They are depicted in~\autoref{fig:localMoves}.
\end{lemma}
\begin{proof} We proceed by a case-by-case analysis, depending on the nature of $\Gamma\cap \Lambda$ locally around $P$. Up to $\Sn{3}$-symmetry, there are six cases to consider, as seen in~\autoref{fig:localTangencies}. The corresponding local moves for each case are depicted in~\autoref{fig:localMoves}. It is important to remark that we are only concerned with ensuring the translation of $P$ remains a tangency point as we move the pair $(v,P)$, even though the second tangency point may cease to be so along the way.  We distinguish between transverse and non-transverse intersections.

  First, assume that the local intersection is transverse and $P$ is in the relative interior of an edge  $e$ of $\Gamma$ and an end of $\Lambda$, as in the picture labeled (1) in~\autoref{fig:localMoves}. Then, we can pick $\varepsilon >0$ so that the tropical line  with vertex $v+w$ is tangent to $\Gamma$ at a point in $e$ for any $w$ in the Minkowski sum $D + \RR_{\leq 0} (1,1)$, where $D :=D(0,\varepsilon)$ is an open disc. By construction, this local move is 2-dimensional and unbounded.

  Second, assume that the intersection around $P$ is again transverse but $P$ satisfies one of the following conditions, corresponding to pictures (2) and (4) in  the figure:
  \begin{enumerate}[(i)]
    \item $P$ is a vertex of $\Gamma$ and lies in the relative interior of an end $\rho$ of $\Lambda$; or
    \item $P$ is the vertex of $\Lambda$ and lies in the relative interior of a bounded edge $e$ of $\Gamma$.
  \end{enumerate}
  In both situations, we can find $\varepsilon >0$ and an  (open) half disc $D$ centered at $\mathbf{0}$ of radius $\varepsilon$, so that the line $\Lambda$ with vertex $v+w$ is tangent to $\Gamma$ at a point in either $e$ or $\rho$ for all $w \in D +\RR_{\leq 0}(1,1)$. The set $D$ is obtained by intersecting $D(v,\varepsilon)$ with a half-space determined by either $\rho$ or $e$. This local move is also 2-dimensional and unbounded.

  The remaining option for a transverse intersection locally around $P$ corresponds to case (5), where $P$ is also the vertex of $\Lambda$. If the local multiplicity is two, the configuration is the one in (5a). In this case, we can translate this point along the three bounded edges of $\Gamma$ adjacent to $P$. The new tangency point will correspond to a type (3) local tangency, as in \autoref{fig:localTangencies}. The local move is 1-dimensional and bounded.

  On the contrary, if the local multiplicity is four, then $P$ is adjacent to three bounded edges of $\Gamma$, with directions $(1,0)$, $(-3,1)$ and $(2,-1)$, as seen in (5b). In this situation, we have two possible local moves: one bounded 1-dimensional move along the horizontal edge of $\Gamma$, and a second one corresponding to the Minkowksi sum of $\RR_{\leq 0}(1,1)$ and a circular sector $D$ bounded by the line with slope one through $v$ and the slope $-1/3$ edge of $\Gamma$. This second move is 2-dimensional and unbounded.

  It remains to address all non-transverse intersections. First, assume we have a type (3) local tangency along a horizontal end of $\Lambda$. In this situation, we can move $(v,P)$ along both directions $(1,0)$ and $(-1,0)$, as seen in the pictures (3a), (3b) and (3c). The movement is 1-dimensional and locally unbounded.
  
  Finally, assume that $P$ is a vertex of both $\Gamma$ and $\Lambda$, and that locally around $P$, both curves intersect along a common horizontal end,  as in (6a) and (6b) in the figure.   The exact values of the remaining directions of the star of $\Gamma$ at $P$ depend on the intersection multiplicity of $P$.

  In both cases, we can find $\varepsilon>0$ and a bounded open circular sector $D$ of $D(0,\varepsilon)$ so that the line with vertex $v+w$ is tangent to $\Gamma$ along a bounded edge of $\Gamma$ adjacent to $P$ for each $w\in D + \RR_{\leq 0} (1,1)$. The sector $D$ is bounded between the vector $(-1,-1)$ and the vectors $(1,-1)$ for (6a), respectively, $(3,-1)$ for (6b), corresponding to the direction of the relevant edge of $\Gamma$ adjacent to $P$.

  In addition, (6a) allows for an extra move: we can translate $(v,P)$ along a bounded vertical segment with direction $(0,1)$. As it occured with type (5b), this  local movement is not pure-dimensional. This concludes our proof.
\end{proof}

\begin{figure}[tb]
  \includegraphics[scale=0.27]{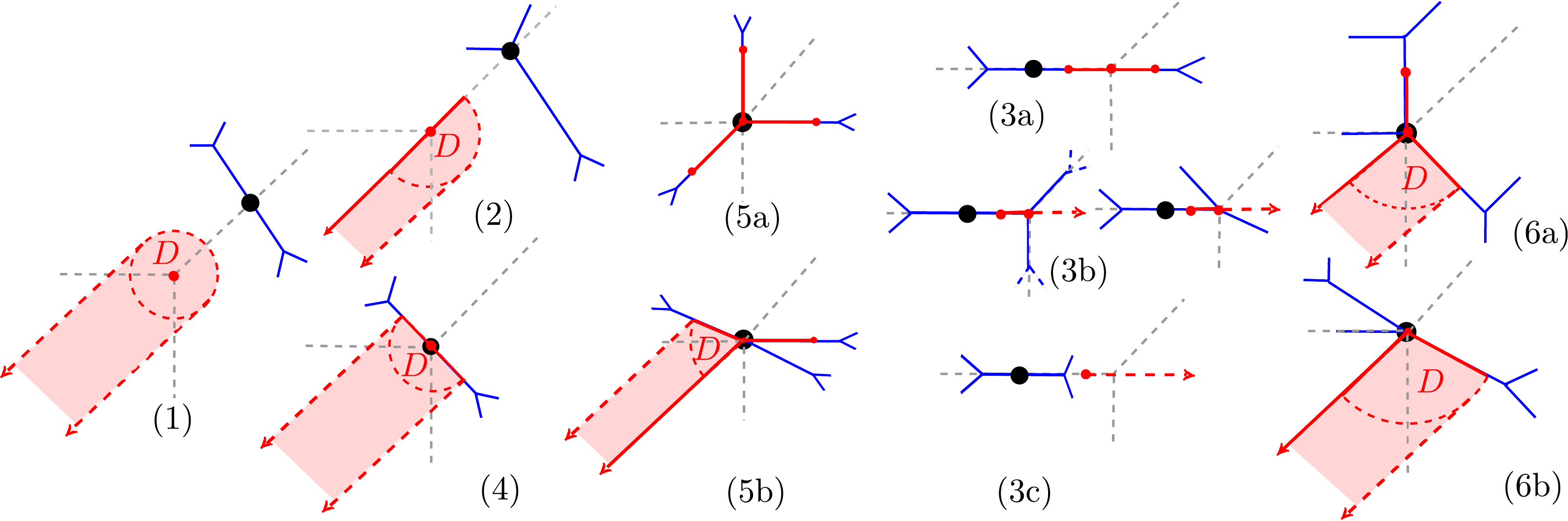}
  \caption{Local moves on bitangents fixing one combinatorial type of tangency. The tangency point $P$ is indicated by a circled black dot.\label{fig:localMoves}}
\end{figure}

Local moves that preserve the two tangencies are obtained by   combining the local moves determined by each tangency point separately. \autoref{lm:localMoves} has  an important topological consequence. Each point of a bitangent class admits a local dimension, corresponding to the dimension of the local moves. Notice that whenever a local move is bounded, its boundary is determined by a line segment. Thus:

\begin{corollary}\label{cor:polyhedralStructure} 
  Bitangent classes to $\Gamma$ are connected polyhedral complexes.
\end{corollary}

In order to combine local moves associated to tangency points lying in the same end of $\Lambda$, it will be useful to compare the position of vertices of $\Gamma$ relative to this end. The following definition arises naturally:

\begin{definition}\label{def:relativeOrder} Consider two vertices $v,v'$ of $\Gamma$. We say that $v$ is \emph{smaller} than $v'$ \emph{relative to  a weight} $\ww\in \RR^2$ if $v\cdot \ww < v'\cdot \ww$.  In particular, if $\ww=(1,-1)$ we say $v$ is smaller than $v'$ \emph{relative to the diagonal end of $\Lambda$} and write $v\prec_d v'$. If $v\cdot (1,-1)=v'\cdot (1,-1)$, then $v$ and $v'$ are aligned along the diagonal. In this case, we write $v=_d v'$.
\end{definition}

We end this section by discussing unboundedness of  cells on 
bitangent shapes. To this end, we define:
\begin{definition}\label{def:unboundedDirections}
  Let $\sigma$ be an unbounded cell of a bitangent shape and pick $\rho \in \RR^2\smallsetminus \{\mathbf{0}\}$. If $\sigma + \RR_{\geq 0}\rho \subset \sigma$ we say $\rho$ is an \emph{unbounded direction} for $\sigma$.
\end{definition}

The next lemma provides a sufficient condition for a cell to be unbounded.

\begin{lemma}\label{lm:Bunbounded} Let $\Lambda$ be a bitangent line to $\Gamma$ with vertex $v$. Assume that the two tangency points are contained in the interior of the same end  of $\Lambda$ with direction $\rho$. Then, the set $v+\RR_{\geq 0}(-\rho)$ is contained in the bitangent class of $\Lambda$. In particular, this class is unbounded.
\end{lemma}
\begin{proof}  We let $P$ and $P'$ be the tropical tangency points of $\Gamma$ and $\Lambda$. Without loss of generality, assume $\rho = (-1,0)$. Then, the connected component of $\mathbb{R}^2\setminus \Lambda$ bounded by the diagonal and vertical ends   does not intersect $\Gamma$. Thus, we can move the vertex of $\Lambda$ horizontally to the right arbitrarily far and each new tropical line is tangent to $\Gamma$ at $P$ and $P'$. All these new tropical lines have the same bitangent class as $\Lambda$. The class is unbounded in the direction $(1,0)$.
\end{proof}

Our next lemma identifies the unbounded directions for each bitangent shape  with the ends of a min-tropical line. We use it in~\autoref{sec:comb-class-bitang} to classify shapes with unbounded cells.

\begin{lemma}\label{lm:unboundedCellsAndDirs} Assume that a bitangent shape has an unbounded component $\sigma$. Then,  we conclude that the relative interior of $\sigma$  lies in one of three unbounded components of $\RR^2\smallsetminus \Gamma$, namely those dual to $(0,0)$, $(4,0)$ or $(0,4)$. Furthermore, each $\sigma$ is unbounded in a single direction: it is $(-1,-1)$ for $(0,0)$, $(1,0)$ for $(4,0)$ and $(0,1)$ for $(0,4)$.
\end{lemma}

\begin{proof} By construction, $\sigma^{\circ}$ intersect at most a single connected component of  $\RR^2\smallsetminus \Gamma$. First, assume $\sigma^{\circ}\subset \Gamma$, so $\sigma^{\circ}$ lies in an end of $\Gamma$. A simple inspection show that for any $v\in \sigma^{\circ}$ with $|v|\gg 0$, the tropical line with vertex $v$ will have a multiplicity one intersection point with another end of $\Gamma$, which cannot happen since $v\in \sigma$.

  Similarly, if $\sigma^{\circ}\not\subset \Gamma$, we let $(i,j)$ be the vertex dual to the unbounded two-dimensional component of $\RR^2\smallsetminus \Gamma$ meeting $\sigma$ and pick $v\in \sigma^{\circ}$ with $|v|\gg 0$. 
  If $(i,j) \neq  (0,0), (4,0)$ or $(0,4)$, then the tropical line with vertex $v$ will have a multiplicity one intersection point with an end of $\Gamma$ in the boundary of this connected component. This leads to a contradiction.

  The second claim in the statement is a consequence of the previous claim. By $\Sn{3}$-symmetry we need only analyze one case, say when $(i,j)=(0,0)$. In this situation, for all $v$ in $\sigma^{\circ}$, the tangency points between $\Gamma$ and the line $\Lambda$ with vertex $v$ must occur along the diagonal end of $\Lambda$. \autoref{lm:Bunbounded} then implies that $(-1,-1)$ is the unique unbounded direction for $\sigma$. For any other direction $\rho'$, a line with vertex in $v + \RR_{\geq 0} \rho'$ will meet an end of $\Gamma$ at a point of multiplicity one.
  \end{proof}

\section{A combinatorial classification of bitangent classes}\label{sec:comb-class-bitang}

Our objective in this section is to  classify the bitangent classes of $\Gamma$  and their shapes. The results are purely combinatorial and rely heavily on those obtained in~\autoref{sec:bitang-class-shap}.
The classification is organized by the minimal number of connected components of $\Gamma\cap \Lambda$ and the properness of this intersection where $\Lambda$ is a given member of this class (see~\autoref{tab:classificationShapes}). 
{Propositions}~\ref{prop:mult4intersection}, \ref{prop:oneComponentIntersections} and~\ref{prop:shapesfortwoconnectedcomp} determine the shapes corresponding to each combination.
Our findings are summarized in~\autoref{fig:2dCells}. \autoref{fig:NP} contains the relevant information on the dual subdivision to $\Gamma$ responsible for each bitangent shape. To simplify the exposition, the information for each subdivision can be found in the proofs of the propositions and lemmas classifying the corresponding shapes.

\begin{figure}[htb]
\includegraphics[scale=0.3]{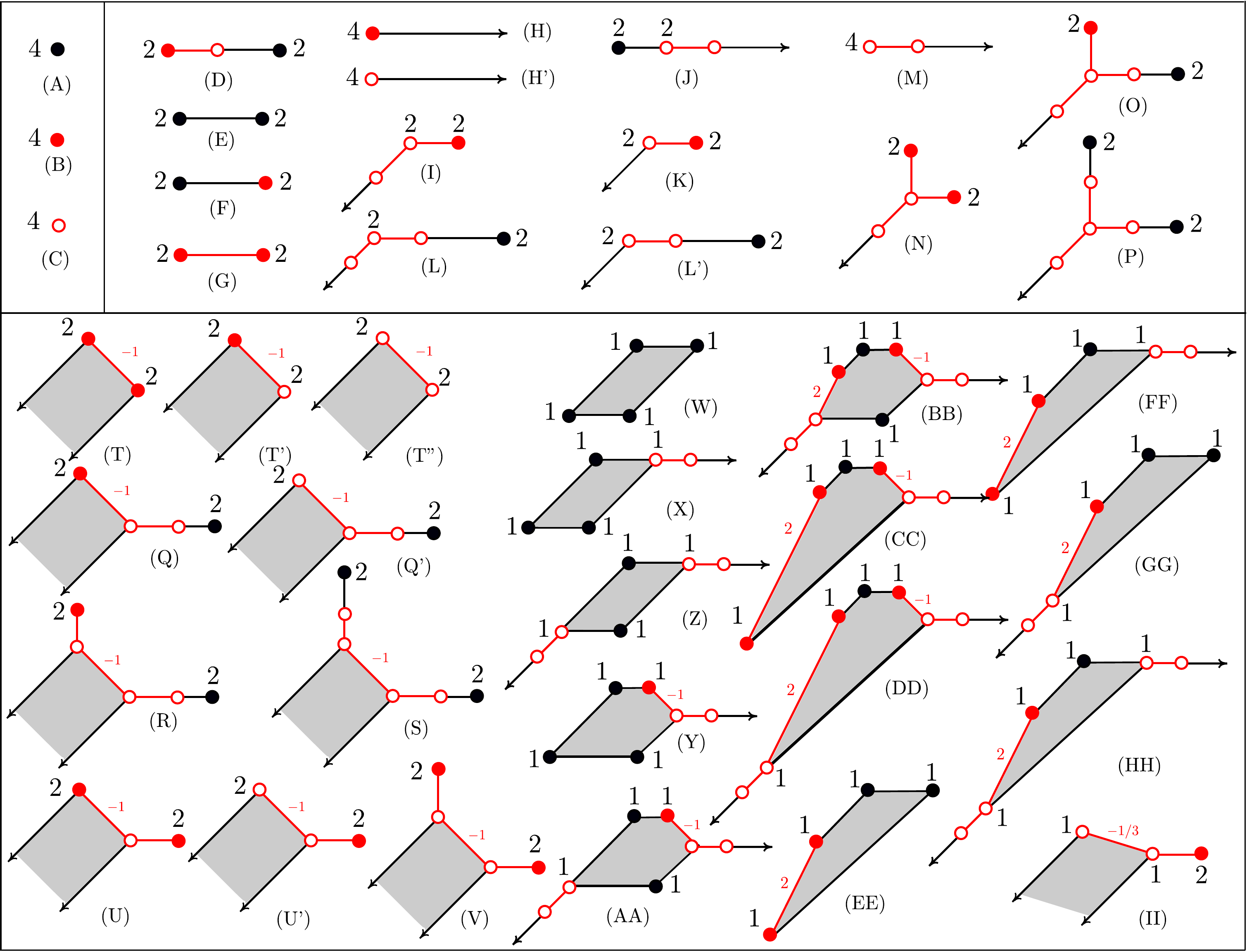}
\caption{Orbit representatives of all 41 shapes of bitangent classes to $\Gamma$, grouped by the dimension of its maximal cells. The numbers above each vertex indicate lifting multiplicities over the complex numbers (discussed in~\autoref{sec:lift-trop-bitang}), whereas the red ones above edges indicate slopes. The black cells of each bitangent class miss $\Gamma$, whereas the red ones lie on it. The unfilled dots are vertices of $\Gamma$.\label{fig:2dCells}}
\end{figure}

\begin{table}[htb]
  \begin{tabular}{|c|c|c|c|}
\hline     min. conn. comp.  &  proper & shapes \\
    \hline
    1 & yes & (II) \\\hline
    1 & no  &  (C), (D),(L),(L'),(O),(P),(Q),(Q'),(R),(S) \\
    \hline 2  & yes/no  &  rest \\
    \hline
      \end{tabular}
    \caption{Classification of the 41 bitangent shapes by the minimal number of components for $\Gamma \cap \Lambda$ and the type of intersection (proper or no) for some $\Lambda$ in the given bitangent class. The shape labels refer to those in~\autoref{fig:2dCells}.\label{tab:classificationShapes}}
  \end{table}

We start by discussing the shapes appearing in the first two rows of~\autoref{tab:classificationShapes}:

\begin{figure}
  \includegraphics[scale=0.28]{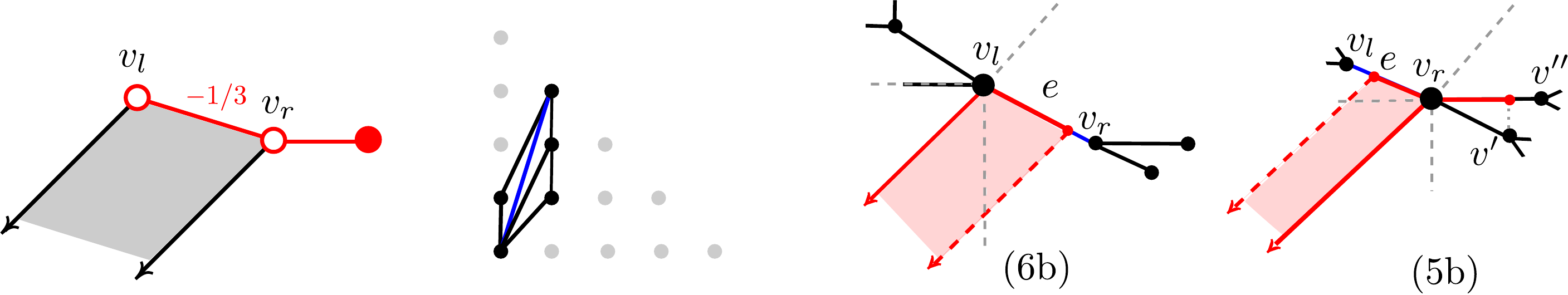}
  \caption{From left to right: shape (II), cells $v_l^{\vee}$, $v_r^{\vee}$ and $(v')^{\vee}$ in the dual subdivision to $\Gamma$ and local moves determining this bitangent shape. The edges $e$ and $e^{\vee}$ are marked in blue.\label{fig:ShapeII}}
\end{figure}

\begin{proposition}\label{prop:mult4intersection}  Let $B$ be a bitangent class of $\Gamma$ associated to a tropical line $\Lambda$ where $\Lambda\cap \Gamma$ has one connected component and the intersection is transverse. Then, $B$ lies in the $\Sn{3}$-orbit of shape (II). 
\end{proposition}

\begin{proof} Since the intersection is transverse, there is a unique tangency point of multiplicity four. By~\autoref{rem:mult4} combined with the $\Sn{3}$-symmetry, such tangency can only arise if $\Gamma$ has a bounded edge $e$ with direction $(3,-1)$.  We let $v_l$ and $v_r$ be the left and right endpoints of $e$ as in~\autoref{fig:ShapeII} . The smoothness and the degree of $\Gamma$ determine its dual subdivision locally around $e^{\vee}$: its  endpoints are $(0,0)$ and $(1,3)$. In turn, this fixes three triangles in the subdivision, as seen in the second picture from the left in the figure, those dual to the vertices $v_l$, $v_{r}$ and a vertex $v'$ adjacent to $v_r$.
  This subdivision completely determines $B$ by analyzing the local moves seen in the same figure, as we now explain.
  
  We start by placing the vertex $v$ of $\Lambda$ at $v_l$. This gives a multiplicity four tangency of type (6b). By~\autoref{lm:localMoves}, we can move $v$  along $e$ until we reach $v_r$, as we see in the third picture in~\autoref{fig:ShapeII}. The vertices of all these lines will be multiplicity four tangency points. The closure in $\RR^2$ of the local moves along the relative interior of $e$ described by the lemma yield the unbounded set $e + \RR_{\leq 0}(1,1)$. 

  The vertex $v_r$ corresponds to a vertex-on-vertex transverse local tangency of multiplicity 4, i.e., of type (5b). The possible local moves away from $v_r$ can be seen in the figure. Notice that the one-dimensional local move along the horizontal edge connecting $v_r$ with a vertex $v''$ of $\Gamma$ is bounded by the location of the vertex $v'$ connected to $v_r$ by the edge of $\Gamma$ with direction $(2,-1)$. Indeed, the vertex $v''$ appears to the right of $v'$ by the information we have already gathered on the dual subdivision. Since there are no further moves to make, we conclude that $B$ corresponds to shape (II).
\end{proof}

\begin{proposition}\label{prop:oneComponentIntersections} Let $B$ be a bitangent class of $\Gamma$ associated to a tropical line $\Lambda$ where $\Lambda\cap \Gamma$ has one connected component which is  non-transverse.  Then, $B$ lies in the $\Sn{3}$-orbit of a shape  labeled by (C), (D), (L), (L'), (O), (P), (Q), (Q'), (R) or (S) in~\autoref{fig:2dCells}.
\end{proposition}

\begin{proof} By construction, the vertex $v$ of $\Lambda$ is also a vertex of $\Gamma$ and the intersection $\Lambda\cap \Gamma$ is bounded (we have a type (3b) tangency). If $\Lambda \cap \Gamma$ consists of three edges, then  $\Star_{\Gamma}(v) = \Star_{\Lambda}(v)$, as see in~\autoref{fig:localTangencies}~
  {(3b)}. The line cannot be moved while preserving the bitangency condition, so  $B= \{v\}$ and its shape is (C).  

    Otherwise, $\Lambda \cap \Gamma$ is a bounded edge $e$ of $\Gamma$,  and the stable intersection equals the two endpoints of $e$. By exploiting the $\Sn{3}$-symmetry, we may assume $e$ is horizontal and its leftmost vertex $v_0$ has local multiplicity one. Furthermore, by~\autoref{rem:mult4} we may further assume that the triangle  dual to $v$ has vertices $(1,0)$, $(1,1)$ and $(2,2)$. In turn, the edges $e'$ and $e''$ of $\Gamma$ adjacent to $v$ must have directions $(2,-1)$ and $(-1,1)$. We let  $v'$ and $v''$ be their second endpoints, respectively.  By construction, the dual triangle to $v_0$ has vertices $(1,0), (1,1)$ and $(0,i)$ for $i=0,\ldots,4$. This information can be seen in~\autoref{fig:ShapeDLetc}.

\begin{figure}[htb]
  \includegraphics[scale=0.25]{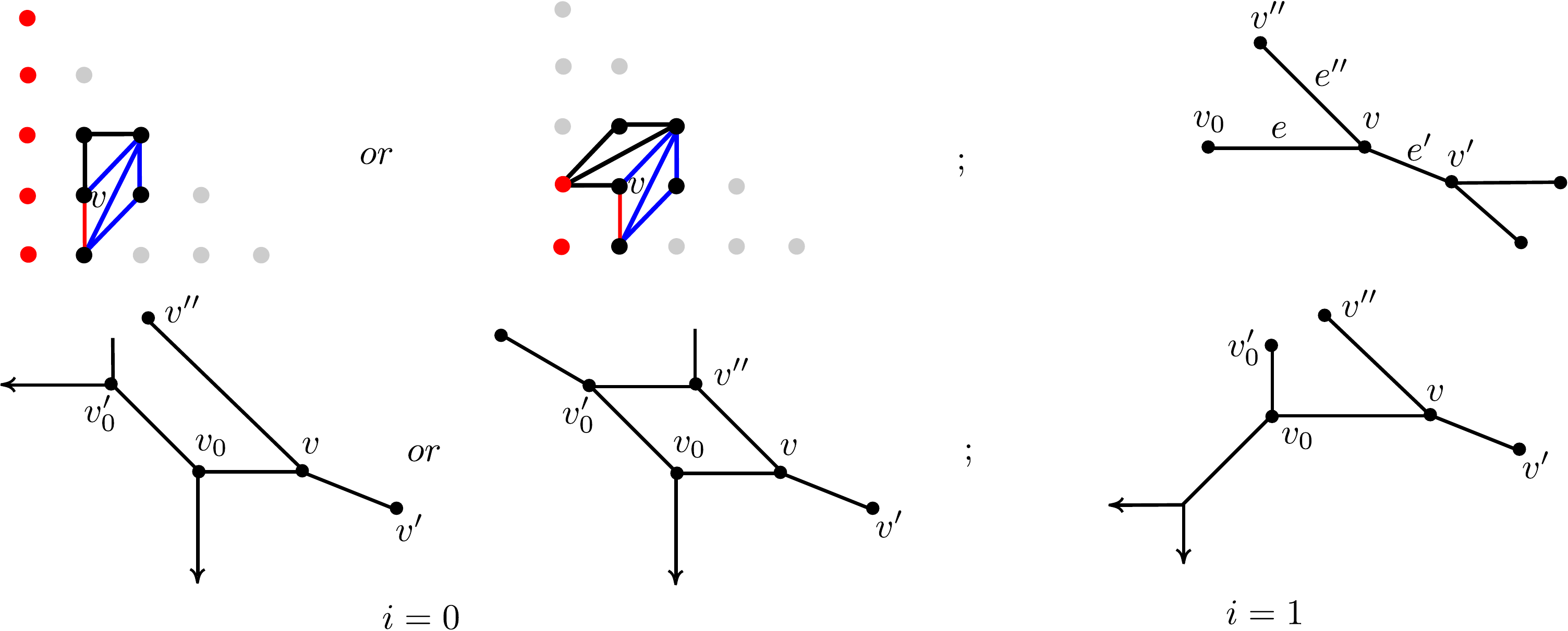}
  \caption{Dual subdivisions and partial information on the tropical curve $\Gamma$ from~\autoref{prop:oneComponentIntersections}. The red dots on the top left figures are the possible points $(0,i)$ that can be chosen to form the triangle $(v_0)^{\vee}$ with the red edge joining $(1,0)$ and $(1,1)$. The (blue) triangle dual to $v$ is labeled. \label{fig:ShapeDLetc}}
\end{figure}

To determine $B$, we start by analyzing the local moves around the vertex $v$ and how can we continue moving the vertex to generate $B$. The relevant information is recorded in~\autoref{fig:localMovesDLetc}. In particular, there are only two possibilities for the triangle dual to $v''$ as shown in~\autoref{fig:ShapeDLetc}.

The edge $e$ restricts our moves around $v$ to the horizontal direction.
   If we move  $v$ in the direction $(1,0)$, the first tangency point of the new line remains at the midpoint of $e$, but the second one lies in the edge $e'$ with direction $(2,-1)$. Our movement stops when the vertical end of the bitangent line meets the vertex $v'$ of $e'$. 

   In turn, if we move $v$ in the direction $(-1,0)$ the tangency points lie in $e$ and $e''$. Notice that this second tangency point belongs to the diagonal end of $\Lambda$. To decide when we stop and how we can continue moving beyond this point, we use the partial other $\prec_d$ from~\autoref{def:relativeOrder} to compare $v_0$ and $v''$. In turn, this will impose certain restrictions on the dual subdivision to $\Gamma$. Each case is depicted in~\autoref{fig:localMovesDLetc}.

   First, assume $v_0\prec_d v''$. By convexity of the connected component of $\RR^2\smallsetminus \Gamma$ dual to $(1,1)$, the triangle dual to $v''$ has vertices $(1,1)$, $(1,2)$ and $(2,2)$ (see the top-left of~\autoref{fig:ShapeDLetc}.) All values of $i$ are possible. In this situation, the movement of $v$ in the direction $(-1,0)$ stops at a point in the relative interior of $e$ (seen in the top-left picture in~\autoref{fig:localMovesDLetc}.) This corresponds to a new bitangent line where  $v''$ is a tangency point of type (2). We cannot move past this point. This yields shape (D).

   On the contrary, assume $v''\preceq_d v_0$. In this situation, we can move from from $v$ along $e$ until we reach $v_0$. If $v''=_d v_0$ we can move beyond $v_0$ along a ray in the direction $(-1,-1)$. This is due to the fact that both edges $e$ and $e''$ have the same lattice length forcing $i=0$ or $1$ by the convexity of the component of $\RR^2\smallsetminus \Gamma$ dual to $(1,1)$. This is seen in the two pictures in the top-right of~\autoref{fig:localMovesDLetc}. The resulting shapes are (L) and (L').

   For the remaining cases, we assume $v''\prec_d v_0$. We let $v_0'$ be the other vertex of $\Gamma$ joined to $v_0$ by a bounded edge whose outer-direction has a positive $y$-coordinate, as in the bottom-right of~\autoref{fig:localMovesDLetc}. If $v_0'=v''$, then $i=1$ and the triangle dual to $v''$ has vertices $(1,1)$, $(0,1)$ and $(2,2)$. Furthermore, we can move the bitangent line with vertex $v_0$ in two directions: an unbounded movement in the direction $(-1,-1)$ and a bounded one along the vertical edge joining $v_0$ with $v''$. Once we reach $v''$, the new tropical bitangent has two tangency points: $v''$ and the midpoint $P$ between $v_0$ and $v''$. Next, we can continue moving in the same vertical direction: $P$ will remain a tangency point, and the second tangency point will be traveling along the remaining edge adjacent to $v''$, which we call $e'''$ (see the bottom-left of~\autoref{fig:localMovesDLetc}). The movement stops once the second endpoint of $e'''$ (called $v'''$) lies in the horizontal end of the new line. We conclude from this analysis that B has shape (P).

\begin{figure}[htb]
  \includegraphics[scale=0.25]{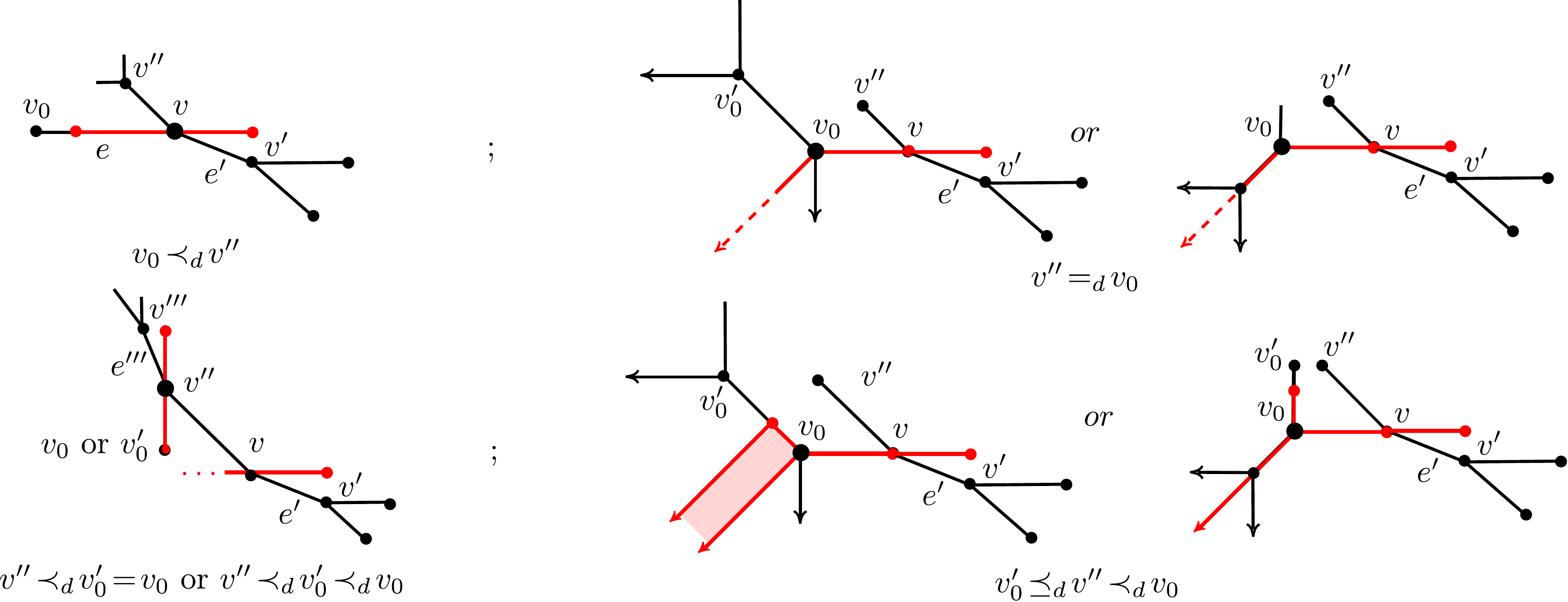}
  \caption{Relevant local moves for vertices of bitangent lines in a bitangent class from~\autoref{prop:oneComponentIntersections} with a shape other than (C) depending on the relative position of $v'', v_0'$ and $v_0$ with respect to $\preceq_d$. The notation for the vertices is that of~\autoref{fig:ShapeDLetc}.\label{fig:localMovesDLetc}}
\end{figure}

Finally, we suppose $v''\prec_d v_0$ and $v_0'\neq v''$. This forces $i=0$ or $1$, as seen in the bottom row of~\autoref{fig:ShapeDLetc}. We claim that B can have shapes (O), (Q), (Q'), (R) or (S), depending on the value of $i$, the dual triangle $(v'')^{\vee}$ and the relative order between $v''$ and $v_0'$ with respect to $\preceq_d$. There are three possible scenarios.

\begin{description}
\item[{Case 1}] $v_0'\preceq_d v''$ and $i=1$. In this situation, the local movement at $v_0$ agrees with that of shape (P), as seen on the bottom-right of~\autoref{fig:localMovesDLetc}. Indeed, we move along the ray with direction $(-1,-1)$ and upwards along the vertical edge joining $v_0$ and $v_0'$ until the diagonal end of the tropical line contains $v''$. Furthermore, since $v_0'\neq v''$, we get $\ell(e'')<2\ell(e)$, so $v_0'\prec_d v''$ and $(v'')^{\vee}$ is the triangle with vertices $(1,1)$, $(1,2)$ and $(2,2)$. Note that $v''$ is a type (2) tangency point of this new line, so the movement stops at a point in the relative interior of the edge $\overline{v_0v_0'}$. This yields shape (O).

\item [{Case 2}]  $v_0'\preceq_d v''$ and $i=0$. It follows that $(v'')^{\vee}$ is the triangle with vertices $(1,1)$, $(1,2)$ and $(2,2)$. The local move at $v_0$ is seen at the bottom-center  picture in~\autoref{fig:localMovesDLetc}: we move along the edge joining $v_0$ and $v_0'$ allowing for an unbounded movement in the direction $(-1,-1)$ from any point in this segment. The movement stops once the new line contains $v''$ in its diagonal end. If $v_0'\prec_d v''$, the stopping point in the interior of the edge $\overline{v_0v_0'}$, and $B$ has shape (Q). If $v_0'=_d v''$, the movement stops at $v_0'$, and $B$ has shape (Q').

\item [{Case 3}] $v'' \prec_d v_0'$. This forces $i=0$ and $(v_0')^{\vee}$ to be the triangle with vertices $(0,0)$, $(1,1)$ and $(0,1)$. The vertex $v_0$ lies in $B$ and  can be moved along the edge $\overline{v_0v_0'}$ and along rays with direction $(-1,-1)$ at each point in this segment until we reach $v_0'$ (as in the bottom-center of~\autoref{fig:localMovesDLetc}). From $v_0'$, we can further move along the vertical edge containing $v_0'$ until the diagonal end of the line meets $v''$.

  If $v_0'$ is not adjacent to $v''$ in $\Gamma$, then $(v'')^{\vee}$ is the triangle with vertices $(1,1)$, $(1,2)$ and $(2,2)$ and the movement stops at the interior point of the vertical edge of $\Gamma$ containing $v_0'$. Thus, $B$ has shape (R).

  On the contrary, if $v_0'$ and $v''$ are adjacent in $\Gamma$, then $(v'')^{\vee}$ is the triangle with vertices $(1,1)$, $(2,2)$ and $(0,1)$. The movement stops at $v''$ and continues as for shape (P), as we see on the bottom-left of \autoref{fig:localMovesDLetc}. Thus, $B$ has shape (S).\qedhere
\end{description}
\end{proof}

\begin{remark}\label{rem:chipFiringsOneIntersection}
  A key argument in the proof of~\autoref{prop:oneComponentIntersections}  involves the comparison between the relative $\preceq_d$-order  among various vertices in the bounded connected component of $\RR^2\smallsetminus \Gamma$ dual to $(1,1)$. In turn, this yields an order between the lattice lengths of certain bounded edges of $\Gamma$ and partial knowledge of the dual subdivision to $\Gamma$. Following the notation of~\autoref{fig:PossibleSkeletons}, we conclude that shapes (P) and (S) can only arise when the skeleton of $\Gamma$ is the graph (212) and we have $\ell_6 = \len(e'')$, $\len(e) + \len(e'')\leq \ell_1=2 \len(e'')$.

  For the remaining shapes listed in~\autoref{prop:oneComponentIntersections} except (C), the skeleton of $\Gamma$ correspond to the graph (111) in~\autoref{fig:PossibleSkeletons} and $\ell_6=\len(e'')$. The partial information on the dual subdivisions to $\Gamma$ provided by each shape imposes different restrictions of the lengths $\min\{\ell_4,\ell_5\}$ and $\ell_1$. For example, to obtain shapes (L) and (L') we must have $\ell_6=\len(e'')$, $\ell_1 = 3\len(e'')$ and $\min\{\ell_4,\ell_5\}\leq \len(e)$. For (O), we have $\min\{\ell_4,\ell_5\}\leq \len(e)-\len(e'')$, $\ell_6=\len(e'')$ and $\ell_1 =\len(e'')+\len(e)$. Similar restrictions arise for  shapes  (D), (Q), (Q') and (R). 

  For these bitangent classes described here, the associated divisor $D$ places one chip on each side of the central loop. The tropical semimodule $R(D)$ associated to the linear system $|D|$ is a line segment~\cite{haa.mus.yu:12}. We can view it in the bounded cells of $B$ that are either outside $\Gamma$ or on the central loop of the graph.

  Finally, if $B$ has shape (C), then the skeleton of $\Gamma$ is the graph (000) from~\autoref{fig:PossibleSkeletons}, with $\ell_1\leq \ell_2\leq \ell_3$. Furthermore, the associated divisor $D$  places one chip $(\ell_2-\ell_1)/2$ and $(\ell_3-\ell_1)/2$ units away from the vertex of the two largest edges, and $R(D)$ is a single vertex.
\end{remark}

To conclude our classification of bitangent classes, we focus on the last row of~\autoref{tab:classificationShapes}. All members of such classes have two distinct tangency points. In order to simplify the exposition, we break symmetry by considering the tangencies to be either in the diagonal or the horizontal end of $\Lambda$. This non-uniform convention will allow us to simplify the lifting obstructions in~\autoref{sec:lift-trop-bitang}.

\begin{proposition}\label{prop:shapesfortwoconnectedcomp}
Let $B$ be a bitangent class of $\Gamma$ where every member  intersects $\Gamma$ in two connected components.
Then, up to $\Sn{3}$-symmetry, $B$ has one of the following shapes: (A), (B), (E), (F), (G), (H), (H'), (I), (J), (K), (M), (N), (T), (T'), (T''), (U), (U'), (V), (W), (X), (Y), (Z) or (AA) through (HH), depicted in~\autoref{fig:2dCells}.
\end{proposition}

\begin{proof} We prove the statement by a case-by-case analysis, based on the dimension and boundedness of the top-dimensional cells of $B$. To simplify the exposition, we treat each case in five separate lemmas below. The classes with two-dimensional top-cells are discussed in~
  {Lemmas}~\ref{lm:twoCellUnbounded} and~\ref{lm:twoCellBounded}. 
  {Lemmas}~\ref{lm:oneCellUnbounded} and~\ref{lm:oneCellBounded} concern classes with one-dimensional top-cells. Finally, zero-dimensional classes are the subject of~\autoref{lm:zeroCell}.
\end{proof}

The next two lemmas address the possible shapes of two-dimensional bitangent classes:

\begin{lemma}\label{lm:twoCellUnbounded}
Let $B$ be a two-dimensional  bitangent class of $\Gamma$ where every member intersects $\Gamma$ in two connected components. Assume $B$ has an unbounded top-dimensional cell.
Then, up to $\Sn{3}$-symmetry, $B$ has one of the following six shapes:  (T), (T'), (T''), (U), (U') or (V), as depicted in~\autoref{fig:2dCells}.
\end{lemma}

    \begin{figure}[tb]
      \includegraphics[scale=0.25]{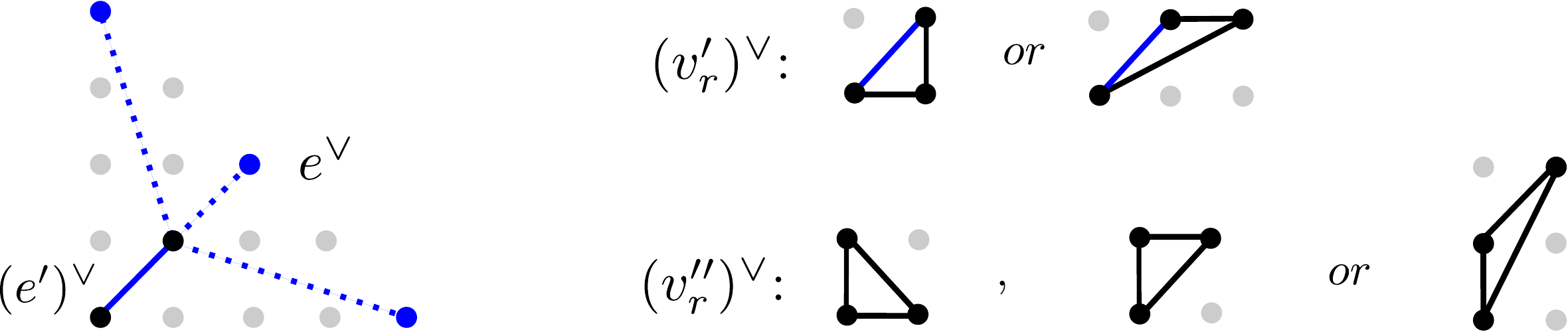}
      \caption{Partial dual subdivisions to $\Gamma$ relevant for~\autoref{lm:twoCellUnbounded}. Here, the vertex $v_r''$ is joined to $v_r'$ by a horizontal edge in $\Gamma$, which can only happen if $(v_r')^{\vee}$ is the first triangle in the top-right.\label{fig:partialNPTelal}}
    \end{figure}

    \begin{proof} We let $\sigma$ be an unbounded two-dimensional cell in $B$. Combining~\autoref{lm:unboundedCellsAndDirs} with the action of $\Sn{3}$, we may assume $\sigma$ lies in the chamber of $\RR^2\smallsetminus \Gamma$ dual to $(0,0)$. Any point $v\in \sigma^{\circ}$ yields a line $\Lambda$ with two distinct type (1) tangency points along its diagonal end, called $P$ and $P'$. Without loss of generality, we assume $P'$ lies in between $v$ and $P$ in this end.

      We let $e$ and $e'$ be the two bounded edges of $\Gamma$ containing $P$ and $P'$, respectively. Note that $e$ and $e'$ are boundary edges of the same connected component of $\RR^2\smallsetminus \Gamma$. Their dual edges $e^{\vee}$ and $(e')^{\vee}$ are adjacent to the same vertex $(i,j)$ in the dual subdivision to $\Gamma$.
      \autoref{tab:edgeOnEdgeMult2} shows the three possible  directions for $e^{\vee}$, namely, $(1,1)$, $(3,-1)$ and $(1,-3)$, whereas $(e')^{\vee}$ must have direction $(-1,1)$. Thus, since the edge $(e')^{\vee}$ contains $(0,0)$, we conclude that $i=j=1$ and the second endpoint of $e^{\vee}$ equals $(k,4-k)$ for $k=0,2$ or $4$. Both dual edges are  marked by  blue segments in the relevant partial subdivisions in the top-left of~\autoref{fig:partialNPTelal}. Since the position of $e^{\vee}$ is uncertain, we use a dotted segment for it.
      
      \begin{figure}[tb]
      \includegraphics[scale=0.4]{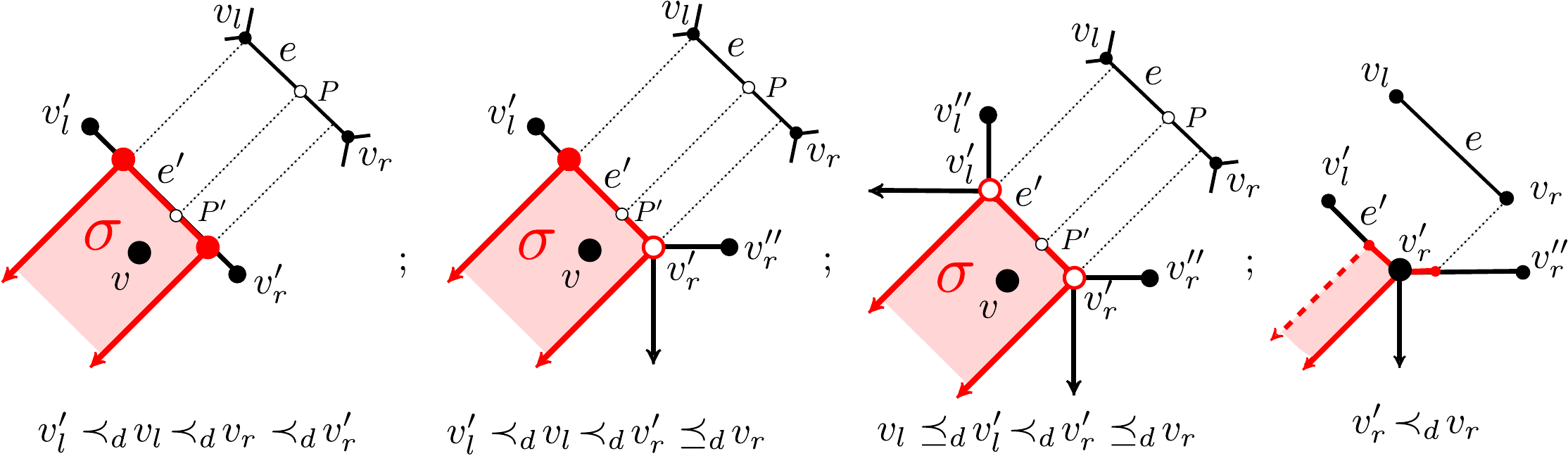}
      \caption{From left to right: possible realizations for the unbounded two-dimensional cell $\sigma$ depending on the relative $\preceq_d$-order of the vertices $v_{l},v_r,v_l',v_r'$ of $\Gamma$, and local movement at $v_r'$ when $v_r'\prec_d v_r$. The movement along the horizontal edge $\overline{v_r'v_r''}$ of $\Gamma$ is restricted by $v_r$.\label{fig:optionsForSigma}}
    \end{figure}

      We let $v_{l},v_r,v_l',v_r'$, be the left and right vertices of $e$ and $e'$, respectively. We use the convention, $v_l\prec_d v_r$ to characterize left and right for the edge $e$ since its slope is undetermined. The position of $(e')^{\vee}$ yields two options for each of the triangles $(v_l')^{\vee}$ and $(v_r')^{\vee}$ (seen on the top-right of~\autoref{fig:partialNPTelal}.) The possible triangles $(v_l')^{\vee}$ are obtained from $(v_r')^{\vee}$ using the map $\tau_0$ from~\autoref{tab:S3Action}.

      Our next objective is to classify the possible shapes of $B$. We start by  focusing our attention on the cell $\sigma$. Since $v\in \sigma^{\circ}$ and $P$ and $P'$ are type (1) tangencies in the diagonal end of $\Gamma$, we can move $v$ locally in the $(1,1)$ direction while remaining in $\sigma^{\circ}$. The movement stops when we reach $P'$, which becomes a type (4) tangency of the new line. By~\autoref{lm:unboundedCellsAndDirs}, we conclude that $B\cap e'=\sigma\cap e'$. This set must be a segment by the description of local moves for type (4) tangencies. The possible shapes for $\sigma$  will be completely characterized by the segment $\sigma\cap e'$  since $\sigma  = \sigma\cap e' + \RR_{\leq 0}(1,1)$.
      
      The set $\sigma \cap e'$ is determined by the relative order of $v_l,v_r,v'_r$ and $v'_l$ with respect to the partial order $\preceq_d$ from~\autoref{def:relativeOrder}. As we move $v$ along $e'$ towards $v_l'$, the point $P$ travels along $e$. The movement stops either when $v$ reaches $v_l'$ (if $v_l'\preceq_d v_l$) or when the diagonal end of the new line meets $v_l$ (if $v_l\preceq_d v_l'$), whichever happens first. The latter yields a stopping point in the relative interior of $e$ if $v_l\prec_d v_l'$.

            Notice that $v_l\prec_d  v\prec_d v_r$ and $v_l'\prec_d v\prec_d v_r'$, so $v_l'\prec_d v_r$. By using the map  $\tau_0$  we reduce our analysis to three cases, namely $v_l'\prec_d v_l\prec_d v_r \prec_d v_r'$, $v_l'\prec_d v_l\prec_d v_r' \preceq_d v_r$, or $v_l\preceq_d v_l'\prec_d v_r' \preceq_d v_r$. These are precisely the three options depicted in~\autoref{fig:optionsForSigma}.

      The convexity of the connected component of $\RR^2\smallsetminus \Gamma$ dual to $(1,1)$ ensures that  $v_r'\preceq_d v_r$ if and only if $(v_r')^{\vee}$ is the triangle with vertices $(0,0)$, $(1,1)$ and $(1,0)$. Thus, $v_r'$ is adjacent to a vertex $v_r''$ of $\Gamma$ along a horizontal edge. The three possible triangles $(v_r'')^{\vee}$ are seen in the bottom-right of~\autoref{fig:partialNPTelal}. Symmetric behavior is observed when comparing $v_l$ and $v_l'$: $v_l\preceq_d v_l'$ if and only if $(v_l')^{\vee}$ is the triangle with vertices $(0,0)$,  $(1,1)$ and $(0,1)$.

     To finish the classification of shapes for the bitangent class $B$, we must analyze what happens if $v_r'\in B$ and/or $v_l'\in B$. By symmetry, we restrict our attention to the case when $v_r'\in B$. There is only one possible movement beyond $v_r'$, and it can only occur if $v_r'\prec_d v_r$.  In this situation, can move along the horizontal edge $\overline{v_r'v_r''}$: one tangency point lies on this edge while the other one travels along the edge $e$ towards $v_r$. The movement stops when the second tangency point reaches $v_r$ or when the vertex of the new line reaches $v_r''$, whichever happens first. This will be determined by the relative order between $v_r$ and $v_r''$ with respect to $\preceq_d$.

     If $v_r''\preceq_d v_r$, the partial information on the dual subdivision recorded so far forces either $v_r''=v_r$ or $v_r''$ to be adjacent to $v_r$ along a slope one bounded edge of $B$. In both situations, the line with vertex $v_r''$ will be a member of $B$ which meets $\Gamma$ in a single connected component. This cannot happen by our assumptions on $B$. We conclude that  $v_r\prec_d v_r''$ so the movement along the horizontal edge $\overline{v_r'v_r''}$ stops at a point in its relative interior, as seen on the right of~\autoref{fig:optionsForSigma}.

     The above discussion confirms that we can attach a one-dimension horizontal cell to $B$ at $v_r'$ if and only if $v_r'\prec_d v_r$. Symmetrically, we can attach a one-dimensional vertical cell to $B$ at $v_l'$ if and only if $v_l\prec_d v_l'$.

      The above analysis on the three options for the two-cell $\sigma$ combined with the restrictions to move past $v_r'$ or $v_l'$ whenever these vertices lie in $B$ yields the six possible shapes in the statement (up to $\Sn{3}$-symmetry.) This concludes our proof.
\end{proof}

\begin{table}[tb]
  \begin{tabular}{|c|c|c|c|c|c|}
\hline    $e'$ vs. $e$  & (1) & (2) & (3) & (4) & (5)\\
    \hline
        (1) & (W) & $\tau_1$(X) & $\tau_1$(Y) & (GG) & (EE) \\
    \hline
    (2) & (X) & (Z) &  (AA) & (HH) & (FF) \\
    \hline
    (3) & (Y) & $\tau_1$(AA)  & (BB) & (DD) & (CC) \\
    \hline
      \end{tabular}
    \caption{Classification of shapes of two-dimensional classes with no unbounded top-cells, following~\autoref{fig:2dCells}. The map $\tau_1$ is described in~\autoref{tab:S3Action}.\label{tab:classification2DBoundedShapes}}
  \end{table}

\begin{lemma}\label{lm:twoCellBounded}
Let $B$ be a two-dimensional  bitangent class of $\Gamma$ for which every member intersects $\Gamma$ in two connected components. Assume all top-dimensional cells in $B$ are bounded.
Then, up to $\Sn{3}$-symmetry, $B$ has one of the following twelve shapes:  (W) through (Z) and (AA) through (HH) (see~\autoref{tab:classification2DBoundedShapes}.)
\end{lemma}

\begin{proof} In order for all two-dimensional cells of $B$ to be bounded, the tangency points for each member must occur in the relative interior of two different ends of the  bitangent line. Without loss of generality, we assume they belong to the  horizontal  and diagonal ones. We let $e$ and $e'$ be the two bounded edges of $\Gamma$ where these two tangencies occur. By picking a point $v$ in the relative interior of a two-cell $\sigma$, we conclude that these two points are of type (1), i.e. they lie in the relative interior of $e$ and $e'$, respectively.

  As in the proof of~\autoref{lm:twoCellUnbounded}, $e$ and $e'$ must lie in the boundary of the same unbounded connected component of $\RR^2\smallsetminus \Gamma$. Furthermore, the dual point corresponding to such component equals $(i,0)$ for $i=1,2$ or $3$. The possible directions for the dual cells $e^{\vee}$ and $(e')^{\vee}$ are listed on~\autoref{tab:edgeOnEdgeMult2}. The fact that $e$ lies to the left of $e'$, combined with the directions for $e^{\vee}$ and $e'^{\vee}$ forces $i=2$. Moreover, only two of the three possible directions for $e^{\vee}$ and $(e')^{\vee}$ can occur, namely $(-2,1)$ and $(-2,3)$ for $e^{\vee}$ and $(1,1)$ or $(-1,3)$ for $(e')^{\vee}$. All four combinations are possible. Using the map $\tau_1$ from~\autoref{tab:S3Action} we can further restrict our analysis to three of them, seen to the left of~\autoref{fig:nightmare}.

  We let $v_1,v_2, v_3$ and $v_4$ be the endpoints of $e$ and $e'$, clockwise oriented. \autoref{fig:nightmare} depicts the location of these vertices along four dotted lines: two horizontal ones containing $v_1$ or $v_2$, and two slope one lines containing $v_3$ or $v_4$. 
  We let $\mathcal{P}$ be the parallelogram determined by these four lines. Note that the upper left corner of $\mathcal{P}$ must lie in the connected component of $\RR^2\smallsetminus \Gamma$ dual to $(2,0)$ so the lines spanned by $e$ and $e'$ must intersect above this point. The intersection $B \cap\mathcal{P}$ will yield a unique two-dimensional cell. This cell will be determined by the location of $e$ and $e'$, relative to $\mathcal{P}$.

  There are up to five cases for the relative position of each edge, which we describe below.  Since we may assume that the slanted side of $\mathcal{P}$ is at least as large as the horizontal size, the options for $e'$ are reduced to three. The first three cases for each edge admit a common description. We accomplish this by writing  these two edges as $\overline{v_jv_{j+1}}$ for $j=1,3$.  For Case (1), the slopes of  $e$ and $e'$ are not further restricted. For all remaining cases, the edges $e$ and $e'$ have directions $(1,2)$ and $(-1,1)$, respectively.
  \begin{itemize}
  \item[(1):] \emph{The edge $\overline{v_jv_{j+1}}$ avoids $\mathcal{P}$}. For $j=1$, this means the bottom and left edges of $\mathcal{P}$ belong to $B$. 
    If $j=3$, then the top and right edges of $\mathcal{P}$ lie in $B$. 
  \item[(2):] \emph{$v_j$ is a vertex of $\mathcal{P}$.} If so, $v_j\in B$ and it has a ray adjacent to $v_j$ in $B$. This ray includes a bounded edge of $\Gamma$ adjacent to $v_j$ followed by an end with the same slope. For $j=1$ this end and edge have direction  $(-1,-1)$, and the bottom and left edges of $\mathcal{P}$ lie in $B$.   If  $j=3$, then the direction is $(1,0)$ and the top and left edges of $\mathcal{P}$ belong to $B$. 
  \item[(3):] \emph{$v_j$ lies in the relative interior of an edge of $\mathcal{P}$}. If so,   $\mathcal{P}\cap \overline{v_jv_{j+1}}\subseteq B$, and a ray preceded by a bounded edge of slope $2$ (for $j=1)$ or $-1$ (for $j=3$) adjacent to  $v_j$ appears in $B$. The ray consists of a bounded edge of $\Gamma$ followed by an end of the same direction as in Case (2).
  \item[(4):] \emph{$v_1$ is a vertex of $\mathcal{P}$ and $e\cap  \mathcal{P}$ is a segment containing $v_1$}. The description of $B$ around $v_1$ is the same as in (3).
  \item[(5):] \emph{$v_1\notin \mathcal{P}$ and $e\cap  \mathcal{P}$ is a segment, not containing $v_1$}.  As with (3), the intersection  $\mathcal{P}\cap \overline{v_1v_{2}}$ belongs to $B$, but the role of $v_1$ is now played by the lower endpoint of this segment. No bounded edge or rays are attached to this endpoint. 
  \end{itemize}
  The above description shows that each of these 15 combinations yields a unique shape, which we indicate in~\autoref{tab:classification2DBoundedShapes}. This concludes our proof.
\end{proof}
  \begin{figure}[tb]
\includegraphics[scale=0.4]{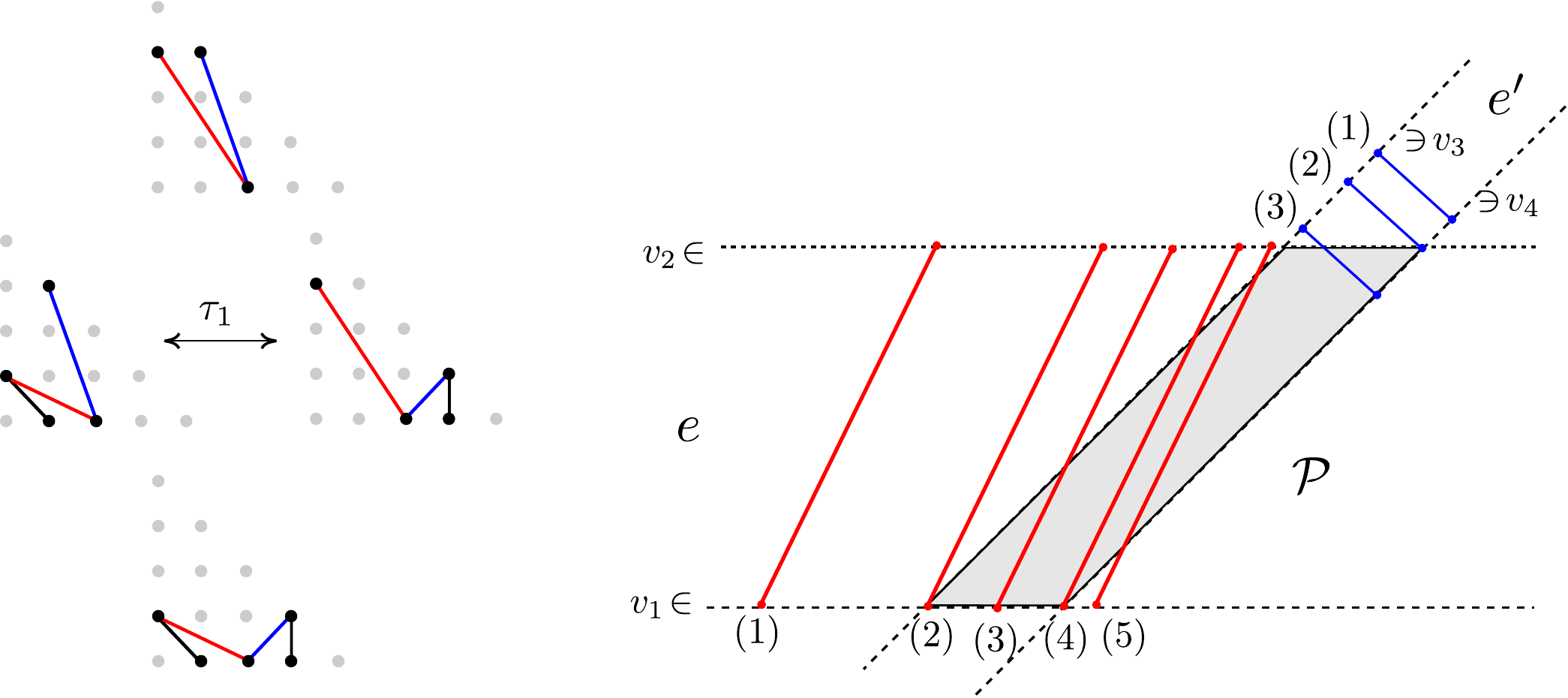}
\caption{From left to right: all possible partial dual subdivisions to $\Gamma$ when all two-dimensional cells of $B$ are bounded, and positions of the (red) edge $e=\overline{v_1v_2}$ (respectively, the (blue) edge $e'=\overline{v_3v_4}$) relative to $\mathcal{P}$.\label{fig:nightmare}}
\end{figure}

The next two lemmas discuss bitangent classes of dimension one.
\begin{lemma}\label{lm:oneCellUnbounded}
Let $B$ be an unbounded one-dimensional  bitangent class of $\Gamma$ where every member intersects $\Gamma$ in two connected components. 
Then, up to $\Sn{3}$-symmetry, $B$ has one of  seven possible shapes:  (H), (H'), (I), (J), (K), (M) or (N).
\end{lemma}

\begin{proof} We let $\sigma$ be an unbounded one-dimensional cell of the shape refining $B$ and let $\rho$ be its unbounded direction, as in~\autoref{def:unboundedDirections}.  We let $v$ be the  unique vertex of  $\sigma$ and set $\Lambda$ to be the associated  bitangent line. By assumption, $\Lambda$ has two multiplicity two tangency points, one of which is $v$. We let $P$ be the second tangency point. By~\autoref{lm:unboundedCellsAndDirs}, it lies in the end of $\Lambda$ with direction $-\rho$ and the connected component of $\RR^2\smallsetminus \Gamma$ containing $\sigma^{\circ}$ is determined by $\rho$. Furthermore, $\sigma = v + \RR_{\geq 0} \,\rho$.

 \autoref{lm:localMoves} restricts the tangency types of $v$ to three cases: (4), (3b) or (6a). In turn, $P$ can only have tangency type (1), (2) or (3c). Since $\dim B=1$, out of the nine possible combinations, only six are possible. Thus, the local tangencies for the pair $(v,P)$ are reduced to six combinations. The shapes associated to each pair are listed in~\autoref{tab:classification1DUnboundedShapes}. They are determined by how we can move from $v$ away from $\sigma$ while remaining in $B$. Each case is explained in detail below.

  \begin{table}[tb]
  \begin{tabular}{|c|c|c|c|}
    \hline     &\multicolumn{3}{c|} {$P$-type}\\
   \cline{2-4} $v$-type & (1) & (2) & (3c) \\
    \hline
    (4) &      ----- & ----- & (H)\\
    \hline

    (3b) &  (N) &  (I) &  (M) or (J)\\
    \hline
    (6a) & ----- &  (K)  &  (H') \\
    \hline
      \end{tabular}
    \caption{Classification of one-dimensional shapes with an unbounded top-cell by the types of local tangencies for $v$ (the vertex) and $P$.\label{tab:classification1DUnboundedShapes}}
  \end{table}

  \begin{figure}[tb]
    \includegraphics[scale=0.2]{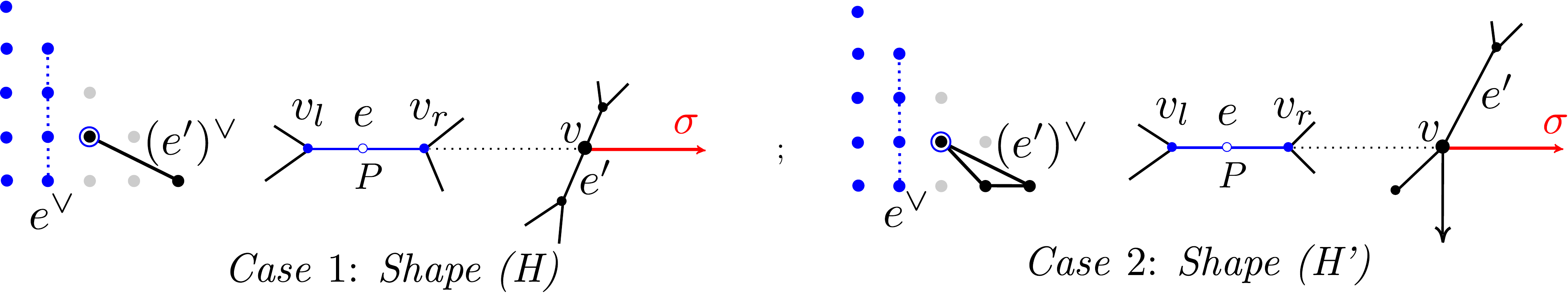}
    \caption{Cases 1 and 2 from~\autoref{lm:oneCellUnbounded}, leading to shapes (H) and (H'). The tangencies $P$ and $v$ are horizontally aligned (using a black dotted line.) The potential location of $e^{\vee}$ is marked with a dotted blue line. One of the blue points $(0,w)$ is a vertex of $v_l^{\vee}$\label{fig:Cases1And2_1dim}}
  \end{figure}

\begin{description}

\item[{Case 1}] type $(v,P)=((4), (3c))$. Exploiting the $\Sn{3}$-symmetry, we assume $P$ lies in the horizontal end of $\Lambda$, as seen on the left of~\autoref{fig:Cases1And2_1dim}. We let $e$ and $e'$ be the edges of $\Gamma$ containing the tangencies $P$ and $v$, respectively. Since $v$ is of type (4), we conclude that $B=\sigma$, so $B$ has shape (H).

  Next, we infer partial information about the dual subdivision to $\Gamma$, which we need for~\autoref{cor:combClass}. Write $v_l$ and $v_r$ for the left and right vertices of the horizontal edge $e$. By~\autoref{lm:unboundedCellsAndDirs}, the edge $e'$ of $\Gamma$ containing $v$ is in the boundary of the connected component of $\RR^2\smallsetminus \Gamma$ dual to $(4,0)$. This information, combined with~\autoref{tab:edgeOnEdgeMult2} ensures the remaining vertex of $(e')^{\vee}$ is $(2,1)$. In turn, since $v_r$ and $e'$ are in the boundary of the connected component of $\RR^2\smallsetminus \Gamma$ dual to $(2,1)$, but $e^{\vee}$ is not, it follows that $e^{\vee}$ has vertices $(1,i)$ and $(1,i+1)$ with $i=0,1,2$, as we see in the left of~\autoref{fig:Cases1And2_1dim}. This edge is joined to $(2,1)$ and a vertex of the form $(0,w)$ to determine the dual triangles $v_l^{\vee}$ and $v_r^{\vee}$, respectively.

\item[{Case 2}] type $(v,P)=((6a), (3c))$. This case is similar to Case 1. Since $v$ has type (6a), exploiting the $\Sn{3}$-symmetry we may assume $P$ lies in the horizontal end of $v$ and that $v^{\vee}$ is the triangle with vertices $(4,0)$, $(3,0)$ and $(2,1)$, as we see in the right of~\autoref{fig:Cases1And2_1dim}. We cannot move beyond $v$ so the bitangent class $B$ has shape (H'). The data on $e^{\vee}$, $v_l^{\vee}$ and $v_r^{\vee}$ matches Case 1.

\item[{Case 3}] type $(v,P)=((3b), (1))$. In this situation, we assume $P$ lies on the diagonal end of $\Gamma$, and we let $v_r$ and $v_l$ be the vertices of the edge $e$ of $\Gamma$ containing $P$, with $v_l\prec_d P \prec_d v_r$.
  
  By~\autoref{lm:unboundedCellsAndDirs}, $v$ lies in the boundary of the connected component of $\RR^2\smallsetminus \Gamma$ dual to $(0,0)$. We let $e'$ be the diagonal edge of $\Gamma$ containing $v$ and let $v'$ be its other endpoint. The dual triangle $v^{\vee}$ must then have vertices $(0,0)$, $(1,0)$ and $(0,1)$.
  
  We claim that out of the three possible dual triangles $(v')^{\vee}$ depicted on the left of~\autoref{fig:Case3_1dim}, only one is feasible. Since the bottom two are related by the  map $\tau_0$, it suffices to analyze the first two.

  The second one (where $v'$ is adjacent to a horizontal end of $\Gamma$) can be ruled out for dimension reasons. Indeed, since $v'$ is a type (6a) tangency and $P$ has type (1), we can move beyond $v'$ while remaining in $B$ if we restrict to a circular sector with center $v'$ and bounded by edges with directions $(1,1)$ and $(2,1)$, as seen on the right of~\autoref{fig:Case3_1dim}. This cannot happen since $\dim(B)=1$. Thus, the star of $v'$ at $\Gamma$ must be a min-tropical line, as depicted in the center of the same figure. 

  Once $\Star_{\Gamma}(v')$ is determined, we let $v'_t$ and $v'_r$ be the vertices of $\Gamma$ connected to $v$ by a vertical and horizontal edge, respectively. The combined tangency types of $v'$ and $P$ allow us to move past $v'$ along these two bounded edges while remaining in $B$. The stopping point will be determined by the relative $\preceq_d$-order  between $v_l$ and $v'_t$, respectively $v_r'$ and $v_r$. 

  We claim that $v_t'\prec_d v_l$ and $v_r\prec_d v_r'$. If so, the movement from $v'$ along the vertical and horizontal edges stops once the diagonal end of the new bitangent lines reaches $v_l$, respectively $v_r$. These stopping points are in the relative interior of the edges $\overline{v'v_t'}$ and $\overline{v'v_r'}$ and are diagonally aligned with $v_l$ and $v_r$, respectively. Thus, $B$ has shape (N). 

  Since the claims are symmetric, it suffices to show $v_r\prec_d v_r'$. We do so by analyzing the three possibilities for the triangle $(v_r')^{\vee}$ (seen in the picture.)  The convexity of the connected component of $\RR^2\smallsetminus \Gamma$ dual to $(1,1)$ ensures that for all three cases we have $v_r\preceq_d v_r'$. Furthermore, equality holds if either $v_r= v_r'$ or $v_r=_d v_r'$ and both vertices are adjacent in $\Gamma$. In both situations, $v_r'\in B$ and the corresponding bitangent line intersects $\Gamma$ in a connected set. This cannot happen by our assumptions on $B$. Thus, in all three cases, we have $v_r\prec_d v_r'$, as we wanted.  

  \begin{figure}[tb]
    \includegraphics[scale=0.23]{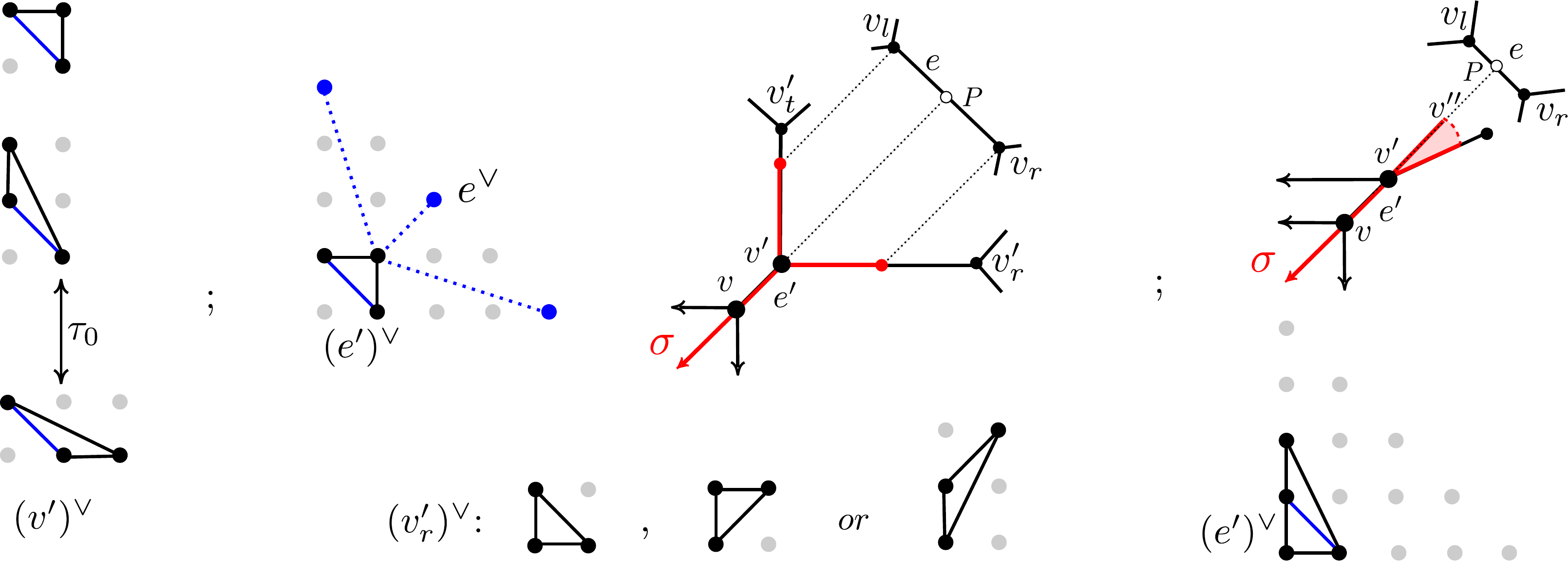}
    \caption{Partial dual subdivisions and local moves for Case 3 from~\autoref{lm:oneCellUnbounded}, leading to shape (N). Points joined by black dotted lines are diagonally aligned. The dotted blue edges in the center correspond to the potential locations of $e^{\vee}$, whereas the blue edge is $(e')^{\vee}$. Imposing the condition $v_l=_d v'$ in the central picture will give shape (I) and Case 4.\label{fig:Case3_1dim}}
  \end{figure}

\item[{Case 4}] type $(v,P)=((3b), (2))$. This situation is very similar to Case 3, with the exception that now $P=v_r$ or $v_l$. The triangle $v^{\vee}$ is fixed as in~\autoref{fig:Case3_1dim}. 
  Since $\dim(B)\neq 2$, there is only one option for the dual triangle to the vertex $v'$ adjacent to $v$ along a slope one edge of $\Gamma$. The partial dual subdivision to $\Gamma$ is depicted  in the center of the same figure, and the star of $v'$ at $\Gamma$ is a min-tropical line.
  
  By applying the map $\tau_0$ if necessary, we may assume $P=v_l$ and we let $v_l'$ be the vertex of $\Gamma$ adjacent to $v'$ by a horizontal bounded edge. The same reasoning as in Case 3 confirms that  we can move from $v'$  along this edge while remaining in $B$ but we must stop at a point in the relative interior of this edge since $v_r\prec_d v'_r$. Thus, $B$ has shape (I).

    \begin{figure}[tb]
    \includegraphics[scale=0.2]{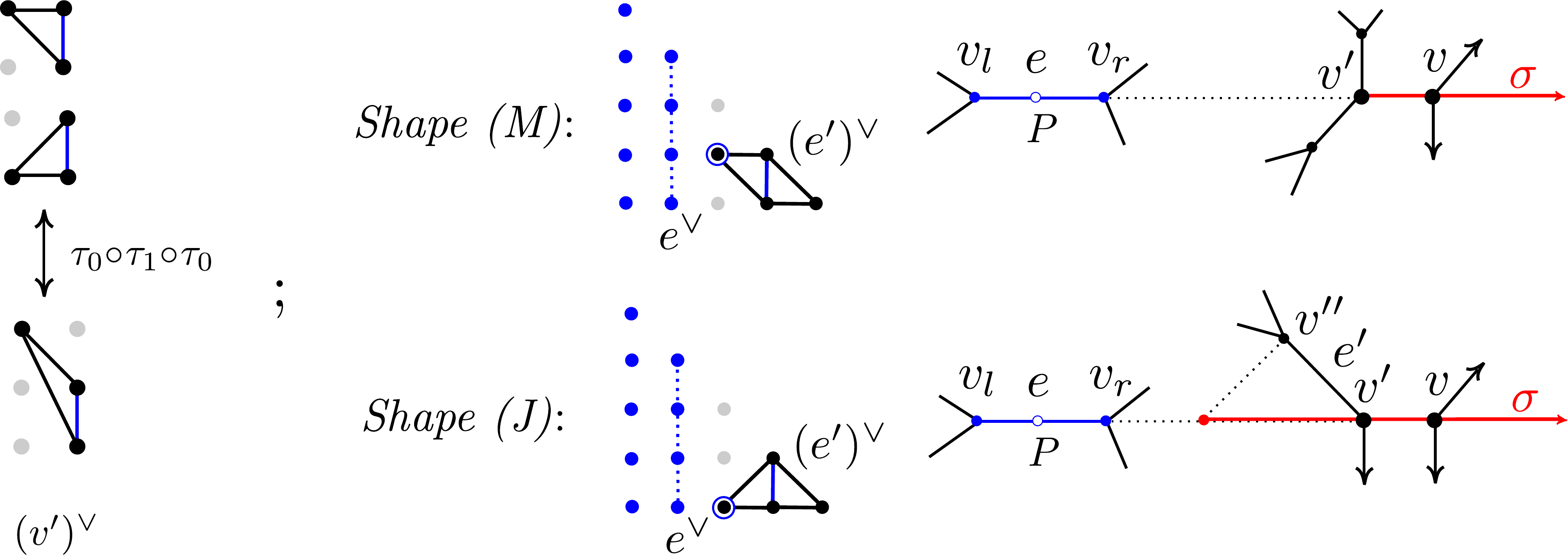}
    \caption{From left to right and top to bottom: possible triangles $(v')^{\vee}$,  partial dual subdivisions and local moves for Case 5 from~\autoref{lm:oneCellUnbounded}, leading to shapes (M) and (J), respectively.\label{fig:Case4_1dim}}
  \end{figure}

\item[{Case 5}] type $(v,P)=((3b), (3c))$. We assume that $P$ lies in the horizontal end of $\Gamma$, so the star of $v$ at $\Gamma$ is a tropical line because $(4,0)\in v^{\vee}$ by~\autoref{lm:unboundedCellsAndDirs}. We let $v'$ be the vertex of $\Gamma$ adjacent to $v$ by a horizontal edge, and let $v_l$ and $v_r$ be the left and right vertices of the horizontal edge $e$ of $\Gamma$ containing $P$. By construction, we can move $v$ along $e$ until we reach $v'$ while remaining in $B$, as in~\autoref{fig:Case4_1dim}.

  It remains to analyze the local moves at $v'$. This will be determined by the dual triangle $(v')^{\vee}$.   There are three possibilities for  $(v')^{\vee}$, as seen on the left of~\autoref{fig:Case4_1dim}, but only two up to $\Sn{3}$-symmetry. Each of them leads to either Shape (M) or (J), as we now explain.

  If $v'$ is adjacent to a vertical bounded end of $\Gamma$, then the line with vertex $v'$ has a type (5a) tangency at $v'$ (seen in the top of the figure.) We cannot move past $v'$ while remaining in $B$ by the tangency point $P$. Thus,  $B$ has shape (M).

  On the contrary,   if $v'$ is adjacent to a vertical end of $\Gamma$, then $v'$ becomes a type (6a) tangency and we can move $v'$ in the direction $(-1,0)$ while remaining in $B$, since the second tangency point will lie in the slope -1 edge $e'$ adjacent to $v'$. The movement stops once the diagonal end of the new line reaches the other endpoint of $e'$, called $v''$ in the figure. This is guaranteed because $v_r\prec_d v''$, and this follows since the connected component of $\RR^2\smallsetminus \Gamma$  dual to $(2,0)$ is convex and contains both $v_r$ and $v''$ in its boundary. Thus, $B$ has shape (J).

  \begin{figure}[tb]
   \includegraphics[scale=0.23]{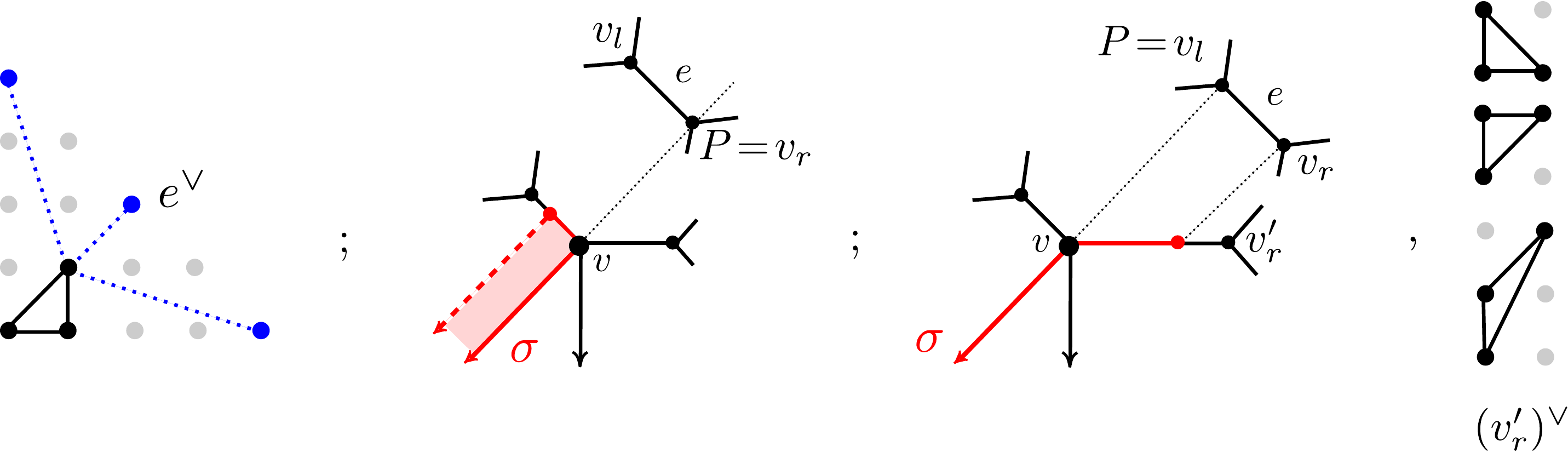}
      \caption{Relevant data for Case 6 from~\autoref{lm:oneCellUnbounded} leading to shape (K): partial dual subdivisions to $\Gamma$ and local movement around selected members of the bitangent class, depending on the location of the tangency $P$ with respect to the edge $e$.\label{fig:Case6_1dim}}
  \end{figure}

\item[{Case 6}] type $(v,P)=((6a), (2))$. We assume $P$ lies in the diagonal end of $\Lambda$. By applying the map $\tau_0$ if necessary, we may assume $v^{\vee}$ has vertices $(0,0)$, $(1,1)$ and $(1,0)$ as in the left of~\autoref{fig:Case6_1dim}. We let $e$ be the edge of $\Gamma$ responsible for the tangency at the vertex $P$. We must decide if $P$ is the left- or right-most vertex of $e$ (recall that by our convention  $v_l\prec_d v_r$.) The two options and the local movement at $v$ within $B$ are illustrated in the figure. Since $\dim B=1$, we must have $P=v_l$ and we can move $v$ beyond $\sigma$ along the horizontal edge of $\Gamma$ containing it.

  We let $v_r'$ be the other endpoint of this edge.
  There are three options for the dual triangle $(v_r')^{\vee}$, depicted to the right of the figure. As was argued for Case 3, our assumptions on $B$ ensure  that $v_r\prec_d v_r'$. This implies that the movement along this horizontal edge that started at $v'$ stops at a point in its relative interior, so $B$ has shape (K).\qedhere
\end{description}
\end{proof}

\begin{lemma}\label{lm:oneCellBounded}
Let $B$ be a bounded one-dimensional  bitangent class of $\Gamma$ where every member intersects $\Gamma$ in two connected components.
Then, up to $\Sn{3}$-symmetry, $B$ has  shape  (E), (F) or (G).
\end{lemma}

\begin{proof} Let $\sigma$ be a one-dimensional cell of $B$, pick a point $v$ in its relative interior and let $\Lambda$ be the tropical bitangent line to $\Gamma$ with vertex $v$. We let $P$ and $P'$ be its two tangency points.
  Since $B$ is bounded, \autoref{lm:Bunbounded} ensures that they lie in distinct ends of $\Lambda$. Furthermore, since $v$ is not a vertex of $\Gamma$, $v\in \sigma^{\circ}$ and $\dim B=1$, one of the tangencies (say $P'$) must be non-transverse of types (3a) or (3c). For the same reasons,  the point $P$ must be of tangency type (1).

  We follow the same notation from the proofs of previous lemmas in this section. We let $e$ and $e'$ be the edges of $\Gamma$ containing $P$ and $P'$, respectively, and let $v_l, v_r, v_l',v_r'$ be the endpoints of $e$ and $e'$, with the conventions $v_l\prec_d v_r$ and $v_l'\prec_d v_r'$. By $\Sn{3}$-symmetry, we assume $P$ and $P'$ lie in the  diagonal  and horizontal ends of $\Lambda$, respectively.

  There are three cases to consider,  depending on the tangency type of $P'$ and whether or not $v_r$ is above the horizontal line $L: v+\RR(1,0)$. Each yields a different shape.

  \begin{description}
  \item[Case 1] $P$ has type (3c) and $v_r$ is above $L$. By construction,  $v$ lies in a connected component of $\RR^2\smallsetminus \Gamma$, unbounded in the direction $(0,-1)$. The corresponding dual vertex $(j,0)$ must be an endpoint of $e^{\vee}$. The three options for the direction of $e^{\vee}$ listed in~\autoref{tab:edgeOnEdgeMult2}, the boundedness of $e'$ and the degree of $\Gamma$ combined force $j=2$. This reduces the possibilities for $e^{\vee}$ and $v_r^{\vee}$ to two cases, as we see on the left of~\autoref{fig:EFShapes}.
    
    In turn, $(e')^{\vee}$ lies in the boundary of a different connected component of $\RR^2\smallsetminus \Gamma$ than $v_r$, and the dual triangle $(v_r')^{\vee}$ is unimodular and contains both $(2,0)$ and $(e')^{\vee}$. Thus, the vertices of $(e')^{\vee}$ are $(1,i)$ and $(1,i+1)$ for $i=0,1$ or $2$.

    The non-transverse tangency $P'$ restricts the local movement around $v$ to the horizontal direction. We can move both left and right from $v$ while remaining bitangent. The tangency point $P'$ is fixes throughout, while the second tangency point travels along $e$ towards $v_l$ and $v_r$, respectively. Since $v_r$ lies above the horizontal line $L$, we can move $v$ in the direction $(1,0)$ until the diagonal end of the bitangent line meets $v_r$. This is a type (2) tangency because $v_r\notin L$, so we cannot move beyond this point. The movement is within the same chamber of $\RR^2\smallsetminus \Gamma$ containing $v$, as seen in the center of~\autoref{fig:EFShapes}.
    
    To decide the stopping point when moving from $v$ in the $(-1,0)$-direction, we compare the relative $\preceq_d$-order between $v_l$ and $v_r'$. We claim $v_r'\prec_d v_l$ by the restrictions on $B$. If so, $B=\sigma\subseteq \RR^2\smallsetminus \Gamma$ and so its shape is (E).

    To prove the claim, it suffices to analyze the three possibilities for the triangle $(v_r')^{\vee}$. If $v_l\preceq_d v_r'$, then $i=1$ and $v_r'=_d v_l$. However, in this situation, $v_r'$ and $v_l$ are adjacent in $\Gamma$, and the vertex $v_r'\in B$. If so, the corresponding bitangent line intersects $\Gamma$ in a connected set, contradicting our hypotheses on $B$. Thus, $ v_r'\prec_d v_l$ as we wanted to show.

    \begin{figure}[tb]
    \includegraphics[scale=0.2]{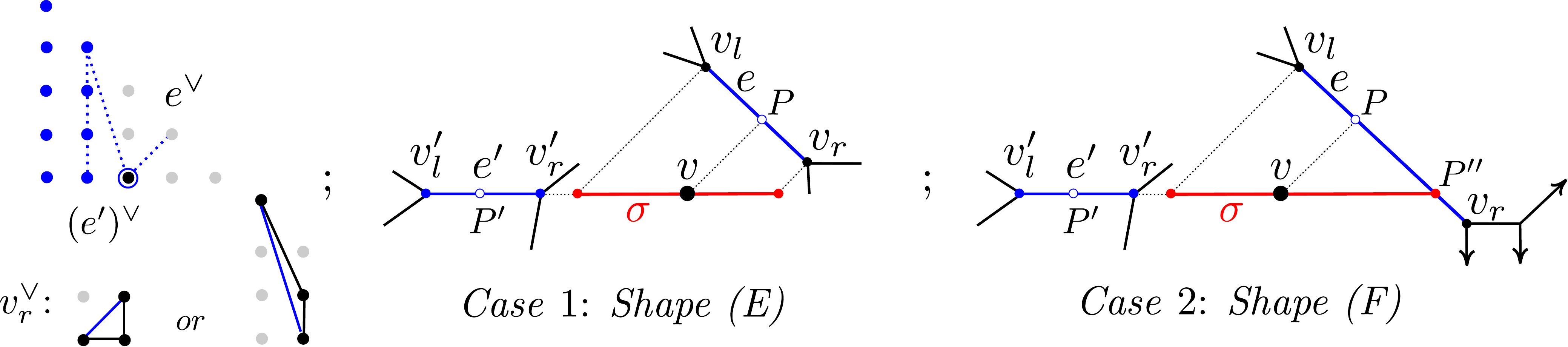}
    \caption{From left to right: partial dual subdivision and bitangent classes with shapes (E) and (F) corresponding to two tangency points of types (3c) and (1) for a bounded one-dimensional class, as in~\autoref{lm:oneCellBounded}.\label{fig:EFShapes}}
    \end{figure}

  \item[Case 2] $P$ has type (3c) and $v_r$ is on or below $L$. We let $P''$ be the intersection point between $L$ and $e$. 
    The same reasoning from Case 1 yields the same partial dual subdivision to $\Gamma$ and the claim $v_r'\prec_d v_l$. Thus, we can move from $v$ in the direction $(1,0)$ until the diagonal end of the new line reaches $v_l$. However, the movement in the direction $(1,0)$ stops at $P''$.
    
    We claim that $P''\neq v_r$,  as seen on the right of~\autoref{fig:EFShapes}. If so, $P''$ becomes a type (4) tangency point and we cannot move past it. This confirms that $B$ has shape (F). To prove our claim, we analyze the possibilities for $v_r^{\vee}$.

    The fact that $v_r$ lies on or below $L$ forces $e'$ to have direction $(-1,1)$ and $v_r^{\vee}$ to be the triangle with vertices $(2,0)$, $(3,0)$ and $(3,1)$. Thus, if $P''=v_r$ we can move beyond $v_r$ along a ray with direction $(1,0)$  while remaining in $B$ because $v_r$ and $P'$ will be tangency points on the horizontal end of a bitangent line. This cannot happen because $B$ is bounded. We conclude that $P''\neq v_r$, as we wanted.

    \begin{figure}
    \includegraphics[scale=0.2]{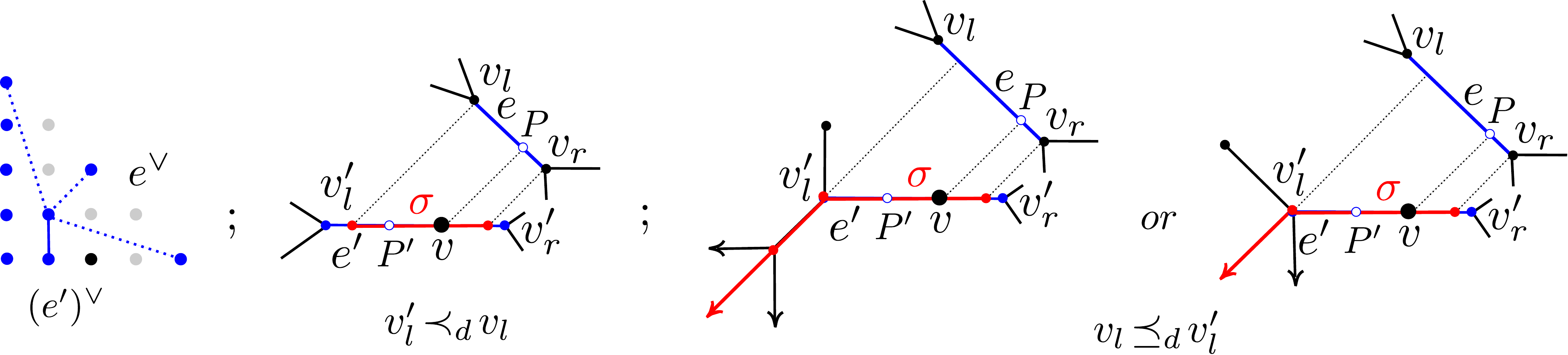}
    \caption{From left to right: partial dual subdivision and bitangent class to $\Gamma$ with a member having tangency points of types (3a) and (1) as in~\autoref{lm:oneCellBounded}; potential local movements depending on the relative $\preceq_d$-order between $v_l$ and $v_l'$.\label{fig:GShape}}
    \end{figure}

  \item[Case 3]  $P$ has type (3a). In this situation, both $e$ and $e'$ are in the boundary of the same connected component of $\RR^2\smallsetminus \Gamma$. In addition $e$ and $e'$ are also in the boundary of connected components of this complement, unbounded in the directions $(1,1)$ and $(0,-1)$, respectively. This forces $e^{\vee}$ and $(e')^{\vee}$ to contain vertices of the form  $(k,4-k)$ and $(j,0)$, respectively, with $k=0,2$ or $4$ and $j=1,2$ or $3$. The restrictions on the directions of $e^{\vee}$ and $(e')^{\vee}$ combined with the fact that these two edges share a vertex, leaves only one option for $j$, namely $j=1$ as we see on the left of~\autoref{fig:GShape}. In turn, this gives three possible locations for $e^{\vee}$ (drawn as blue dotted segments.)

    We can move $v$ horizontally (both left and right) while remaining in $B$. To decide the stopping points, we determine the relative $\preceq_d$-order of the vertices $v_l$ and $v_l'$ respectively, $v_r$ and $v_r'$. We claim that   $v_l' \prec_d v_l$ and $v_r\prec_d v_r'$, so our movement away from $v$ in both directions stops at points in the relative interior of $e$. Thus,  $B$ has shape (G) as seen in the center of~\autoref{fig:GShape}.

To prove the claims, we argue by contradiction. First, assume $v_r'\preceq_d v_r$, so $v_r' \in B$. The convexity of the connected component of $\RR^2 \smallsetminus \Gamma$ dual to $(1,1)$ and the possible realizations of $v_r^{\vee}$ then force $v_r'$ and $v_r$ to be equal. 
If so, the bitangent line to $\Gamma$ with vertex $v_r'$  intersects $\Gamma$ in a connected set, contradicting our assumption on $B$.  Thus, $v_r\prec_d v_r'$.

Second, assume $v_l\preceq_d v_l'$, so $v_l'\in B$. Then, the same convexity argument used above forces the triangle $(v_l')^{\vee}$ to contain one the points $(0,0)$ or $(0,1)$ as a vertex. In both cases, we can move from $v_l'$ along a ray with direction $(-1,-1)$ while remaining in $B$, as the figure shows. This cannot happen since $B$ is bounded. \qedhere
  \end{description}
   \end{proof}

Our final result classifies  the shapes of the remaining zero-dimensional bitangent classes:

\begin{lemma}\label{lm:zeroCell}
Let $B$ be a zero-dimensional  bitangent class for $\Gamma$ whose unique  member intersects $\Gamma$ in two connected components. Then, $B$ has either shape (A) or (B). 
\end{lemma}

\begin{proof} Our assumptions on $B$ imply that the corresponding bitangent line $\Lambda$ has two non-transverse tangency points of type (3) (called $P$ and $P'$), each on a different end of $\Lambda$. Without loss of generality, we assume them to be the horizontal and vertical ends, respectively, and let $e$ and $e'$ be the corresponding bounded edges of $\Gamma$ realizing these two tangencies.  By construction, $P$ is of type (3c), while $P'$ can have types (3a), (3b) or (3c).

  Let $v_l$ and $v_r$ be the left and right vertices of $e$. Analogously, call $v_t'$ and $v_b'$ the top and bottom vertices of $e'$.
  Since $\Lambda\cap \Gamma\subset e\cup e'$, it follows that $v_l$ and $v_b'$ are in the boundary of two connected components of $\RR^2\smallsetminus \Gamma$ unbounded in directions $(-1,0)$ and $(0,-1)$, respectively.   This fact imposes strong restrictions on the dual edges $e^{\vee}$ and $(e')^{\vee}$. Indeed, $e^{\vee}$ and $(e')^{\vee}$ are obtained by joining vertices $(1,i)$ and $(1,i+1)$, respectively, $(j,1)$ and $(j+1,1)$, for $i=0,\ldots, 2$ and $j=0,\ldots,2$.

  There are three possibilities for the tangency type of $P'$. In turn, this depends on whether or not  $B\cap \Gamma =\emptyset$. In the first situation, $B$ has shape (A),  $P'$ has type (3c), and $B$  lies in an unbounded connected component of $\RR^2\smallsetminus \Gamma$ containing $v_r$ and $v_t'$ in its boundary. Thus, such component must be dual to $(2,2)$, and this vertex belongs to  both $v_r^{\vee}$ and $(v_t')^{\vee}$, as we see on~\autoref{fig:NP}.

  In the second case, $P'$ is either of type (3a) or (3b), in particular, $B\cap \Gamma \subseteq e'$. First, assume its type is (3a). If so, then  $P$ lies  in the relative interior of $e'$ and $B$ has shape (B). In turn,  $v_r$  and $e'$ lie in the boundary of the same connected component of $\RR^2\smallsetminus\Gamma$. The information on $(e')^{\vee}$ and $e^{\vee}$ forces this component to be dual to $(2,1)$, so $j=2$. The value of $i$ is unrestricted, but $(2,1)$ must be a vertex of the triangle $v_r^{\vee}$, as~\autoref{fig:NP} shows.

  Finally, assume $P'$ has type (3b), i.e., $P'=v_t'$. We argue this cannot happen by exploiting the partial information on the dual subdivision to $\Gamma$ we have collected so far. Indeed, suppose that $P=v_t'$. Note that the open segment $(v_r,v_t')$ of the horizontal end of $\Gamma$ lies in one connected component of $\RR^2\smallsetminus \Gamma$.   Since $(2,2)\in (v_t')^{\vee}$, an analysis of the three possibilities for the dual triangle $(v_t')^{\vee}$ ensures that $v_r$ and $e'$ lie in the same connected component of $\RR^2\smallsetminus \Gamma$, so $j=2$ as before. This fixes $(v_t')^{\vee}$ and shows that $\Star_{\Gamma}(v_t')$   is a tropical line. This cannot happen by the horizontal alignment of $v_r$ and $v_t'$.
\end{proof}

Our combinatorial classification of bitangent shapes yields the following convexity result for bitangent classes summarized in~\autoref{thm:combclass}:

\begin{corollary}\label{cor:MinConvexSets} The bitangent classes associated to a smooth tropical plane quartic $\Gamma$ are min-convex sets. Out of 41 possible $\Sn{3}$-representatives of bitangent shapes, only eight of them are finitely generated tropical polytopes, namely (A) through (G) and (W).
\end{corollary}

\begin{remark}\label{rm:abstractVsEmbedded}
  By~\cite[Theorem 14 and Proposition 15]{haa.mus.yu:12}, the linear systems of all effective tropical theta characteristics on  the skeleton of $\Gamma$ are finitely generated tropical convex sets. 
  \autoref{cor:MinConvexSets} highlights a fundamental difference between complete linear systems on abstract and embedded tropical curves. 
\end{remark}

The proof of 
{Propositions}~\ref{prop:mult4intersection} and~\ref{prop:oneComponentIntersections}, and 
{Lemmas}~\ref{lm:twoCellUnbounded} through~\ref{lm:zeroCell} yield the following 
combinatorial
consequence which plays a crucial role in the lifting  results discussed in~
{Sections}~\ref{sec:lift-trop-bitang} and~\ref{sec:real-lift-trop}:

\begin{corollary}\label{cor:combClass} The presence of a bitangent shape  for $\Gamma$ partially determines the dual subdivision to $\Gamma$. \autoref{fig:NP} summarizes our findings for chosen $\Sn{3}$-representatives of each shape. Restrictions for non-representative shapes arise  via  the maps $\tau_0$ and $\tau_1$ from~\autoref{tab:S3Action}.
\end{corollary}

\begin{remark}\label{rm:NPColorCoding} Next, we explain the color coding in~\autoref{fig:NP} (which is available on the online version of this article.) Solid black edges must be part of the given subdivision but do not contain any tangency points. Dotted colored edges indicate potential edges, one of which must occur. The red, green and purple ones correspond to horizontal, vertical and diagonal type (3a), (3b) or (3c) tangencies, while blue edges come from type (1) or (2) tangencies. Pink edges are responsible for a combination of a non-proper tangency along a bounded edge, followed by in unbounded cell of $B$ with the same direction.
  \pagebreak

  \thispagestyle{empty}
\begin{landscape}
  \begin{figure}
    \includegraphics[scale=0.48]{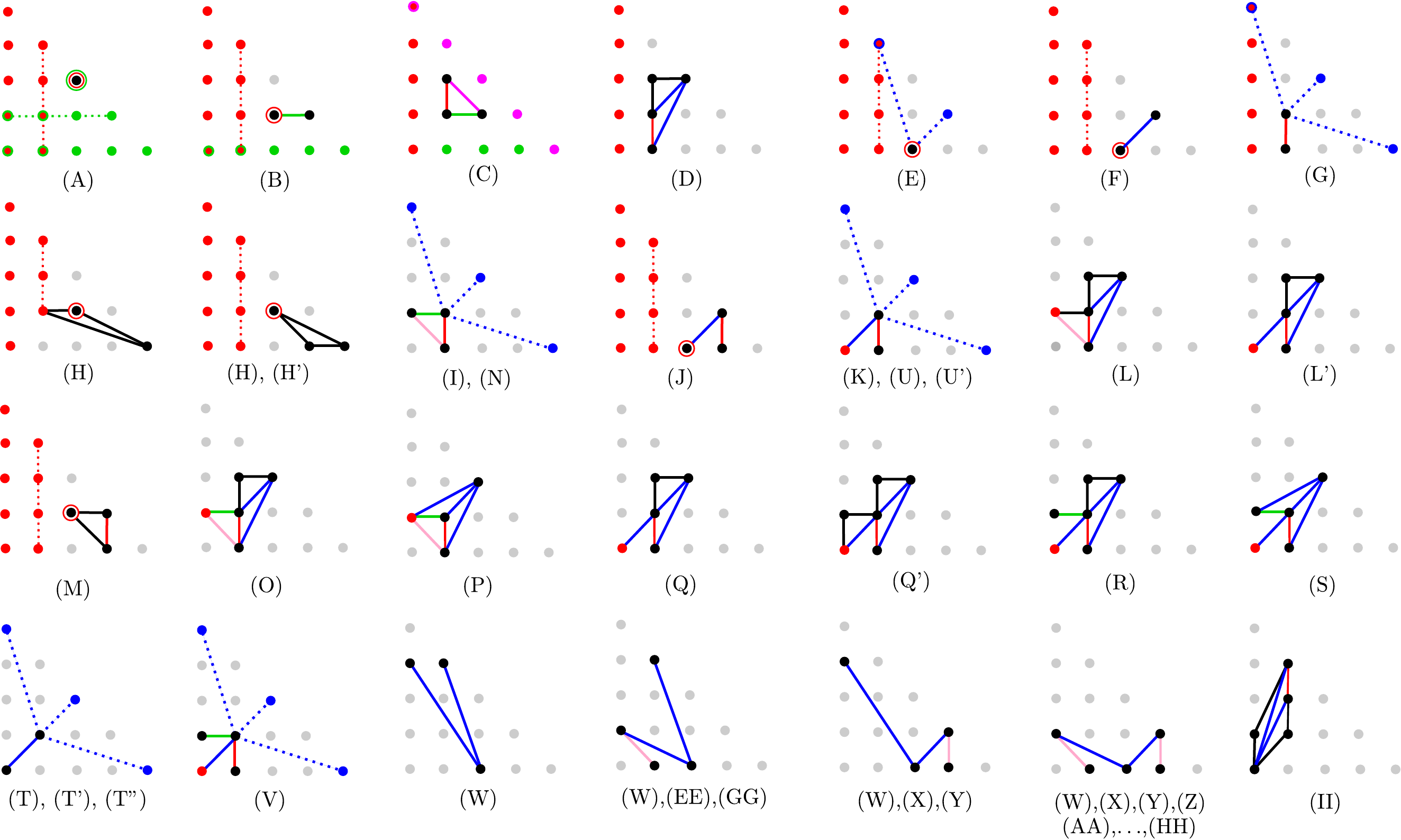}
    \caption{Partial dual subdivisions corresponding to all 41 bitangent classes to $\Gamma$.\label{fig:NP}}
  \end{figure}
  \end{landscape}

  Similarly, black vertices are always present, whereas colored vertices are either endpoints of the corresponding optional dotted edges or they form a triangle with an edge of the same color. Finally, colored circled black dots indicate a connected component of $\RR^2\smallsetminus \Gamma$ that contains the endpoint of a type (3c) tangency (with the same color) in its closure.
\end{remark}

\begin{remark}\label{rem:CNonGeneric} The partial dual subdivision to $\Gamma$ induced by a singleton bitangent shape $\{v\}$ of type (C) as depicted in~\autoref{fig:NP} assumes $\Gamma$ is generic in the following sense. The minimum length among the three edges adjacent to the vertex $v$ with $\Star_v(\Gamma) = \Lambda$ must be attained for exactly one edge. If this were not the case, the points $(0,0)$ and $(4,0)$ should also be colored in green, since they could form a triangle with the edge joining $(1,1)$ and $(2,1)$.
  \end{remark} 

\section{Lifting tropical bitangents}\label{sec:lift-trop-bitang}

The techniques developed by Len and the second author in~\cite{len.mar:20} produce concrete formulas for lifting tropical bitangents over the field $\KK$. We are interested in extending these results to \emph{real} lifts, that is to compute the number of bitangent triples $(\ell,p,p')$ associated to a fixed tropical bitangent where $\ell$ is defined over $\KKr$. In this section, we determine the local lifting multiplicities over $\KKr$ as in~\cite{len.mar:20}. These refined formulas will be used  in~\autoref{sec:real-lift-trop}  to prove~\autoref{thm:reallifting}.

We start by reviewing the lifting multiplicities over $\KK$ for each tangency type. The combinatorial classification developed in~\autoref{sec:comb-class-bitang} allows us determine all possible combinations of local tangencies that can occur within each class by imposing restrictions on the dual subdivision to $\Gamma$ where a given shape can arise.  Most notably, the proofs of~
 {Propositions}~\ref{prop:mult4intersection},~\ref{prop:oneComponentIntersections} and~\ref{prop:shapesfortwoconnectedcomp} combined with~\autoref{tab:liftingMultiplicities} provide extra information regarding which points of $B$ admit lifting to a $\KK$-bitangent triple $(\ell, p, p')$. We record this data by assigning this lifting multiplicity as the \emph{weight} of the point in $B$. Its value can be $0,1,2$ or $4$.

 Here is the precise statement that justifies the weight assignment in~\autoref{fig:2dCells} under the genericity assumptions from~\autoref{rem:genericity}.

 \begin{theorem}\label{thm:combinationsOfTangencies} There are 24 combinations of distinct unordered pairs of local tangencies that arise from tropical bitangents to generic smooth tropical plane quartics (see~\autoref{tab:possibleTangencyPairs}.) Furthermore, only 14 of them lift to  bitangent triples defined over $\KK$.

   In turn, five out of the 24 pairs admit a multiplicity four tangency. Their types are (1), (3b), (4), (5b) and (6b). Assuming the plane quartic  has no hyperflexes, only the last two lift to a $\KK$-bitangent triple. They are members of a  bitangent class of shape (II).
\end{theorem}

  \begin{table}[tb]
  \begin{tabular}{|c|c|c|c|c|c|c|c|c|}
    \hline   types & (1) & (2) & (3a) & (3b) & (3c) & (4) & (5a) & (6a)  \\
    \hline (1) & $\checkmark$ & $\checkmark$ & $\checkmark$ & $\checkmark$ & $\checkmark$ & $\checkmark$ & $\checkmark$ & $\checkmark$  \\
    \hline (2) & &   {$\boxed{\checkmark}$} & {$\boxed{\checkmark}$} & $\checkmark$ & {$\boxed{\checkmark}$} & {$\boxed{\checkmark}$} & {$\boxed{\checkmark}$} & {$\boxed{\checkmark}$} \\
    \hline (3a) &  & & & & {$\boxed{\checkmark}$} & & &   \\ 
    \hline (3b) & & & & {$\boxed{\checkmark}$} & $\checkmark$ & & &    \\
    \hline (3c) & & & & & {$\boxed{\checkmark}$} & {$\boxed{\checkmark}$}& {$\boxed{\checkmark}$}& {$\boxed{\checkmark}$}    \\
    \hline
  \end{tabular}
      \caption{Combinations of unordered pairs of distinct local tangencies that arise from tropical bitangents to generic smooth tropical plane quartics $\Gamma$. The boxed ones are the only ones arising from classical bitangent triples. 
        \label{tab:possibleTangencyPairs} }
\end{table}

The proof of the multiplicity four claim in the theorem is given in~\autoref{sec:appendix1}. The statement regarding the multiplicity two tangencies follows from~\autoref{tab:liftingMultiplicities} and~\autoref{thm:liftingProducts}. As an example, we show why the unordered pair of tangency types ((2),(3b)) is never realizable over $\KK$ under our genericity assumptions. The remaining non-liftable combinations can be argued by the same methods.
Indeed,~\autoref{cor:combClass} and~\autoref{rm:NPColorCoding} imply that a pair of  (3b) and (2) local tangencies only occurs at certain vertices of shapes (X), (Z), (AA), and (FF) through (HH). The corresponding bitangent lines $\Lambda$ have two properties: first,   the set $\Gamma \cap \Lambda$ is disconnected and, second, the tangency points lie in the relative interior of the same end of $\Lambda$. Thus, by~\autoref{thm:liftingProducts}~
{(ii)} we conclude that $\Lambda$ does not lift to a $\KK$-bitangent triple.
\smallskip

In order to address liftings over $\KKr$, we start by reviewing the lifting formulas over $\KK$ from~\cite{len.mar:20}. Recall that we assume our bitangent lines $\ell$ are given by an equation of the form $y+ m + nx =0$, as in~\eqref{eq:line}.
By~\autoref{lm:multivariateHensel}, each bitangent lift $(\ell,p,p')$  is uniquely determined by the initial forms of $m,n$, and the coordinates of $p$ and $p'$. 
In turn,  the formulas for the initial forms of the coefficients $m$ and $n$ of $\ell$ are either rational functions of the local equations of $q$ at the tangency points (giving local lifting multiplicity one), or they involve square roots in the coefficients of these local equations (and so the lifting multiplicity is two.) In the first case, the initial forms will automatically be real because~$q\in \KKr[x,y]$ and the resulting bitangent lift will be defined over  $\KKr$.

Thus, in order to determine lifting multiplicities over $\KKr$, we need only study combinations of tropical bitangents corresponding to  vertices of bitangent classes with weight 2 or 4. In this situation, the local lifting formulas will involve square roots of the local coefficients of $q$. 
This allows us to restrict our attention to local liftings formulas for five tangency types: (3a), (3c), (5a), (4) and (6a). Each of these cases will yield either 0 or 2 lifts over $\KKr$. Applying these results for the higher weight vertices of each shape will yield a total of zero or four real lifts for each bitangent class, as stated in~\autoref{thm:reallifting}.
Type (3b) tangencies that  lift over $\KK$ only occur for shape (C). We discuss them in~\autoref{pr:shapeC}.

As discussed in~\autoref{sec:lift-bitang} the local system defined  by a multiplicity two tangency point determines either $\overline{m}$, $\bar{n}$ or $\overline{m}/\bar{n}$ depending on the end of $\Lambda$ containing the tangency. 
Our formulas for certifying the realness of one of these quantities will depend on the positivity of the product of the initial forms of certain coefficients $a_{ij}$ of $q$ and other coefficients of $\ell$. We must consider eight different combinations of tangency types and ends carrying the given tangency. We list them in~\autoref{tab:localLifts3a3c}.  \autoref{fig:LocalNP} shows the cells in the dual subdivision to $\Gamma$ that are involved in each case, using the same color-coding as in~\autoref{fig:NP}. These cells indicate the specific coefficients  $a_{ij}$'s that feature in each formula.

Throughout, we write $s_{ij}$ for the sign of the initial form $\overline{a_{ij}}\in \RR$.
Our next results explain all the formulas in~\autoref{tab:localLifts3a3c}, starting with types (3a) and (3c):

\begin{proposition}\label{pr:realLifts3a3c} Assume $V(q)$ is defined over $\KKr$. Let $\Lambda$ be a tropical bitangent to $\Gamma$ with a local tangency $P$ of type (3a) or (3c) along a horizontal (respectively, vertical) edge $e$ of $\Gamma$. Then, there are either zero or two real solutions for $m$ (respectively, $m/n$) for the local equations in $(m,n,p)$ determined by the tropical tangency $P:=\Trop\,p$.

  Necessary and sufficient conditions for the existence of a real solution are determined by the positivity of specific polynomials in the signs of $\bar{n}$ and  the four coefficients $\overline{a_{ij}}$ from~\eqref{eq:q} appearing in the dual triangles to the vertices of $e$ (seen in~\autoref{fig:LocalNP}).  The precise formulas are given in~\autoref{tab:localLifts3a3c}.
  Furthermore, whenever real solutions for $m$ (respectively, $m/n$) exist, then the sign of its initial form  equals $s_{uv}s_{u,v+1}$ (respectively, $s_{uv}s_{u+1,v}$). Here, $(u,v)$ is the left-most (respectively, bottom) vertex of $e^{\vee}$.
\end{proposition}

\begin{figure}
  \includegraphics[scale=0.45]{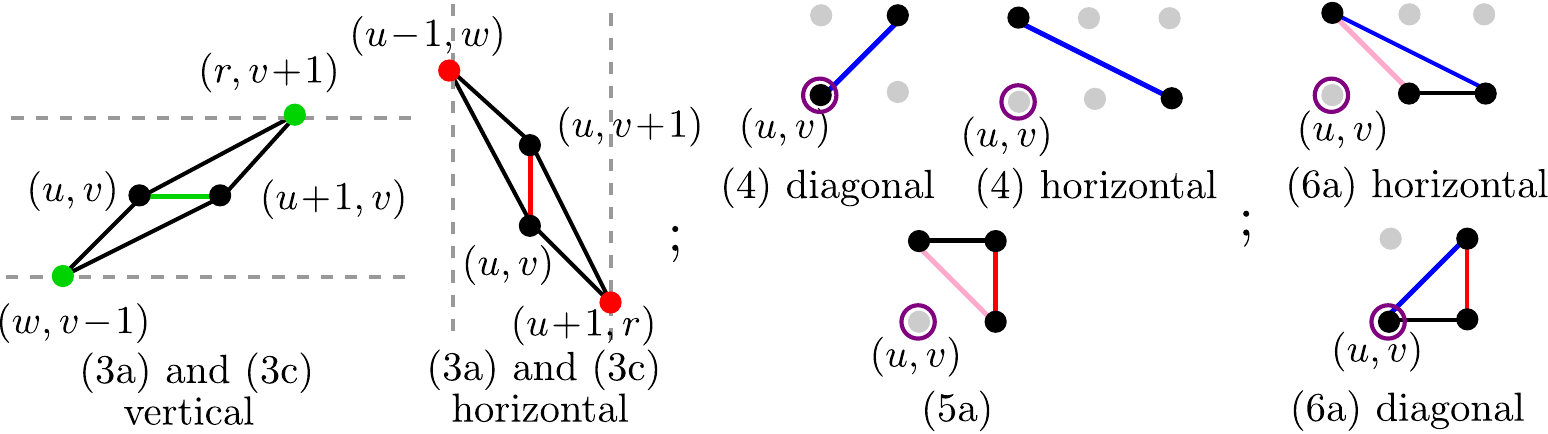}
  \caption{Relevant cells in the Newton subdivision of $q$ and their position in $\ZZ^2$, for each tangency type and end of $\Lambda$. For cases  (4), (5a) and (6a), this data can be deduced from a single vertex, labeled by $(u,v)$.\label{fig:LocalNP}}
\end{figure}

\begin{proof}   Without loss of generality we assume $q\in R[x,y]\smallsetminus \mathfrak{M}\,R[x,y]$. 
  The statement and formulas when the edge $e$ is vertical follow from those for the horizontal case by using the map $\tau_0$ from~\autoref{tab:S3Action}. Thus, it suffices to  prove the statement  when $e$ is  horizontal.

 As~\autoref{fig:LocalNP} indicates, the dual edge $e^{\vee}$ in the Newton subdivision of $q$ has vertices $(u,v)$ and $(u,v+1)$. The are connected to vertices $(u-1,w)$ and $(u+1,r)$. We argue for each local tangency type separately. Following  the modification and re-embedding techniques from~\cite{len.mar:20}, we write $m=m_1+ m_2$ with $\val(m) = \val(m_1)<\val(m_2)$. We set $m_1 = a_{u,v}/a_{u,v+1} \in \KKr$ as in~\cite[Proposition 3.7]{len.mar:20} and determine two solutions for the  initial forms of $m_2$ using~\cite[Lemma 3.9]{len.mar:20}. Each one will yield a unique $m_2$ by~\autoref{lm:multivariateHensel}.

 \textbf{Case (3c):} After translation if necessary,  we set $P=(0,0)$. Since $P$ is the midpoint of $e$,  the vertices of $e$ equal $(-\lambda,0)$ and $(\lambda,0)$ for some $\lambda>0$. Furthermore, we have $\val(a_{u,v})=\val(a_{u,v+1})=0$, $\val(a_{u-1,w})=\val(a_{u+1,r})=\lambda$ and all other $a_{ij} \in \mathfrak{M}$. In addition, since $(u,j)$ for $j\neq v, v+1$ is not in the dual cell to the vertex $(-\lambda, 0)$, it follows that $\val(a_{u,j})>\lambda$ for $j\neq v, v+1$.

 Thus, our proposed solutions $m$ will have the form $m=m_1 +m_2$ with $\val(m)=\val(m_1)=0$ and $\val(m_2) = \lambda$. In particular, $\overline{m}=\overline{m_1}$, so the sign of $\overline{m}$ agrees with that of $\overline{m_1}$, namely $s_{uv}s_{u,v+1}$. In what follows, we explain how to compute the parameter $m_2$.

We start by re-embedding  
$V(q)$ in $(\KK^*)^3$ via the ideal $I=\langle q, z-y-m_1\rangle$. We let
$\tilde{q}(x,z-m_1) = \sum_{i,j} \tilde{a}_{ij} x^i z^j \in \KKr[x,z]$ be the polynomial defining the projection of $V(I)$ onto the $xz$-plane.
Our  choice of $m_1$ fixes the  valuations of  four coefficients in $\tilde{q}$:
\begin{equation*}\label{eq:case3a}\val(\tilde{a}_{u-1,0}) = \val(\tilde{a}_{u+1,0}) = \lambda, \quad \val(\tilde{a}_{u,1})= 0 \text{ and } \val(\tilde{a}_{u,0})>\lambda.
\end{equation*}
Furthermore, we have $\val(\tilde{a}_{u,1} - (-m_1)^{v}\,a_{u,v+1}) > 0$, whereas 
\begin{equation}\label{eq:coeffs_case3a}
     \val(\tilde{a}_{u-1,0} - a_{u-1,w} (-m_1)^{w}) > \lambda \quad \text{ and } \quad
     \val(\tilde{a}_{u+1,0} - a_{u+1,r} (-m_1)^{r})  > \lambda.
\end{equation}

As a consequence, the Newton subdivision of $\tilde{q}$ satisfies two combinatorial properties:
\begin{enumerate}[(i)]
\item it   contains the triangle with vertices $(u-1,0)$, $(u+1,0)$ and $(u,1)$;
  \item the lattice point $(u,0)$  is not a vertex of this subdivision.
\end{enumerate}
In particular, the tangency point $p$ viewed in $V(I)$ has tropicalization $\Trop(p)=(0,0,-\lambda)$.

In turn,~\cite[Lemma 3.9]{len.mar:20} implies that $\val(m_2) =\lambda$ and that there are two solutions for $m_2$, each determined by its initial form. Their values are
\begin{equation}\label{eq:case3aM}
  \overline{m_2} = \pm\frac{2}{\overline{\tilde{a}_{u,1}}}\,{\sqrt{\overline{\tilde{a}_{u-1,0}}\; \overline{\tilde{a}_{u+1,0}}}}.
\end{equation}
 We conclude that $m_2\in \KKr$ if, and only if, the radicand is positive. The expressions in~\eqref{eq:coeffs_case3a} yield the formula for determining the positivity of this radicand as listed in~\autoref{tab:possibleTangencyPairs}.

\begin{table}[tb]
  \begin{tabular}{|c|c|c|c|}
    \hline   type & condition for real solutions  &  coefficient & end of $\Lambda$   \\
\hline     \multirow{2}{*}{(3a)}& $(-1)^{w+v+1} (s_{uv} s_{u,v+1})^{w+v} s_{u-1,w}\, s_{u,v+1} \operatorname{sign}(\bar{n})>0$ & $m$ & horizontal  \\
\cline{2-4}&     $(-1)^{w+u+1} (s_{uv} s_{u+1,v})^{w+u} s_{w,v-1} \,s_{u+1,v} \operatorname{sign}(\bar{n})>0$  & $m/n$  & vertical  \\
    \hline
    \multirow{2}{*}{(3c)} & $(-1)^{r+w} (s_{uv} s_{u,v+1})^{r+w} s_{u+1,r} \,s_{u-1,w}>0$   & $m$ & horizontal \\
\cline{2-4}     & $(-1)^{r+w} (s_{uv} s_{u+1,v})^{r+w} s_{r,v+1} \,s_{w,v-1}>0$  & $m/n$  & vertical\\
\hline
\multirow{2}{*}{(4), (6a)} & $-\operatorname{sign}(\bar{n}) s_{uv} s_{u+1,v+1} > 0$ & $m$ & diagonal\\
\cline{2-4} & $-\operatorname{sign}(\overline{m}) s_{u,v+1}\,s_{u+2,v}>0$ & $n$ & horizontal
\\
\hline
\multirow{2}{*}{(5a)} & $\operatorname{sign}(\bar{n}) s_{u+1,v} s_{u,v+1} > 0$ & $m$ & diagonal\\
\cline{2-4} & $\operatorname{sign}(\overline{m}) s_{u+1,v+1} s_{u+1,v} > 0$ & $n$ & horizontal\\
    \hline
  \end{tabular}
  \caption{Real local lifting conditions for $\overline{m}$, $\bar{n}$ and $\overline{m}/\bar{n}$ imposed by tangency types of lifting multiplicity two and ends of  $\Lambda$ producing the tangencies. For types (4) and (6a), the second tangency point lies in the end of $\Lambda$ producing the first tangency. For type (5a), we indicate the end of $\Lambda$ where the second tangency occurs. Here, $s_{ij}$ is the sign of the initial coefficient $\overline{a_{ij}}$ of the monomial $x^iy^j$ in $q$. The subindices correspond to vertices depicted in~\autoref{fig:LocalNP}.\label{tab:localLifts3a3c}}
\end{table}

 \textbf{Case (3a):} After translation if necessary, we may assume that the leftmost vertex of $e$ is $(0,0)$ and fix $(N,0)$ as the vertex of $\Lambda$ with $N>0$. For any bitangent triple $(\ell,p,p')$ lifting $\Lambda$ we have $\val(m)=0$,  $\val(n)=N$ and $\Trop\,p = (N/2,0)$.
 Our assumption on $q$ ensures that
   $\val(a_{u,v})=\val(a_{u,v+1})=\val(a_{u-1,w})=0$ and all other $a_{ij}$ lie in $\mathfrak{M}$. In addition, since the local (3a) tangency lies in the relative interior of the edge $e$, we have $\val(a_{u+1,r}) > N$. Since any point $(u+1,j)$ with $j\neq r$ lies outside the triangle dual to the right-most vertex of $e$ we conclude that $\val(a_{u+1,j})>\val(a_{u+1,r})$ for all $j\neq r$.

 By~\autoref{thm:liftingProducts}~
 {(ii)}, the second tangency point between $\Lambda$ and $\Gamma$ must be on a non-horizontal end of $\Lambda$. Upon acting by $\Sn{3}$, we may assume it is the diagonal one. In particular, the corresponding local equations will fix the input values for $n$, thus explaining  the appearance of this parameter in the formulas for the (3a) entries of~\autoref{tab:localLifts3a3c}.

 We modify slightly the strategy for (3c) tangencies,  picking  $m_1 = a_{u,v}/a_{u,v+1} +m_1' \in \KKr$ for some suitable $m_1'\in \mathfrak{M}$ to be determined later. This choice of $m_1$  satisfies $\overline{m_1} = \overline{a_{u,v}}/\overline{a_{u,v+1}}$ and it allows us to  re-embed $V(q)$ in $(\KK^*)^3$ via the ideal $I=\langle q, z-y-m_1-n\,x\rangle$.

 We study $\Trop\, V(I)$ through its projection to the $xz$-plane, i.e., by determining relevant cells in the Newton subdivision of  $\tilde{q}(x,z)=q(x,z-m_1-n\, x)$. Our choice of $m_1'$ will ensure that the Newton subdivision of $\tilde{q}$ satisfies the same two properties as it did for type (3c). In turn, this will give $\Trop\,p = (N/2,0,-N/2)$. However, in this new setting, the parameter $n$ contributes to the relevant four coefficients of $\tilde{q}$, namely $\tilde{a}_{u-1,0}$, $\tilde{a}_{u,1}$, $\tilde{a}_{u+1,0}$ and $\tilde{a}_{u,0}$.

 Our restrictions on the valuations of various $a_{i,j}$'s given above allow us to determine the valuations of the first three coefficients by ``cone feeding'' using the monomials $x^{u-1}y^w$, $x^{u}y^v$, $x^{u}y^{v+1}$ and $x^{u+1}y^r$ in $q$ and the substitution $y=z-m_1-nx$. A simple arithmetic computation reveals that $\val(\tilde{a}_{u-1,0}) = \val(\tilde{a}_{u,1})=0$ and    $\val(\tilde{a}_{u+1,0})= N$. Furthermore,
 \begin{equation}\label{eq:case3c}
   \begin{split}
     \val(\tilde{a}_{u-1,0}- a_{u-1,w} (-m_1)^w)>0, \quad \val(\tilde{a}_{u,1}-a_{u,v+1}\,(-m_1)^v)>0\quad \text{ and } \quad \\\val(\tilde{a}_{u+1,0} - a_{u,v+1} (-m_1)^v\,(-n))>N.
   \end{split}
\end{equation}

 Notice that the parameter $m_1'$ is not needed to determine the initial forms of $\tilde{a}_{u-1,0}, \tilde{a}_{u,1}$ and $\tilde{a}_{u+1,0}$ by the valuation constraints we imposed. However, it will be needed if we want to ensure $(u,0)$ is not a vertex of the Newton subdivision of $\tilde{q}$. To this end, it suffices to find $m_1'\in \mathfrak{M}$ for which $\val(\tilde{a}_{u,0})\geq N$. We choose this particular bound because it is simpler to attain than the weaker sufficient condition $\val(\tilde{a}_{u,0})> N/2$.

By construction, the only monomials from $q$ that can contribute terms to $\tilde{a}_{u,0}$ with valuation strictly smaller than $N$ are those of the form $x^{u}y^j$ with $j$ arbitrary. This expression is the following univariate polynomial in $R[m_1']$:
\begin{equation}\label{eq:univpoly3a}
f(m_1'):=\sum_{j} a_{u,j} (-1)^j (\frac{a_{u,v}}{a_{u,v+1}} + m_1')^j.
\end{equation}
By~\autoref{lem:chooseTail3a} below and the genericity constraints on $q$, we know that $f(m_1')=0$ has a solution in $\mathfrak{M}$. This fact then yields the desired condition $val(\tilde{a}_{u,0})\geq N$.

Once the parameter $m_1'$ is fixed,  we apply~\cite[Lemma 3.9]{len.mar:20} again to conclude that the values for $\overline{m_2}$ agree with those in~\eqref{eq:case3aM} since the triangle dual to the tropical tangency $(N/2, -N/2)$ is the same as the one for (3c) tangencies. When this radicand is positive we will get (two) solutions $m_2\in \KKr$. The expressions in~\eqref{eq:case3c} yield the desired formula in~\autoref{tab:localLifts3a3c}. Since the map $\tau_0$ sends $n$ to $1/n$, the role of $\operatorname{sign}(\overline{n})$ does not change when moving from a horizontal (3a) tangency to a vertical one.
\end{proof}

\begin{lemma}~\label{lem:chooseTail3a} The univariate polynomial $f(m_1')\in R[m_1']$ from~\eqref{eq:univpoly3a} has a root in $\mathfrak{M}$ if its constant coefficient is non-zero.
\end{lemma}
\begin{proof}
  By the Fundamental Theorem of Tropical Algebraic Geometry (see~\cite[Theorem 3.2.3]{MSBook}), it is enough to show that the tropical hypersurface defined by the tropical polynomial $\trop(f) \in \RR[m_1']$ has a point in $\RR_{>0}$.

  Simple algebraic manipulations of~\eqref{eq:univpoly3a} reveal that $f(m_1') =\sum_{k} b_{k} (m_1')^k$ with 
  \begin{equation*}\label{eq:univpoly3aExpanded}
    \begin{aligned}
      b_0:=& \sum_{j\neq v, v+1} (-1)^j a_{u,j} \Big(\frac{a_{u,v}}{a_{u,v+1}}\Big )^j, \quad b_k:=\sum_{j\geq k} (-1)^j a_{u,j} \binom{j}{k} \Big(\frac{a_{u,v}}{a_{u,v+1}}\Big )^{j-k} \text{ if }\,k>v+1,  \text{ and }\\
      b_k:= & (-1)^{v+1} a_{u,v+1} \binom{v}{k-1} \Big(\frac{a_{u,v}}{a_{u,v+1}}\Big )^{v-k+1}\!\!\! +\!\! \sum_{\substack{j\geq k\\ j\neq v, v+1}} \!\!(-1)^k a_{u,j}  \binom{j}{k} \Big(\frac{a_{u,v}}{a_{u,v+1}}\Big )^{j-k} \text{ otherwise}.
    \end{aligned}
      \end{equation*}
  In particular, we conclude that $\val(b_k)=0$ for $1\leq k \leq v+1$ whereas $\val(b_k)>0$ for all other values of $k$. Thus, the tropicalization of $f$ in the variable $M_1' = -\val(m_1')$ becomes
  \[\trop(f)(M_1'):= \max_{k\geq 0}\{-\val(b_{k}) + k M_1' 
  \}.\]

  An easy computation reveals that $M_1' = -\val(b_0)<0$ achieves this maximum twice, namely at $k=0,1$. Thus, $f(m_1') = 0$ has a solution  in $\KK$ with $\val(m_1') = \val(b_0)>0$.   By construction, its initial form  is $\overline{m_1'} = -\overline{b_0}/\overline{b_1}$. Since $\overline{f'}(\overline{m_1'}) = \overline{b_1}\neq 0$, ~\autoref{lm:multivariateHensel} ensures that $m_1'$ is unique and, furthermore, lies in $\KKr$.
\end{proof}

   As seen in~\autoref{tab:liftingMultiplicities},   lifting formulas for type (4) and (6) tangencies depend on the second tangency point of $\Lambda$. In particular, having lifting multiplicity two restricts the direction $v$ of the bounded edge of $\Gamma$ responsible for the tangency.
   In particular, since $|\det(u,v)|=2$  we know that the end of $\Lambda$ (with direction $v$) containing the second tangency point is also responsible for the first tangency (associated to  $u$). Such situations manifest themselves at vertices of shapes  (H), (H'), (J), (K), (L'), (Q), (Q'), (T), (T'), (T''), (U) or (U').
   
 Next, we discuss real lifts to these local tangencies of types (4) and (6a):

 \begin{proposition}\label{pr:realLift46a} Assume $V(q)$ is defined over $\KKr$ and that the vertex of $\Lambda$ produces a local tangency of type (4) or (6a)  along the diagonal (respectively, horizontal) end of $\Lambda$. Assume its local lifting multiplicity is two. Then, the local system at this tangency determines two possible solutions for the coefficient $m$ (respectively, $n$) over $\KK$. These solutions lie in $\KKr$ if, and only if, an explicit monomial in the signs of $\bar{a}_{ij}$ and $\bar{n}$ (respectively, $\overline{m}$) is positive. The precise formula is given in~\autoref{tab:localLifts3a3c}.
  \end{proposition}

\begin{proof} We prove the statement for the diagonal end of $\Lambda$. The result for the horizontal one will be deduced by applying the map $\tau_1$ from~\autoref{tab:S3Action}.

  By construction, the second tangency point must be of type (2) or (3c) and will lie in the diagonal of $\Lambda$. In both situations, this data determines the value of $n$.  We assume $\bar{n}$ to be in $\KKr$, since otherwise $\ell$ will not be defined over $\KKr$.
  As usual, we place the vertex of $\Lambda$ at $(0,0)$ and take $q\in R[x,y]\smallsetminus \mathfrak{M} R[x,y]$. In particular, $m,n\in R\smallsetminus \mathfrak{M}$ for any lift.

We compute explicit solutions for $m$ and the tangency point $p=(x,y)$ with $\Trop\,p = (0,0)$ from their initial forms using   the techniques from~\cite[Proposition 3.10]{len.mar:20}. In  the notation in~\autoref{fig:LocalNP},  the system in $(\overline{m},\bar{x},\bar{y})$ defined by the initial forms of $q$, $\ell$ and the Wronskian $W = W(q,\ell,\bar{x},\bar{y})$ becomes $\overline{q}=\overline{\ell}=\overline{W}=0$, where $\overline{\ell} :=  \bar{y} + \overline{m} + \bar{n}\, \bar{x}$, 
  \begin{equation*}\label{eq:tangencies46a}
        \overline{q} :=  \bar{x}^u\bar{y}^v(\overline{a_{u,v}} + \overline{a_{u+1,v+1}} \,\bar{x}\,\bar{y} + \delta \;\overline{a_{u+1,v}}\, \bar{x}) \quad \text{ and }\quad
            \overline{W} := \det(Jac(\overline{q}, \overline{\ell}; \bar{x},\bar{y})).
    \end{equation*}
  We set $\delta=0$ if $(0,0)$ is a type (4)  tangency and $\delta=1$, if it has type (6a).

  The monomial factor in $\overline{q}$ can be neglected when determining $\overline{W}$. Thus,
  \begin{equation*}\label{eq:W4_6a}
    \overline{W} = \bar{n}\, \overline{a_{u+1,v+1} }\,\bar{x} - (\overline{a_{u+1,v+1}}\,\bar{y} + \delta\, \overline{a_{u+1,v}}) = 0.
  \end{equation*}
  Solving for $\bar{x}$ in $\overline{q}$ and substituting this expression for $\overline{W}$ yields the quadratic expression
  \begin{equation*}
    \big (\overline{a_{u+1,v+1}}\,\bar{y} + \delta \,\overline{a_{u+1,v}}\big ) ^2 +     \overline{a_{u,v}}\, \overline{a_{u+1,v+1}} \, \bar{n}  = 0.
  \end{equation*}
  Its solutions are $\bar{y} =\frac{1}{\overline{a_{u+1,v+1}}}\big(-\delta\; \overline{a_{u+1,v}} \pm \sqrt{-\overline{a_{u,v}}\; \overline{a_{u+1,v+1}}\ \bar{n}}\big )$. Substituting these expressions in  $\overline{\ell}$  determines a unique solution in $\overline{m}$ for each one. This solution is real if, and only if, $-\overline{a_{u,v}}\; \overline{a_{u+1,v+1}}\ \bar{n}>0$. Computing this sign yields the formula in~\autoref{tab:localLifts3a3c}.

  A simple calculation shows that the initial form of the Jacobian of the local system  with respect to $(m,x,y)$ does not vanish at our solutions $(\overline{m},\bar{x},\bar{y})$ since it equals $-2\,\bar{x}^{2u+1}\, \bar{y}^{2v} \, \bar{n} \,\overline{a_{u+1,v+1}}^2$.~\autoref{lm:multivariateHensel}   guarantees a unique solution $(m,x,y)$ for $(q,\ell,W)$ for each  $(\overline{m},\bar{y},\bar{x})$. Thus, the condition to ensure local solutions over $\RR$ yields the  same number of solutions over $\KKr$.
\end{proof}

\begin{remark} \label{rem:type6a} \autoref{pr:realLift46a} is enough to address almost all liftings where one of the tangency points is of type (6a). The only missing case corresponds to the top-left vertex in shape (T''), since its dual cell  in the Newton subdivision of $q$ is the triangle with vertices $(0,0)$, $(0,1)$ and $(1,1)$. The formula to guarantee two real solutions in $m/n$ becomes $-\operatorname{sign}(\bar{n}) s_{00}s_{11} >0$. We derive it from the (diagonal) condition in~\autoref{tab:localLifts3a3c} by applying the map $\tau_0$ and noticing that $(m,n)$ becomes $(m/n, 1/n)$ under this transformation. 
\end{remark}

\begin{remark}\label{rem:type5a} Local tangencies of type (5a) that are liftable over $\KK$ occur for three representative shapes: (I), (L) and (M). In the first two cases, the second tangency point  has type (2) and lies on  the diagonal of $\Lambda$. For the latter, the second tangency is of type (3c) along a horizontal edge of $\Gamma$.

  The local lifting condition over $\KKr$ for the diagonal (5a) tangency is given in~\autoref{tab:localLifts3a3c} and can be obtained following the same arguments as in the proof of~\autoref{pr:realLift46a}.
  The condition for the horizontal (5a) tangency is derived from the diagonal one by applying the map $\tau_1$. In this case, the roles of $m$ and $n$ are switched and the relevant signs are those associated to the vertices $(u+1,v)$ and $(u+1,v+1)$ in~\autoref{fig:LocalNP}. 
\end{remark}

Almost all local conditions listed in~\autoref{tab:localLifts3a3c} for determining solutions over $\KKr$ for one coefficient of $\ell$ require us to know the sign of the initial form of the second coefficient of $\ell$. This data is obtained from the second tangency of $\Lambda$, whose tangency type can be read off from~\autoref{tab:possibleTangencyPairs}. 
In particular, when this point has type (2) and lies in the diagonal or vertical ends of $\Lambda$, then it uniquely determines $n$ or $m/n$ in $\KKr$, respectively. Our last result in this section computes the sign of these expressions in terms of the coefficients of $q \in \KKr[x,y]$:

\begin{lemma}\label{lm:tangency2} Assume that $\Lambda$ has a type (2) tangency point  along its diagonal (respectively, horizontal or vertical) end. Let $e$ be the edge of $\Gamma$ producing this multiplicity two tangency.  Then, the sign of the unique solution $n$ (respectively, $m$ or $m/n$) to the local tangency equations equals $-s_{v_l}s_{v_r}$, where $v_l$ and $v_r$ are the vertices of the cell $e^{\vee}$ in the Newton subdivision of $q$.
\end{lemma}
\begin{proof}  We prove the analogous statement when the type (2) tangency point $P$ lies in the horizontal end of $\Lambda$, since the explicit local solutions for $m$ computed in~\cite[Proposition 4.6]{len.mar:20} require this setting. Since the local lifting multiplicity of $P$ is one by~\autoref{tab:liftingMultiplicities}, we know that $m\in \KKr$.  We show that the sign of $m$ becomes $-s_{v_l}s_{v_r}$ if $e^{\vee}$ joins $v_l$ and $v_r$. The statement for the diagonal and vertical ends follow by applying  $\tau_0$ and $\tau_1$ since these turn $m$ into $m/n$ and $n$, respectively.

It remains to prove the formula for $\operatorname{sign}(m)$.  Up to translation, the cell dual to $P$ is the triangle with vertices $(0,0), (u,v)$ and $(r,w)$, with $v_l-v_r = (u-r,v-w)$. Up to a monomial shift, we have $\overline{q}:=\bar{a}+\bar{b}\,\bar{x}^u\,\bar{y}^v+\bar{c}\,\bar{x}^r\,\bar{y}^w$. Thus,~\cite[Table 1]{len.mar:20} implies
  \[\overline{m} = -(\frac{B^u}{A^r})^{(uw-vr)},\quad \text{ where } A = -\frac{\bar{c}\,r\,B}{u \bar{b}} \text{ and } B= -\frac{\bar{a}\,u}{\bar{c}\,(u-r)}.
  \]

  Since the tangency is horizontal, we have $uw-vr = \pm 1$ and $u-r = \pm 2$. Thus,
  \begin{equation}\label{eq:signM}\operatorname{sign}(\overline{m}) = (-1)^{r+1} (r\,u)^r \operatorname{sign}\big (\bar{c}\,\bar{b}\big ) ^r.
  \end{equation}
As  \autoref{tab:edgeOnEdgeMult2} shows, there are three $\Sn{3}$-representatives for the direction of $(u-r,v-w)$, namely $(-2,1)$, $(2,1)$ or $(-2,3)$. This gives six possible combinations for $(u,v)$ and $(r,w)$:
  \[ ((u,v),(r,w))  = \pm ((1,-1),(-1,0)),\;   \pm ((1,1), (-1,0)), \text{ or } \pm((1,-1), (-1,2)).
  \]
\noindent    Substituting these values in~\eqref{eq:signM} yields $-s_{uv}s_{rw} = -s_{v_l}s_{v_r}$ as we wanted to show.
\end{proof}

\section{Real liftings of tropical bitangents}\label{sec:real-lift-trop}

\autoref{sec:lift-trop-bitang} addresses local liftings over $\KKr$ and provides
the combinations of tangency types that can occur for members of all bitangent shapes. In this section we combine these results to prove~\autoref{thm:reallifting}. In particular, we provide explicit sign constraints that characterize when a given bitangent shape has exactly four real lifts. These conditions are stated in~\autoref{tab:lifting} and were obtained by analyzing each of the 41  $\Sn{3}$-representatives of bitangent shapes from~\autoref{fig:2dCells}.

\begin{table}
  \begin{tabular}{|c|c|}
    \hline
    Shape & Lifting conditions \\
    \hline\hline
(A) & $\!\!\!(-s_{1v}s_{1,v+1})^i s_{0i} s_{22}>0$
     \text{ and }\;
    $(-s_{u1}s_{u+1,1})^j s_{j0} s_{22}>0$\\
   \hline     (B)  & $(-s_{1v}s_{1,v+1})^{i+1}s_{0i}s_{21} > 0$\;\;  \text{ and }\;
    $(-s_{21})^{j+1}{s_{31}}^js_{1v}s_{1,v+1}s_{j0} > 0$ \\
   \hline (C) &
      $(-s_{11})^{i+j}(s_{12})^i(s_{21})^js_{0i}s_{j0}>0 \;\text{ and }\;  (-s_{21})^{k+j}(s_{12})^k(s_{11})^js_{k,4-k}s_{j0}>0$. \\
   \hline  (H),(H')& \;$(-s_{1v}s_{1,v+1})^{i+1} s_{0i} s_{21}> 0$ \;
   \text{and}\;  $s_{1v}s_{1,v+1}s_{21}s_{40} < 0$\\
   \hline (M)  & $(-s_{1v}s_{1,v+1})^{i+1}s_{0i}s_{21} > 0$ \;\text{and}\;\; $s_{1v}s_{1,v+1}s_{30}s_{31} > 0$ \\
\hline   \hline (D) & $(-s_{10}s_{11})^i s_{0i} s_{22} > 0$ \\
   \hline
     (E),(F),(J) & $(-s_{1v}s_{1,v+1})^is_{0i}\, s_{20} > 0 $\\
   \hline (G) & $(-s_{10}s_{11})^i s_{0i} \,s_{k,4-k} > 0$ \\
   \hline
   (I),(N)  & $ s_{10}s_{11} s_{01} s_{k,4-k} < 0 $\\
   \hline (K),(T),(T'),&    \multirow{2}{*}{
     $  s_{00} s_{k,4-k} > 0 $}\\
(T''),(U),(U'),(V) & \\
   \hline (L),(O),(P)    & $ s_{10}s_{11} s_{01} s_{22} < 0$ \\
   \hline (L'),(Q),(Q'), &  \multirow{2}{*}{$s_{00}s_{22}>0$}\\
   (R),(S) & \\
   \hline\hline
   rest & no conditions\\
\hline  \end{tabular}
  \caption{Real lifting sign conditions for each representative bitangent class.\label{tab:lifting}}
 \end{table}

Throughout we assume $V(q)$ is generic as in~\autoref{rem:genericity} and defined over $\KKr$. Members of bitangent classes with weight one will automatically lift over $\KKr$, so we need only look at vertices with weights two or four. For each such member, we combine the local sign conditions given in~\autoref{tab:localLifts3a3c} arising for each tangency point and verify the system of two inequalities is either consistent or has no solution. This proves~\autoref{thm:reallifting} for bitangent classes with a single higher weight member. 
When two members of a  class have weight two, we show the local conditions on each vertex agree, thus giving a global affirmative answer for the class.
Our first result  involves shapes that always lift over $\KKr$:
\begin{proposition}\label{pr:exactly4} Bitangent  shapes (II), (W), (X), (Y), (Z) or (AA) through (HH) have four real lifts.
\end{proposition}
\begin{proof} Since all weight one members lift over $\KKr$, we only need to show the statement holds for (II). We work with the $\Sn{3}$-representative depicted in~\autoref{fig:2dCells} and show that its single weight two member has two bitangent lifts over $\KKr$. Combined with the two real lifts for the weight one vertices, we get four real lifts in total for shape (II).

  We let  $\Lambda$ be the tropical bitangent associated to the weight two vertex for this shape. Its two tangency points have types (3a) and (2) along the horizontal and vertical ends of $\Lambda$, respectively, as we see in the right-most picture in~\autoref{fig:ShapeII}. The first point yields two solutions for $m\in \KK$ with $\overline{m}\in \RR$, whereas the second one determines a unique value for $m/n\in \KKr$. In particular, $\bar{n}\in \RR$.

We use the notation from~\autoref{fig:ShapeII} for the remaining of this proof.   In order to check the validity of the local lifting conditions  given in~\autoref{tab:localLifts3a3c} for the horizontal tangency along the edge $\overline{v_rv''}$ of $\Gamma$ we must determine   the sign of both $\overline{m}$ and $\bar{n}$ and  relevant vertices in the Newton subdivision of $q$ seen in the figure. In particular, the three relevant parameters $u,v,w$ are determined by $v_r^{\vee}$, i.e. $(u,v,w) = (1,2,0)$ for the horizontal type (3a) tangency. In turn, \autoref{pr:realLifts3a3c} provides the sign of $\overline{m}\in \RR$: it is $s_{12}s_{13}$.

  The vertex where the tangency (2) occurs is adjacent to an edge $e$ with direction $(-2,1)$. Its dual edge $e^{\vee}$ has endpoints $v_r=(0,0)$ and $v_l=(1,2)$.  \autoref{lm:tangency2} determines the sign of $\overline{m}/\bar{n}$ from these two vertices: it equals $-s_{00}s_{12}$. Thus,
  \begin{equation}\label{eq:signnII}
    \operatorname{sign}(\bar{n}) = \operatorname{sign}(\bar{n}/\overline{m})\operatorname{sign}(\overline{m}) = 
    -s_{00}{(s_{12})}^2s_{13} = -s_{00}s_{13}.
  \end{equation}

  We will have two solutions in $\KKr$ for $m$ and thus two real bitangent lifts for $\Lambda$ if and only if the formula  in~\autoref{tab:localLifts3a3c} for a horizontal type (3a) tangency holds. Replacing the values for $(u,v,w)$ computed earlier and $\operatorname{sign}(\bar{n})$ given in~\eqref{eq:signnII} yields the positive expression
  \[
(-1)^{0+2+1} (s_{12} s_{13})^{0+2} s_{00}\, s_{13} (-s_{00}s_{13}) =  \big (s_{00}\,s_{12}\,(s_{13})^2\big)^2 = 1.   \qedhere\]
\end{proof}

The remaining 28 shapes in~\autoref{fig:2dCells} have four bitangent lifts over $\KK$, and it is possible for none of them to be defined over $\KKr$. To prove~\autoref{thm:reallifting} for these cases, we group together shapes arising from partial dual subdivisions to $\Gamma$ sharing similar combinatorial properties, as we did in~\autoref{sec:lift-trop-bitang}.

Next, we address those shapes listed in the second row of~\autoref{tab:classificationShapes} excluding (C). For all these shapes, the $\KK$-liftable bitangent members have a disconnected intersection with $\Gamma$. The relevant tangency pairs are listed in~\autoref{tab:possibleTangencyPairs}.

\begin{proposition}\label{pr:secondRowTable} Bitangent  shapes (D), (L), (L'), or (O) through (S) have either zero or four lifts over $\KKr$.
  Furthermore, the $\Sn{3}$-representatives listed in~\autoref{fig:2dCells} have four lifts over $\KKr$ if, and only if,
  \begin{equation}\label{eq:shapesDetc}
  (-1)^i(s_{10}\,s_{11})^i\,s_{0i}s_{22} >0.
  \end{equation}
Here, 
$(0,i)$ is the red marked vertex in the corresponding partial Newton subdivisions depicted in~\autoref{fig:NP}. In particular,  $i=0$ for shapes (L'), (Q), (Q'), (R) and (S), whereas $i=1$ for (L), (O) and (P). The value of $i$ for shape (D) is not restricted.
\end{proposition}
\begin{proof} We discuss the statement for the $\Sn{3}$-representatives depicted in~\autoref{fig:2dCells}. The proof of~\autoref{prop:oneComponentIntersections} ensures that each of the two members of each bitangent class with weight two have two distinct tangency points.  For the rightmost vertex, the corresponding line $\Lambda$ has  local  (3c) and (2) type tangencies, along its horizontal and vertical ends. By~\autoref{pr:realLifts3a3c}, the first point yields two values for $m\in \KK$ with $\overline{m}\in \RR$, whereas the second determines a unique $m/n\in \KKr$.

  In the notation of~\autoref{fig:LocalNP} we have $(u,v,r,w)=(1,0,2,i)$ for the horizontal type (3c) tangency. We can read this information from the labeled dual triangles in the two possible partial dual subdivisions to $\Gamma$ depicted in the top-left of~\autoref{fig:ShapeDLetc}. Note that the value of $w$ changes depending on the bitangent shape, as predicted by in the statement. By~\autoref{pr:realLifts3a3c}, the sign of $\overline{m}$ equals $s_{10}s_{11}$.
  
As we see from~\autoref{fig:ShapeDLetc}, the edge $e'$ responsible for the vertical type (2) tangency has direction $(2,-1)$. Its dual $(e')^{\vee}$ has vertices $v_l=(1,0)$ and $v_r=(2,2)$. 
  Replacing this data in the formula for the (3c) horizontal tangency in~\autoref{tab:localLifts3a3c} and ignoring terms with even powers yields expression \eqref{eq:shapesDetc} as the condition  for having two solutions for $m$ in $\KKr$.

  To conclude, we must show the same necessary and sufficient condition for lifting over $\KKr$ arises for the other member of the class that has weight two.
  We let $\Lambda'$ be the corresponding tropical bitangent and let $m'$ and $n'$ be the coefficients of one of its lifts in $\KK$. We argue for each shape separately, identifying the parameter values needed for the corresponding lifting restrictions from~\autoref{fig:ShapeDLetc}:
  
  \begin{description}
  \item[(D)] The local tangencies are of type (3a) and (2) and occur along the horizontal and diagonal ends of $\Lambda'$. In this situation, the first point determines two possible values for $m'$ in $\KK$ and the second one determines a unique solution $n'\in \KKr$. The dual edge $e^{\vee}$ arising from the latter has vertices $v_l'=(1,1)$ and $v_r'=(2,2)$, so the sign of $\overline{n'}$ is $-s_{11}s_{22}$ by~\autoref{lm:tangency2}.
    Furthermore, the values for $u,v$ and $w$ for the first tangency point agree with those given for the line $\Lambda'$. Replacing this information in the formula from horizontal (3a) tangencies in~\autoref{tab:localLifts3a3c} yields~\eqref{eq:shapesDetc}.
  \item[(L) and (L')] The local tangencies are of type (5a) (respectively, (6a)) and (2). The diagonal end of $\Lambda'$ is responsible for both tangencies. By~\autoref{rem:type5a} (respectively, \autoref{pr:realLift46a}), the first point determines two solutions  $m'\in \KK$, and the second a single value $n'\in \KKr$. Since $v_l'=(1,1)$ and $v_r'=(2,2)$, we get $\operatorname{sign}(\overline{n'}) =  -s_{11}s_{22}$ by~\autoref{lm:tangency2}. The first tangency point has parameters $u=v=0$ for both shapes. 

    Applying the formulas in~\autoref{tab:localLifts3a3c} for (5a) and (6a) for the diagonal end gives the conditions $-s_{11}s_{22}s_{10}s_{01}>0$ for shape (L) and $s_{00}s_{22}>0$ for shape (L'). Since $i=0$ for shape (L') and $i=1$ for shape (L),  these two inequalities agree with  expression~\eqref{eq:shapesDetc}.
  \item[(O)] The two tangencies are of type (3a) and (2), and occur along the vertical and diagonal ends of $\Lambda'$. By~\autoref{fig:ShapeDLetc}, $(u,v,w) = (0,1,1)$ for the first point and the local equations determine $m'/n'$. The second tangency has the same behavior as with shape (D) and it determines a unique value for $n'$ with  $\operatorname{sign}(\overline{n'})= -s_{11}s_{22}$. Replacing these values in the formula from the table yields $-s_{01}s_{10}s_{11}s_{22}$. This agrees with the condition in~\eqref{eq:shapesDetc} since $i=1$.
    \item[(P)] In this case, we have a vertical (3c) tangency and a type (2) one along the horizontal end of $\Lambda'$.  The second point contributes a unique solution for $n'$ which lies in $\KKr$. By~\autoref{fig:ShapeDLetc}, the first tangency point has associated combinatorial data $(u,v,r,w)= (0,1,2,1)$.~\autoref{tab:localLifts3a3c} gives the condition $-s_{01}s_{10}s_{11}s_{22}>0$ for having two solutions $m'/n'$ in $\KKr$. This is the same as the one in~\eqref{eq:shapesDetc} since $i=1$.
    \item [(Q) and (Q')] The two local tangencies are of types (2) and either (4) or (6a), respectively. Both arise from intersections with the diagonal end of $\Lambda'$. The relevant vertices   in the dual subdivision to $\Gamma$ for these two points are $(u,v)=(0,0)$, $v_l'=(1,1)$ and $v_r'=(2,2)$. In particular, $\operatorname{sign}(\overline{n'}) =  -s_{11}s_{22}$ from ~\autoref{lm:tangency2}. Replacing this information in the formula for diagonal tangencies of types (4) or (6a)  from~\autoref{tab:localLifts3a3c} recovers~\eqref{eq:shapesDetc} since $i=0$.
    \item[(R)] The reasoning for shape (O) works for this situation as well, except that the vertical (3a) tangency yields $(u,v,w)=(0,1,0)$, and we must set $i=0$ in~\eqref{eq:shapesDetc}.
  \item[(S)] The claim  follows by the same arguments used for (P), with the exception that $(u,v,r,w)=(0,1,2,0)$ for the vertical type (3c) tangency and $i=0$ in~\eqref{eq:shapesDetc}.
  \end{description}
  
   Thus, we conclude that either both members of the class lift twice over $\KKr$ or none of them do, and condition~\eqref{eq:shapesDetc} characterizes when  real solutions exist.
  \end{proof}

Next, we discuss the remaining shapes except for shape (C), which we treat later:
\begin{proposition}\label{pr:rest} Bitangent shapes (A), (B), (E) through (K), (M), (N), (T), (T'), (T''), (U), (U') or (V) have either zero or four bitangent lifts over $\KKr$. The conditions for having real lifts are given in~\autoref{tab:lifting}.
  \end{proposition}
\begin{proof} We use the $\Sn{3}$-representative shapes and the associated partial Newton subdivisions to $\Gamma$ summarized in~\autoref{fig:NP} to determine the parameters $u,v,r,w$ used in the corresponding formulas from~\autoref{tab:localLifts3a3c}. In all cases, we have either one or two members for each class that lift over $\KK$, with a combined lifting multiplicity of 4. We prove the statement for each case, grouping shapes whenever the dual subdivisions to $\Gamma$ share some common features.
  \vspace{-1ex}
  
    \begin{description}
    \item[(A)] The unique member $\Lambda$ of this bitangent class has two tangency points, both of type (3c) along the  horizontal and vertical ends. The first point determines $m$, whereas the second determines $m/n$. We must determine the values of the parameters $u$ (respectively, $v$) $r$ and $w$ needed to use the formulas in~\autoref{tab:localLifts3a3c}. For a concrete example, we refer to~\autoref{fig:quartic}.

      In the notation of~\autoref{fig:LocalNP}, we have  $(u,r,w)=(1,2,i)$ ($v$ is unknown as the dotted vertical segment in the figure indicates) arising from the horizontal tangency. This gives the first of the two formulas in~\autoref{tab:lifting}, ensuring that the coefficient $m$ has two solutions in  $\KKr$. 

      In turn, for the second point we have $(v,r,w) = (1,2,j)$ ($u$ is unknown as seen in the dotted horizontal segment in the figure). Replacing these values in the formula from~\autoref{tab:localLifts3a3c} yields the second formula in~\autoref{tab:lifting}. If the condition is satisfied, we will have two values for $m/n$ in $\KKr$. 
Thus, combining both positivity constraints ensures that the shape has four $\KKr$-bitangent lifts. If any of them fails, there are no $\KKr$-lifts.
    \vspace{1ex}
  \item[(B)] We have two tangency points; a horizontal (3c) one with $(u,r,w)=(1,1,i)$ ($v$ is unknown) and a vertical (3a) one with $(u,v,w)=(2,1,j)$. By~\autoref{pr:realLifts3a3c}, the first point determines two values for $m\in \KK$ with $\overline{m}\in \RR$ of sign $s_{1v}s_{1,v+1}$. The second one determines two solutions for $m/n\in \KK$ with $\overline{m}/\bar{n}\in \RR$ of sign $s_{21}s_{31}$. Thus, $\bar{n}\in \RR$ and its sign is $s_{1v}s_{1,v+1}s_{21}s_{31}$. Substituting these values in the conditions for real liftings in~\autoref{tab:localLifts3a3c} and eliminating factors with even exponents yield the two positivity conditions listed in~\autoref{tab:lifting}.
    \vspace{1ex}

  \item[(E), (F) and (J)] Both weight-two vertices in each of these bitangent shapes have a horizontal (3c) tangency with $(u,r,w)=(1,0,i)$ and a diagonal tangency (of type (2), (4) or (6a)). The second tangency has local lifting multiplicity one and the local equations provide a unique value for $n$, therefore in $\KKr$. The parameter values for the (3c) tangency are the same for both members, and they are determined from~
    Figures~\ref{fig:EFShapes} and~\ref{fig:Case4_1dim}, respectively. The local equations yield two solutions for $m\in \KK$ for each member. Thus, the condition to have bitangent lifts in $\KKr$ is the same for both members of the bitangent shape, since the formula in~\autoref{tab:localLifts3a3c} only depends on $u,v,w$ and $r$. Eliminating factors with even exponents yields the expected inequality from~\autoref{tab:lifting}.
    \vspace{1ex}

  \item[(G)] As we see from the first two pictures in~\autoref{fig:GShape}, members associated to a weight-two vertex of this shape have a horizontal (3a) tangency with $(u,v,w)= (1,0,i)$ combined with a diagonal type (2) tangency with $v_l= (1,1)$ and $v_r=(k,4-k)$, for some $k=0,2$ or $4$.
    By~\autoref{lm:tangency2}, the unique value of $n$ obtained from the type (2) tangency lies in $\KKr$ and $\operatorname{sign}(\bar{n}) = -s_{11}s_{k,4-k}$. Thus, the condition for lifting each member over $\KKr$ is the same and it is obtained directly from the formula in~\autoref{tab:localLifts3a3c}.
  \vspace{1ex}

\item[(H) and (H')] As we see in~\autoref{fig:Cases1And2_1dim}, the weight-four member has two horizontal tangencies; one of type  (3c)  with $(u,r,w)=(1,1,i)$ and one of types (4) or (6a), respectively, with $(u,v)=(2,0)$. By~\autoref{pr:realLifts3a3c}, the first point yields two solutions for $m\in \KK$ with $\overline{m} \in \RR$ of sign $s_{1v}s_{1,v+1}$ ($v$ is unknown). The second tangency point produces two values for $n\in \KK$. Replacing the combinatorial data and the sign of $\overline{m}$ into the formulas provided in~\autoref{tab:localLifts3a3c} yields the two positivity conditions in~\autoref{tab:lifting}.
  \vspace{1ex}
  \item[(I) and (N)] The rightmost weight-two member $\Lambda$ of both shapes has a horizontal (3a) tangency with $(u,v,w)=(1,0,1)$ and a diagonal (2) tangency with $v_l= (1,1)$ and $v_r=(k,4-k)$ for $k=0,2$ or $4$. These values are obtained from~\autoref{fig:Case3_1dim}. By~\autoref{lm:tangency2}, the type (2) point determines a unique solution for $n\in \KKr$ with 
    $\operatorname{sign}(\bar{n})= -s_{11}s_{k,4-k}$. Replacing this data into the lifting conditions for the (3a) horizontal  tangency  yields $-s_{10}s_{11}s_{01}s_{k,4-k}>0$.

    To derive the statement we must show that this condition also arises for the other weight-two member of the class, labeled $\Lambda'$. We check this for each of the two shapes separately. 

    If the class has shape (I) (respectively, (N)) the tangency points of $\Lambda'$ are of types (2)  and (5a) (respectively, vertical (3a)). The former occurs along the diagonal end of $\Lambda'$ with $v_l= (1,1)$ and $v_r=(k,4-k)$ for $k=0,2$ or $4$ (the same values as for $\Lambda$). Therefore,  the sign of $\overline{n'}$ arising from $\Lambda'$ agrees with the one for $\bar{n}$ from $\Lambda$.

    For both shapes, the second tangency point of $\Lambda'$ determines two solutions for $m'$ and the required data for the corresponding lifting conditions in~\autoref{tab:localLifts3a3c} is $(u,v)=(0,0)$ for shape (I) and $(u,v,w) = (0,1,1)$ for shape (N). A simply substitution recovers the positivity constraint obtained earlier for $\Lambda$.
  \item[(K)] By~\autoref{fig:Case6_1dim}, the rightmost weight-two member $\Lambda$ has tangencies of type (3a) and (2) along the horizontal and diagonal ends. The first point determines two values for $m$ and $(u,v,w)=(1,0,0)$. The second point has associated vertices $v_l=(1,1)$ and $v_r=(k,4-k)$ with $k=0,2,4$ in the dual subdivision to $\Gamma$ and yields a unique solution for $n\in \KKr$ with $\operatorname{sign}(\bar{n})=-s_{11}s_{k,4-k}$ by~\autoref{lm:tangency2}. Replacing this data on the corresponding lifting conditions on~\autoref{tab:localLifts3a3c} for the first point yields $s_{00}s_{k,4-k}>0$.

  We must check the other weight-two member of this class (called $\Lambda'$) has the same constraint for lifting over $\KKr$.
  The tangency points of $\Lambda'$ are of type (6a) with $(u,v)=(0,0)$ and of type (2). The latter occurs along the diagonal end of $\Lambda'$ and has the same  associated  vertices $v_l$ and $v_r$ as  $\Lambda$. In particular, the latter point determines a unique $n'\in \KKr$ with initial form of the same sign as the one for $\Lambda$. Replacing this into the formula for diagonal type (6a) tangencies from~\autoref{tab:localLifts3a3c} recovers  $s_{00}s_{k,4-k}>0$, as we wanted.
  \vspace{1ex}

\item[(M)] In this case, the weight-four member has one horizontal (3c) tangency with $(u,r,w)= (1,1,i)$, which determines two solutions in $m$ with $\overline{m}\in \RR$ of sign $s_{1v}s_{1,v+1}$ by~\autoref{pr:realLifts3a3c}. The value of $r$ corresponds to the  marked vertex in~\autoref{fig:Case4_1dim}
  This second tangency has type (5a) with $(u,v)=(2,0)$ and yields two solutions for $n$.  Replacing this information in the two relevant lifting formulas from~\autoref{tab:localLifts3a3c} gives the two expressions in~\autoref{tab:lifting}.
  \vspace{1ex}
\item[(T), (T') and (T'')] The combinatorial information necessary to determine the constraints for this case can be recovered from Figures~\ref{fig:partialNPTelal} and~\ref{fig:optionsForSigma}. We show that each weight-two vertex of these shapes yield the same constraint.

  We start with the two vertices of (T). They have diagonal tangencies of types (4) and (2). The latter contributes a unique value for $n\in \KKr$ with $\operatorname{sign}(\bar{n})=-s_{11}s_{k,4-k}$, where $k=0,2$ or $4$. The other tangency point fixes two values for $m$ and $(u,v)=(0,0)$ for both members. The lifting conditions for both vertices of (T) arising from the corresponding formula in~\autoref{tab:localLifts3a3c} is the same, namely  $s_{00}s_{k,4-k}>0$.

  Next, we look at the rightmost member of (T'), which has a diagonal (6a) tangency with $(u,v)=(0,0)$ and a diagonal type (2) tangency giving $\operatorname{sign}(\bar{n})= - s_{11}s_{k,4-k}$. Replacing this information in the appropriate formula from the same table recovers  $s_{00}s_{k,4-k}>0$.
  
  Finally, we look at the leftmost member of (T''). Applying $\tau_0$ turns this member into the rightmost member of a class with the same shape. By~\autoref{rem:type6a}, the lifting condition is invariant under $\tau_0$, thus, it is also valid for the leftmost vertex of shape (T'').

    \vspace{1ex}
  
\item[(U), (U')  and (V)] As~\autoref{fig:NP} shows, the combinatorial data for the rightmost vertex with weight two agrees with the analogous  member of shape (K). Thus, the lifting conditions over $\KKr$ for this vertex is $s_{00}s_{k,4-k}>0$.  We must check the same inequality arises from the other weight-two member of each shape.

  For shapes (U) and (U'), the tangencies occur along the diagonal end of the line and they match those corresponding to the left-most vertex of shapes (T) and (T''), respectively. Thus, the lifting conditions for this member are also $s_{00}s_{k,4-k}>0$.

  For shape (V), the tangencies are of type (2) along the diagonal end of a line and of type (3a) along the vertical end, with $(u,v,w)=(0,1,0)$. The parameters are obtained from~Figures\ref{fig:partialNPTelal} and~\ref{fig:optionsForSigma} after applying the map $\tau_0$.
The first tangency gives a unique solution for $n$ in $\KKr$ with $\operatorname{sign}(\bar{n})=-s_{11}s_{k,4-k}$. Substituting this the lifting formulas for $m$ from the vertical (3a) tangency yields $s_{00}s_{k,4-k}>0$.\qedhere
  \end{description}
  \end{proof}

We conclude by discussing the real-lifting conditions for shape (C). Up to translation, we place this class at (0,0). The Newton subdivision of $q\in \KKr[x,y]$ depicted in~\autoref{fig:NP}~
{(C)} has three colored marked vertices: $(0,i)$, $(j,0)$, $(k,4-k)$. They determine the endpoints of the three edges adjacent to $(0,0)$ in $\Gamma$.

We let $\lambda_1, \lambda_2, \lambda_3$ be the length of the vertical, horizontal and diagonal edges of $\Gamma$ adjacent to $(0,0)$. By~\autoref{rem:genericity}, we   assume $\lambda_1<\lambda_2\leq \lambda_3$. This assumption ensures that $j=1,2$ or $3$. Our next result provides the real-lifting conditions for this representative shape, adapting the ideas seen in proof of~\autoref{pr:realLifts3a3c} for a type (3a) tangency.

\begin{proposition}\label{pr:shapeC} Consider the $\Sn{3}$-representative bitangent class of shape (C) discussed above. Then, this class has either none or exactly four bitangent lifts over $\KKr$. The conditions for having real lifts depend on the marked vertices $(0,i)$, $(j,0)$ and $(k,4-k)$:
  \[
      (-s_{11})^{i+j}(s_{12})^i(s_{21})^js_{0i}s_{j0}>0  \quad\text{ and }\quad  (-s_{21})^{k+j}(s_{12})^k(s_{11})^js_{k,4-k}s_{j0}>0. 
  \]
  \end{proposition}
\begin{proof}
  We let $(\ell,p,p')$ be one of the four bitangent lifts of the class over $\KK$.  A simple computation of the divisor associated to $\Lambda$ in the skeleton of $\Gamma$ places the two tropical tangency point in $\Gamma$ at $\Trop\, p=(\frac{\lambda_1-\lambda_2}{2},0)$ and $\Trop\, p'=(\frac{\lambda_3-\lambda_1}{2},\frac{\lambda_3-\lambda_1}{2})$.

  By construction, $\val(a_{11}) = \val(a_{12}) = \val(a_{21}) < \val(a_{lr})$ for all other $(l,r)$ in the support of $q$.     Thus, after scaling $q$ by $1/a_{12}$ if necessary, we may assume $q\in R[x,y]$ with $a_{12}=1$ and $\val(a_{11}) = \val(a_{21})=0<\val(a_{rs})$. Furthermore,
  \begin{equation}\label{eq:valsaij}
    \val(a_{0i})=\lambda_2,\quad \val(a_{j0})=\lambda_1 \quad \text{ and }\quad \val(a_{k,4-k})=\lambda_3.
  \end{equation}
  In addition, $\val(a_{0,p})>\lambda_2$ if $p\neq i$, $\val(a_{p,0})>\lambda_1$ if $p\neq j$ and $\val(a_{p,4-p})>\lambda_3$ if $p\neq k$.
  
 The coefficients of $\ell$ will have the form $m=m_1 + m_2$ and $n=n_1+n_2$ with $m_1, n_1\in \KKr$
  satisfying $\val(m)=\val(m_1)=0<\val(m_2)$ and $\val(n)=\val(n_1)=0<\val(n_2)$. ~\autoref{lm:liftingC} below provides an explicit formula for $m_1$ and $n_1$ in $\KKr$ with $\overline{m_1}=\overline{a_{11}}$ and $\overline{n_1} = \overline{a_{21}}$. Once $m_1$ and $n_1$ are determined, there will be two independent solutions for $m_2,n_2\in \KK$. Each of them will be fixed by its initial form.

  Following~\cite[Proposition 3.12]{len.mar:20}, we use $m_1$ and $n_1$ to re-embed $V(q)$ in $\KK^3$ via the ideal $I:=\langle q, z-y-m_1-n_1x\rangle \subset \KKr[x,y,z]$. Our choice of $m_1$ and $n_1$ guarantees that viewed in $\KK^3$, the tangency points have tropicalizations $\Trop\,p = (-\frac{\lambda_2-\lambda_1}{2}, 0, -\frac{\lambda_1+\lambda_2}{2})$ and $\Trop\,p' = (\frac{\lambda_3-\lambda_1}{2}, \frac{\lambda_3-\lambda_1}{2}, -\lambda_1)$.

  After projecting to the $xz$-plane, these two points are vertices of $\Trop\,V(\tilde{q})$ where $\tilde{q}(x,z) = q(x,z-m_1-n_1x)$.  By~\autoref{lm:liftingC}, their dual triangles in the Newton subdivision of $\tilde{q}$ satisfy the requirements of~\cite[Lemma 3.9]{len.mar:20}. Moreover, the projection of $\Trop\,V(\langle \ell,  z-y-m_1-n_1x\rangle)$ to the $xz$-plane has vertex $(\frac{\lambda_3-\lambda_2}{2}, -\frac{\lambda_1+\lambda_2}{2})$.  This gives two independent values for $m_2$ and $n_2$ satisfying     $\val(m_2) = \frac{\lambda_1+\lambda_2}{2}, \val(n_2) = \frac{\lambda_3+\lambda_1}{2}$, with 
  \begin{equation}\label{eq:solsC}
 \overline{m_2} = \pm \frac{2}{\overline{\tilde{a}_{11}}}{\sqrt{\overline{\tilde{a}_{00}}\;\overline{\tilde{a}_{20}}}} \quad \text{ and } \overline{n_2} = \pm \frac{2}{\overline{\tilde{a}_{21}}}{\sqrt{\overline{\tilde{a}_{20}}\;\overline{\tilde{a}_{40}}}}.
  \end{equation}
    The formula for $\overline{m_2}$ can be obtained from~\cite[Table 1]{len.mar:20}. The one for $\overline{n_2}$ is obtained in the same way, after acting by $\tau_0$.

    Thus, $m_2,n_2\in \KKr$ if, and only if, the two radicands in~\eqref{eq:solsC} are positive. The explicit values for  $\overline{\tilde{a}_{00}}$,  $\overline{\tilde{a}_{20}}$  and $\overline{\tilde{a}_{40}}$ depend on  $i,j$ and $k$, and are listed in~\autoref{lm:liftingC} below.
    The formulas in~\autoref{tab:lifting} are obtained from these positivity constraints after replacing each $s_{lr}$ by $s_{lr}s_{12}$, to free ourselves from the simplifying assumption $a_{12}=1$. 
\end{proof}

\begin{lemma}\label{lm:liftingC} Let $q\in R[x,y]$ and $\Gamma$ be as in~\autoref{pr:shapeC} with $a_{12}=1$. Then, there exists $m_1,n_1\in \KKr$ of valuation zero, with $\overline{m_1}=\overline{a_{11}}$ and $\overline{n_1}=\overline{a_{21}}$, for which the coefficients $\tilde{a}_{lr}$ of $\tilde{q}(x,z)=q(x,z-m_1-n_1x)$   satisfy:
  \begin{enumerate}[(i)]
  \item $\val(\tilde{a}_{00})=\lambda_2$, $\val(\tilde{a}_{40})=\lambda_3$, $\overline{\tilde{a}_{00}} = \overline{a_{0i}}(-\overline{a_{11}})^i$, and
    $\overline{\tilde{a}_{40}} = \overline{a_{k,4-k}}(-\overline{a_{21}})^{4-k}$;
  \item $\val(\tilde{a}_{11}) = \val(\tilde{a}_{21})=0$,  $\overline{\tilde{a}_{11}} = -\overline{a_{11}}$ and $\overline{\tilde{a}_{21}} = -\overline{a_{21}}$;
  \item $\val(\tilde{a}_{10}) > \frac{\lambda_1+\lambda_2}{2}$, $\val(\tilde{a}_{30})> \frac{\lambda_3+\lambda_1}{2} $,     $\val(\tilde{a}_{20}) =\lambda_1$, and
    \begin{equation}\label{eq:initiala11}
    \overline{\tilde{a}_{20}} = \overline{a_{20}}  \; \text{ if } j=2, \quad \text{ or } \quad \overline{\tilde{a}_{20}} =-\overline{a_{j0}}\left ({\overline{a_{21}}} / {\overline{a_{11}}}\right)^{2-j} \; \text{ if }j=1,3.
    \end{equation}
        \end{enumerate}
  \end{lemma}

\begin{proof} We follow the techniques developed in the proof of~\cite[Proposition 3.7]{len.mar:20}. The conditions in the statement will guarantee that the Newton subdivision of $\tilde{q}$ has specific triangles involving $(1,1)$, $(1,2)$ and the points $(0,0)$, $(2,0)$ and $(4,0)$ in the $x$-axis, as seen in~\cite[Figure 17]{len.mar:20}. In what follows, we give explicit formulas for $m_1$ and $n_1$. To simplify notation,  we  write them as  $m_1=a_{11} + m_1'$ and $n_1=a_{21}+n_1'$, with $\val(m_1'),\val(n_1')>0$. This is the same approach that we used for determining local lifting conditions for type (3a) tangencies in the proof of~\autoref{pr:realLifts3a3c}.

  The ``cone feeding'' process allows us to determine the expected valuations of all coefficients $\tilde{a}_{lr}$ where $(l,r)$ is in the trapezoid with vertices $(0,0), (1,1)$, $(2,1)$ and $(4,0)$. The claims for $\tilde{a}_{00}$ and $\tilde{a}_{40}$ in item (i) follow from \eqref{eq:valsaij} since $a_{0,i}$ and $a_{k,4-k}$ have the lowest valuation among all $a_{0r}$ and $a_{r,4-r}$, respectively. Similar reasoning  yields the conditions in (ii). The parameters $m_1'$ and $n_1'$ play no role in ensuring that (i) and (ii) hold. To finish, we must show the claims in (iii). We do so by picking appropriate values for $m_1'$ and $n_1'$.

  We first focus on $\tilde{a}_{10}$ and $\tilde{a}_{30}$. We collect those $a_{lr}$ yielding contributing summands for these coefficients with the potential of having valuation $< \lambda_2$ and $\lambda_3$, respectively. This yields $(r,l) = (1,1)$, $(1,2)$, $(1,0)$ for $\tilde{a}_{10}$, and $(r,l)=(1,2), (2,1)$, $(0,3)$ and $(3,0)$ for $\tilde{a}_{30}$.  This process determines two univariate polynomials in $m_1'$ and $n'_1$:
  \begin{equation}\label{eq:vanishing10And30}
    \hspace{5ex}\begin{cases}
          f:= (m_1')^2 + a_{11} \, m_1' + a_{10}\, ,\\
  g:= -a_{03}\,(n_1')^3 + (1-3\,a_{03}\,a_{21})((n_1')^2 +  a_{21}\, n_1') + (a_{30}-a_{03}\,a_{21}^3).
\end{cases}
     \end{equation}
  Picking $m_1'$ and $n_1'$ to be roots of $f$ and $g$ with positive valuation, would ensure $\val(\tilde{a}_{10})\geq \lambda_2$ and $\val(\tilde{a}_{30})\geq \lambda_3$. These inequalities imply those in (iii) since $\lambda_3,\lambda_2>\lambda_1$.

  In order to choose the appropriate roots of $f$ and $g$, we notice that their valuation and initial forms are determined by the initial forms of $f$ and $g$ in~\eqref{eq:vanishing10And30}:   $\overline{m_1'}$ and $\overline{n_1'}$ must be solutions to $\overline{f}=\overline{g}=0$ in $\RR$. 

  An analysis using (max) tropical polynomials confirms the existence of solutions with the desired properties. We analyze each equation separately and invoke~\autoref{lm:multivariateHensel}  to lift the input solutions of $\overline{f}=0$ and $\overline{g}=0$ to solutions $(m_1', n_1')$ over $\KKr$.

  The restrictions on the valuations of the coefficients of $q$ ensures that $M=-\val(a_{10})<0$ lies in the (max) tropical hypersurface determined by $\trop(f)$, so $f$ admits a solution $m_1'$ with $\val(m_1)=\val(a_{10})>0$. Furthermore, the initial form $\overline{m_1'}$ solves $\overline{a_{10}} + \overline{a_{11}}\,\overline{m_1'} = 0$. Since $\overline{f'}(\overline{m_1'})\neq 0$, we can lift $\overline{m_1'}$ to a unique solution of $f(m_1')=0$ in $\KKr$.

  The analysis for $g$ is very similar, except that we need to separate two cases since the constant term of $g$ could be zero in which case we pick $n_1'=0$. This choice satisfies our desired properties. Note that $g(0)=0$ can only occur if $\val(a_{03})=\val(a_{03})$. This forces $j\neq 3$ by our restrictions on the valuations of the coefficients of $q$.
  
  On the contrary, if $g(0)\neq 0$, then  $N=-\val(a_{3,0}-a_{03}\,a_{21}^3)<0$ lies in the (max) tropical hypersurface determined by $\trop(g)$, so $g$ has a solution $n_1'$ with $\val(n_1')= - N>0$. Its initial form $\overline{n_1'}$ solves $\overline{a_{21}} \, \overline{n_1'} + \overline{a_{3,0}-a_{03}\,a_{21}^3} = 0$. Again, since $\overline{g'}(\overline{n_1'})\neq 0$, we can find a unique solution to $g(n_1')=0$ in $\KKr$ with the given properties.

  To conclude, we must verify that our choice for $(m_1',n_1')$ with $\val(m_1') = \val(a_{10})$ and $\val(n_1')=\val(a_{3,0}-a_{03}\,a_{21}^3)$
  ensures $\val(\tilde{a}_{20})=\lambda_1$. As before, we write the contributions of $a_{rl}$ to $\tilde{a}_{20}$ producing summands with potential valuations $\leq \lambda_1$. They come from $(r,l) = (1,1)$, $(2,1)$, $(1,2)$, and $(2,0)$. This gives the expression
\[h:=a_{11}n_1'+a_{21}m'_1 + 2n_1'm_1' - a_{20} (|j-2|-1),
  \]
  since $\val(a_{20}) \geq \lambda_1$ and equality holds if, and only if,  $j=2$. Combining the information already collected for $m_1'$ and $n_1'$ with the three choices for $j$ yields $\val(h) = \lambda_1$, thus $\val(\tilde{a}_{20}) = \lambda_1$. In all cases, the minimum valuation among all summands of $h$ is achieved exactly once.   We compute  $\overline{\tilde{a}_{20}}$ from this unique term and obtain~\eqref{eq:initiala11}, as desired.     \end{proof}

\section{Tropical totally real bitangents classes}
\label{sec:trop-totally-real}

The machinery developed in~\cite{len.mar:20} by Len and the second author to determine bitangents to generic smooth plane quartics from their tropical counterparts produces all bitangent triples to a given $\Lambda$ from the local equations at the tropical tangent points. In particular, these techniques allow us to  decide when the tangency points lie in $(\KKr)^2$.

Our main result in this section shows that this occurs whenever $\ell$ is defined over~$\KKr$:

\begin{theorem}\label{thm:totallyKKr} Fix a generic quartic polynomial $q\in \KKr[x,y]$ with $\Gamma=\Trop\,V(q)$ smooth and generic, and let $\Lambda$ be a tropical bitangent to $\Gamma$. Assume $(\ell,p,p')$ is a bitangent lift of $\Lambda$. Then, if $\ell$ is defined over $\KKr$, so are $p$ and $p'$.
  \end{theorem}

\begin{proof} To prove the statement we analyze solutions for local lifting equations for $\Lambda$ and $\Gamma$ at the two tropical tangency points.~\autoref{sec:appendix1} discusses this result in the presence of a multiplicity four tropical tangency. Thus, we focus on the case  where $\Trop\,p\neq \Trop\,p'$.

  The methods developed in~\cite[Section 4]{len.mar:20} determine $(\ell,p,p')$ by computing the initial forms $\overline{m},\bar{n}$ of the coefficients of $\ell$ and the coordinates of $\overline{p}=(\bar{x},\bar{y})$ and $\overline{p'}=(\overline{r},\overline{s})$ as solutions to the systems $\overline{q}=\overline{\ell}=\overline{W}=0$ at both $\Trop\,p$ and $\Trop\,p'$. As we discussed in~\autoref{sec:lift-trop-bitang}, each system determines either $\overline{p}$ or $\overline{p'}$ and one of $\overline{m}$, $\bar{n}$ or $\overline{m}/\bar{n}$, depending on the location of the tropical tangency point in $\Lambda$. The formulas from~\cite[Table 1]{len.mar:20} express these solutions as rational functions in the initial forms  $\overline{a_{ij}}$ and their radicands.

  Furthermore, radicands appear only in the presence of local multiplicity two, as we saw when proving the validity of the formulas in~\autoref{tab:localLifts3a3c}. These radicands feature in the solutions for all three unknowns of each local system: either $\overline{m},\bar{n}$ or $\overline{m}/\bar{n}$ and both entries in $\overline{p}$, respectively $\overline{p'}$. Thus, the local solutions  are real if and only if the initial of the relevant coefficient (or ratio of coefficients) of $\ell$ is real.
  Since the local systems are independent and~\autoref{lm:multivariateHensel} applies to both $\KKr$ and $\KK$, we conclude that $\ell$ is defined over $\KKr$ if and only if the whole triple $(\ell,p,p')$ is.
  \end{proof}

We provide an example to illustrate the close resemblance among the explicit solution formulas for all three unknowns of our local systems, complementing our discussion in the proof of~\autoref{pr:realLift46a}.

\begin{example} We fix a bitangent $\Lambda$ to a tropical plane quartic $\Trop \,V(q)$ with $q\in \KKr[x,y]$ generic, and let $(\ell,p,p')$ be a bitangent triple associated to $\Lambda$. We assume the vertex of $\Lambda$ is the tangency point $\Trop\,p$ and has tangency type (5a) with $(u,v)=(0,0)$, following the notation of~\autoref{fig:LocalNP}. To fix ideas, we place the second tangency point $\Trop\,p'$  along the diagonal end of $\Lambda$, thus it determines $n$ and $p'$. Our first tropical tangency point will fix $m$ and $p$. Assuming both $m,n\in \KKr$, we will show that $p\in (\KKr^*)^2$.

  By~\autoref{lm:multivariateHensel}, 
  it suffices to check $\overline{p}=(\bar{x},\bar{y})\in \RR^2$.
  Our local system in $(\overline{m},\bar{x},\bar{y})$ becomes $\overline{q}=\overline{\ell}=\overline{W}=0$ with:
  \[
  \overline{q} := \overline{a_{10}}\, \bar{x} + \overline{a_{01}}\,\bar{y} +\overline{a_{11}}\,\bar{x}\,\bar{y}, \quad \overline{\ell} := \bar{y} +\overline{m} +\bar{n}\,\bar{x}\,\text{ and }\, \overline{W} := (\overline{a_{10}}+\overline{a_{11}}\,\bar{y})-\bar{n}(\overline{a_{01}} +\overline{a_{11}}\,\bar{x}).
  \]
  We view $\overline{a_{10}}$, $\overline{a_{01}}$, $\overline{a_{11}}$ and $\bar{n}$ as parameters and use \texttt{Singular}~\cite{DGPS} to compute elimination ideals of the ideal $I=\langle \overline{q},\overline{\ell},\overline{W}\rangle$ in a polynomial ring in  $\overline{m}$, $\bar{x}$ and $\bar{y}$ with coefficients in the quotient field $\RR(\overline{a_{10}}, \overline{a_{01}}, \overline{a_{11}},\bar{n})$.

  Eliminating $\bar{x}$ and $\bar{y}$ from $I$ yields the quadratic equation  in $\overline{m}$
  \[\overline{a_{11}}^2 \overline{m}^2 -2 \overline{a_{11}} (\overline{a_{10}} +\overline{a_{01}}\;\bar{n})\overline{m} + (\overline{a_{10}}^2 - 2 \overline{a_{10}}\;\overline{a_{01}}\;\bar{n} + \overline{a_{01}}^2\,\bar{n}^2)=0,
  \]
  which we use to solve for $\overline{m}$. The radicand expression's sign for both solutions is $s_{10}\,s_{01}\operatorname{sign}(\bar{n})$.

  Similarly, eliminating $\overline{m}$ and $\bar{y}$ from $I$ produces a quadratic equation in $\bar{x}$:
  \[
\overline{a_{11}}^2\,\bar{n}\,\bar{x}^2 + 2\,\overline{a_{01}}\,\overline{a_{11}}\,\bar{n}\,\bar{x} + (\overline{a_{01}}^2\,\bar{n} - \overline{a_{01}}\,\overline{a_{10}})=0.  \]
  The sign of the radicand for both solutions is again $s_{10}\,s_{01}\,\operatorname{sign}(\bar{n})$.

  Finally, eliminating  $\overline{m}$ and $\bar{x}$ yields the quadratic equation $\overline{a_{11}}^2\,\bar{y}^2 + 2\,\overline{a_{10}}\,\overline{a_{11}}\,\bar{y} + (\overline{a_{10}}^2 - \overline{a_{01}}\,\overline{a_{10}}\,\bar{n})=0$ in $\bar{y}$. 
The sign of the radicand for each solution $\bar{y}$ is the same as before. From this it follows that if $\overline{m}\in \RR$ then $\bar{x}$ and $\bar{y}$ are also in $\RR$, as we wanted.
\end{example}

 After passing to the tropical limit, we can use these methods to study totally real lifts of tropical bitangents to smooth plane quartics, that is real bitangents lines where the tangency points are also real.
More precisely,~\autoref{thm:totallyKKr} has the following consequence:

\begin{corollary}\label{cor:totallyreal}
  Real lifts of bitangents to generic smooth tropical plane quartics are also totally real.
\end{corollary}

An alternative proof of~\autoref{cor:totallyreal}  can be given using Klein's formula from~\cite{Kle76} relating the number $I$ of real inflection points and the number $B$ of real but not totally real bitangents  of a curve of degree $d$: 
\begin{equation}\label{eq:Klein}
I+2B=d(d-2).
\end{equation}
By \cite[Theorem 5.7]{BLdM11}, any real lift of a generic smooth tropical curve of degree $d$ has $d(d-2)$ real inflection points. Thus, from Klein's formula~\eqref{eq:Klein} it follows that there are no real bitangents which are not totally real.
This result gives rise to the following question:

\begin{question}\label{qn:realNotTotallyReal} What do tropicalizations of quartics with real, but not totally real bitangents look like? Can we determine their skeletons?
\end{question}

From~\autoref{thm:totallyKKr} we can conclude that they cannot be smooth tropical quartics in $\RR^2$. We suspect that they come from faithfully embeddings producing a tropical quartic on a tropical two-dimensional linear space of a higher-dimensional space $\mathbb{R}^n$. 
The study of tropical plane curves on these alternative tropical $2$-dimensional linear spaces, their moduli spaces and their properties is an active topic of research in tropical geometry~\cite{HMRT18}.
We believe that an answer to~\autoref{qn:realNotTotallyReal} will require us to extend the tropical lifting techniques beyond the smoothness and genericity constraints from~\autoref{rem:genericity}, in the spirit of~\cite{LL17}.
We leave this task for further research.

\section*{Acknowledgments}
We wish to thank Erwan Brugall\'e, Alheydis Geiger, Aziz Burak Guelen, Yoav Len, Mario Kummer and Yue Ren for very fruitful
conversations.  
We would like to thank two anonymous referees for their in-depth reports on an earlier version, with numerous suggestions which helped us to improve the exposition of the material.
The first author was supported by  NSF Standard Grants DMS-1700194 and DMS-1954163 (USA). The second author was
supported by the DFG-collaborative research center TRR 195 (INST 248/235-1).
Computations were made in \texttt{Singular}~\cite{DGPS} and~\texttt{Sage}~\cite{sagemath} (using the packages \texttt{tropical.lib}~\cite{JMM07a} and~\texttt{Viro.sage}~\cite{Viro:Sage}),
and the online tool~\cite{Viro:sage2} for combinatorial patchworking by Boulos El Hilany, Johannes Rau and Arthur Renaudineau, available at:

\href{https://matematicas.uniandes.edu.co/~j.rau/patchworking_english/patchworking.html}{https://matematicas.uniandes.edu.co/~j.rau/patchworking\_english/patchworking.html}
                      
Part of this project was carried out during the 2018 program on \emph{Tropical Geometry, Amoebas and Polytopes} at the Institute Mittag-Leffler in Stockholm (Sweden).  The authors would like to thank the institute for its hospitality and  excellent working conditions.

\medskip

\appendix
\section{Multiplicity four local tangencies}
\label{sec:appendix1}

In this section we  compute bitangent lifts $(\ell,p,p')$ in the presence of tropical tangencies of multiplicity four, clarifying some small inaccuracies in part of the proof of~\cite[Theorem 4.1]{len.mar:20}.
We are primarily interested in transverse intersections of multiplicity four at  a vertex of $\Gamma$, which we set as $P$. In this situation, $\Trop\,p=\Trop \,p'=P$. Since we assume $V(q)$ has no hyperflexes, we have $p\neq p'$.

Tropical tangencies of multiplicity four at a vertex of $\Gamma$ appear on bitangent classes of shape (II) on the closure of its two-dimensional cell. They come in various types. This includes the weight one vertices of (II), whose local types are (6b) and (5b).  At the end of this section we discuss the missing  multiplicity four tangencies listed in~\autoref{tab:possibleTangencyPairs}: those corresponding to types (4) and (3b). 

We start by discussing the first two cases. To match our computations with that of~\cite[Theorem 4.1]{len.mar:20} we use $\Sn{3}$-symmetry and work with the image of (II) under the map $\tau_0\circ \tau_1\circ \tau_0$. In particular, we are interested in the vertices of  a bounded edge of $\Trop\, V(q)$ with direction $(4,1)$. We call them $v_l$ and $v_r$.   Their dual triangles in the Newton subdivision of $q$ have vertices $\{(1,0),(0,3), (0,4)\}$ and $\{(1,0), (1,1), (0,4)\}$, respectively. We indicate them as $v_l^{\vee}$ and $v_r^{\vee}$ at the center of~\autoref{fig:mult4}. We let $\Lambda_l$ and $\Lambda_r$ be the corresponding bitangents.

The local equations for $q$ around $v_l$ and $v_r$ become
\begin{equation}\label{eq:localMult4}
  q_l(x,y) = a \,x + b\, y^3 + c\, y^4 \quad \text{ and } \quad q_r(x,y) = a\, x + b'\, xy + c\, y^4.
  \end{equation}

We let $\overline{p} = (\bar{x}, \bar{y})$ and  $\overline{p'}=(\overline{r}, \overline{s})$ be the initial forms of $p$ and $p'$. Even though $p\neq p'$, it is a priori possible that $\overline{p} = \overline{p'}$. 
Our first two lemmas give necessary conditions for $\Lambda_l$ and $\Lambda_r$ to lift to a classical bitangent line $\ell$. In addition, they provide tests to rule out $\overline{p}=\overline{p'}$.

\begin{lemma}\label{lm:vl_NecConditions} If $(\ell,p,p')$ is a bitangent  triple with  $\Trop\,\ell = \Lambda_l$, then  $\overline{m},\bar{n}\in \resK^*$ satisfy
  \begin{equation}\label{eq:conditionsvl}
    (\overline{m},\bar{n})= (\bar{b}/(4\,\bar{c}), 4\,\bar{a}\,\bar{c}^2/\bar{b}^3) \quad\text{ or } \quad(\overline{m}, \bar{n}) =(-\bar{b}/(8\,\bar{c}), 8\,\bar{a}\,\bar{c}^2/\bar{b}^3).
  \end{equation}
Furthermore, if $\overline{p}=\overline{p'}$, then the first option must occur.
\end{lemma}
\begin{proof}
As usual, we assume $q(x,y)$ lies in $R[x,y]\smallsetminus \mathfrak{M}\,R[x,y]$, and $v_l = (0,0)$. In particular, this implies that $a, b, c, m, n \in R$. The local equations at $v_l$ are given by the vanishing of  
  \begin{equation}\label{eq:leftVertexEqn}\overline{q_l}\!:=  \bar{a} \,\bar{x} + \bar{b}\, \bar{y}^3 + \bar{c}\, \bar{y}^4, \quad \overline{\ell}\!:= \bar{y} + \overline{m} +\bar{n}\, \bar{x}  \;\text{ and } \overline{W_l}\!:= \det(Jac(\overline{q_l}, \overline{\ell})) = \bar{a} -\bar{n} \bar{y}^2 (3\bar{b} + 4\,\bar{c}\bar{y}).
  \end{equation}
 We wish to   find necessary and sufficient conditions in $\overline{m},\bar{n} \in \resK^*$ that guarantee two solutions $\bar{x}$ (counted with multiplicity.)  We use $\overline{\ell}$ to eliminate $\bar{y}$ from the system~\eqref{eq:leftVertexEqn}, by setting $\bar{y}=-\overline{m}-\bar{n}\bar{x}$. This leads to an ideal $I$ in $\resK[\overline{m}^{\pm},\bar{n}^{\pm}][\bar{x}^{\pm}]$ generated by the following two polynomials:
  \begin{equation*}
    \begin{aligned}
  \overline{q_l}'& =
  \bar{c}\,\bar{n}^4\bar{x}^4\! +\! \bar{n}^3(4\,\bar{c}\,\overline{m}\! -\! \bar{b})\bar{x}^3\! +\! 3\overline{m}\bar{n}^2(2\bar{c}\,\overline{m}\bar{n}^2 \!-\! \bar{b})\bar{x}^2 \!+\! (4\bar{c}\,\overline{m}^3\bar{n} - 3\bar{b}\overline{m}^2\bar{n} + \bar{a})\bar{x} \!+\! \overline{m}^3(\bar{c}\overline{m}\! -\! \bar{b}),\\
 \overline{W_l}' & = 
4\bar{c}\bar{n}^4\bar{x}^3 + 3\bar{n}^3(4\bar{c}\overline{m} - \bar{b})\bar{x}^2 + 6\overline{m}\bar{n}^2(2\bar{c}\overline{m}\bar{n}^2 - \bar{b})\bar{x} + (4\bar{c}\overline{m}^3\bar{n} - 3\bar{b}\overline{m}^2\bar{n} + \bar{a}).    \end{aligned}
  \end{equation*}
  Simple manipulations of $\overline{q_l}'$ and $\overline{W_l}'$ in~\sage~\cite{sagemath} produce new elements in $I$ of lower degree in $\bar{x}$. In particular, we obtained the following polynomials in $I$:
  \begin{equation*}
    \begin{aligned}
      f_l&:=12 \bar{a}\,\bar{n}^4(4\bar{c}\overline{m}-\bar{b})
  \big (2\,\bar{n} (2\bar{b}^2\bar{c}\overline{m}\bar{n} + \bar{b}^3\bar{n} - 6\bar{a}\bar{c}^2)\bar{x} + (\bar{b}^2\overline{m}\bar{n} + \bar{a}\bar{c})\big ),\\
  g_l&= (3\,\bar{b}^2\bar{n}^3)\bar{x}^2 + 6\,\bar{n}(\bar{b}^2\overline{m}\,\bar{n} - 2\,\bar{a}\,\bar{c})\bar{x} + (3\,\bar{b}^2\overline{m}^2\bar{n} + 4\,\bar{a}\,\bar{c}\,\overline{m} - \bar{a}\,\bar{b}).
    \end{aligned}
  \end{equation*}

  Since $p$ and $p'$ are tangencies points and $(\overline{m},\bar{n}) \in (\resK^*)^2$,  the ideal $I$ cannot contain a non-trivial linear polynomial in $\bar{x}$. Thus,  $f_l$ must be the zero polynomial. This happens if and only if either $\overline{m} = \bar{b}/4\bar{c}$ (and $\bar{n}$ is free), or 
    \[2\,\bar{b}^2\,\bar{c}\,\overline{m}\,\bar{n} + \bar{b}^3\bar{n} - 6\,\bar{a}\,\bar{c}^2 = \bar{b}^2\,\overline{m}\,\bar{n} + \bar{a}\,\bar{c} \!=\! 0.\] 
The last system  admits a unique joint solution, namely the second option listed in~\eqref{eq:conditionsvl}.

When $\overline{m} = \bar{b}/4\bar{c}$, the following  combination of  $\overline{W_l}$ and $g_l$, evaluated accordingly, yields:
\[((2\,\bar{b}^2\,\bar{c}\,\bar{n}^2\,\bar{x} - \bar{b}^3\,\bar{n} + 8\,\bar{a}\,\bar{c}^2)\,\overline{W_l} - \bar{b}^4\,g_l)_{|{\overline{m} = \bar{b}/4\bar{c}}} = 128\,(\bar{b}^3\,\bar{n} - 4\,\bar{a}\,\bar{c}^2)\,\bar{a}\,\bar{c}^3\,\bar{x} = 0.
\]
Since $\bar{x}\neq 0$, this implies $\bar{n} = 4\,\bar{a}\,\bar{c}^2/\bar{b}^3$, as~\eqref{eq:conditionsvl} indicates.

     Since $\overline{p}=\overline{p'}$ if and only if   $\bar{x}$ is a double root of $g_l$, we conclude that the discriminant of $g_l$ in the variable $\bar{x}$ must vanish. A direct computation gives:
          \begin{equation*}\label{eq:discrL}
        (-12\,\bar{a}\bar{n}^2)^{-1} \Delta(g_l) = \bar{b}^2(16\,\bar{c}\,\overline{m} - \bar{b})\,\bar{n} - 12\,\bar{a}\,\bar{c}^2).
          \end{equation*}
         After checking both options from~\eqref{eq:conditionsvl} we see that  $\Delta(g_l)$ vanishes only for the first one. 
\end{proof}

\begin{lemma}\label{lm:vr_NecConditions}If $(\ell,p,p')$ is a bitangent triple and $\Trop\,\ell = \Lambda_r$, then $\overline{m},\bar{n} \in \resK^*$ satisfy
  \begin{equation*}\label{eq:conditionsvr}
    (\overline{m},\bar{n})= (-\bar{a}/\overline{b'}, -\overline{b'}^3/(4\,\bar{a}^2\,\bar{c}^2)) \text{ or } (\bar{a}\,(7\pm 4\sqrt{2}\,i)/(9\,\overline{b'}),\overline{b'}^3(7 \mp 4\sqrt{2}\,i)/(108\,\bar{a}^2\,\bar{c})).   
  \end{equation*}
  Furthermore, if $\overline{p}=\overline{p'}$, then the first pair cannot occur.
\end{lemma}
\begin{proof}
  We follow the same strategy as in the proof of~\autoref{lm:vl_NecConditions}, with $v_r=(0,0)$ and  $q\in R[x,y]\smallsetminus \mathfrak{M}R[x,y]$, so all $a,b',c,m$ and $n$ lie in $R$. In this case, 
\begin{equation}\label{eq:rightVertexEqn}
\overline{q_r} = \bar{a}\,\bar{x}+\overline{b'}\bar{x}\,\bar{y} + \bar{c}\bar{y}^4, \quad \overline{\ell}\!:= \bar{y} + \overline{m} +\bar{n}\, \bar{x}  \;\text{ and }\; \overline{W_r}=(\bar{a}+\overline{b'}\bar{y}) - \bar{n}(\overline{b'}\bar{x} + 4\bar{c}\bar{y}^3).
\end{equation}
We substitute $\bar{y}=-\overline{m}-\bar{n}\bar{x}$ in both $\overline{q_r}$ and $\overline{W_r}$
and obtain the ideal $I_r=\langle \overline{q_r}', \overline{W_r}'\rangle$ in $\resK[\overline{m}^{\pm},\bar{n}^{\pm}][\bar{x}^{\pm}]$. Algebraic manipulations  yield two polynomials in $I_r$ of low-degree in $\bar{x}$:
\begin{equation*}
  \begin{aligned}
    f_r&:= -2\overline{m}\big(\bar{n}(\bar{c}(9\bar{b}^2\overline{m}^2 \!+\! 10\bar{a}\bar{b}\overline{m} \!+\! 9\bar{a}^2)\bar{n} \!-\! 2\bar{b}^3)\bar{x}\! +\!    (\bar{c}\overline{m}(9\bar{b}^2\overline{m}^2 \!-\! 2\bar{a}\bar{b}\overline{m} \!-\! 3\bar{a}^2)\bar{n} \!+\! \bar{b}^2(  \bar{a} \!-\!\bar{b}\overline{m})\big),\\
    g_r&:= (2\bar{b}\bar{n}^2)\bar{x}^2 + \bar{n}(\bar{b}\overline{m} - 3\bar{a})\bar{x} + \overline{m}(  \bar{a}- \bar{b}\overline{m} ).
  \end{aligned}
  \end{equation*}
Since $\ell$ is a bitangent, $I$ has no non-trivial linear polynomials. Thus $f_r=0$, so both its coefficients in $R[\overline{m}^{\pm}, \bar{n}^{\pm}]$ must vanish. Call them $A_1$ and $A_0$, depending on the degree of $\bar{x}$. A simple algebraic manipulation gives:
\[ A_0-\overline{m}\,A_1=-\,(\overline{b'}\,\overline{m}+\bar{a})(12\,\bar{a}\,\bar{c}\,\overline{m}\,\bar{n}-\overline{b'}^2)=0.
\]
Assuming $\overline{m}=-\bar{a}/\bar{b}$  and solving $A_0=A_1=0$ for $\bar{n}$ gives the first option in~\eqref{eq:conditionsvr}.
In turn, setting $\bar{n} = \overline{b'}^2/(12\,\bar{a}\,\bar{c}\,\overline{m})$ and substituting this expression in both  $A_0$ and $A_1$ yields a single non-monomial factor that must vanish. This is precisely  the discriminant $\Delta(g_r)\!=\! 9\overline{b'}^2\overline{m}^2 - 14\bar{a}\overline{b'}\overline{m} + 9\bar{a}^2$. Replacing the two roots of $\Delta(g_r)$ back into the expression for $\bar{n}$ gives the second option in~\eqref{eq:conditionsvr}.

Finally, a direct computation gives $\Delta(g_r)(-\bar{a}/\overline{b'}, -\overline{b'}^3/(4\,\bar{a}^2\,\bar{c}^2)) = 32\bar{a}^2\neq 0$. Thus, this solution in $\overline{m},\bar{n}$ is not valid whenever $\overline{p}=\overline{p'}$. 
  \end{proof}

The next proposition  provides explicit formulas to lift the bitangent classes associated to $v_l$ and $v_r$ to unique bitangent triples $(\ell,p,p')$ satisfying $\overline{p}\neq \overline{p'}$. Each lifting is unique, as was stated in~\cite[Theorem 4.1]{len.mar:20}.

\begin{proposition}\label{pr:liftingFormulasMult4} Assume $V(q)$ has no hyperflexes. Then, the tropical bitangents $\Lambda_l$ and $\Lambda_r$ each lift uniquely to classical bitangent triples $(\ell,p,p')$ with $\overline{p}\neq \overline{p'}$. \autoref{tab:liftingMult4} gives explicit formulas for the tuple of initial forms $\overline{m},\bar{n}$, $\overline{p}$ and $\overline{p'}$ in each case.
\end{proposition}

\begin{proof} 
  {Lemmas}~\ref{lm:vl_NecConditions} and~\ref{lm:vr_NecConditions} yield unique potential values for $(\overline{m},\bar{n})$ where $\overline{p}\neq\overline{p'}$. We substitute each value in the three equations from~\eqref{eq:leftVertexEqn}, respectively~\eqref{eq:rightVertexEqn}, and solve for $\bar{x},\bar{y}$. We obtain two distinct solutions in $(\bar{x},\bar{y})$, each corresponding to $\overline{p}$ and $\overline{p'}$. They can be seen in \autoref{tab:liftingMult4}.

  To conclude we use \autoref{lm:multivariateHensel} to show that these initial forms lift uniquely to a triple $(\ell,p,p')$ with $q(p)=\ell(p)=W(p)=q(p')=\ell(p')=W(p')=0$. We use $q=q_l$ or $q_r$ and $W=W_l$ or $W_r$ depending on the location of the tropical bitangent line.

  A~\sage~\cite{sagemath} computation shows that the initial forms of the Jacobians  $(\overline{m},\bar{n},\bar{x},\bar{y},\overline{r},\overline{s})$ for the local systems at $v_l$ and $v_r$  are Laurent monomials in the coefficients of~\eqref{eq:localMult4}, namely $-9\sqrt{3}\,\bar{a}^3\bar{b}^2/\bar{c}$ and $128\sqrt{2}\,\bar{a}^5\bar{c}/\overline{b'}^2$. These expressions are units in $\resK$, so we have unique lifts from the six initial forms to the data $(\ell,p,p')$.
\end{proof}

  \begin{table}[tb]
  \begin{tabular}{|c||c|c|c|c|c|c|}
    \hline $\Trop\, \ell$ &  $\overline{m}$ & $\bar{n}$ & $\bar{x}$ & $\bar{y}$ & $\overline{r}$ & $\overline{s}$  \\
    \hline \hline
    
    $\Lambda_l$ & $-\frac{\bar{b}\,\bar{c}}{8}$ & $\frac{8\bar{a}\bar{c}^2}{\bar{b}^3}$ & $\frac{\bar{b}^4(3+2 \sqrt{3})}{64\,\bar{a}\bar{c}^3}$ & $-\frac{\bar{b}\,(1+\sqrt{3})}{4\,\bar{c}}$ & $\frac{\bar{b}^4(3-2 \sqrt{3})}{64\,\bar{a}\bar{c}^3}$ & $-\frac{\bar{b}\,(1-\sqrt{3})}{4\,\bar{c}}$\\
    \hline
    $\Lambda_r$ & $-\frac{\bar{a}}{\overline{b'}}$ & $\frac{\overline{b'}^3}{4\bar{a}^2\bar{c}}$ & $\frac{4\,\bar{a}^3\bar{c}(1+\sqrt{2})}{\overline{b'}^4}$ & $-\frac{\sqrt{2}\,\bar{a}}{\overline{b'}}$ &
    $\frac{4\,\bar{a}^3\bar{c}(1-\sqrt{2})}{\overline{b'}^4}$ & $\frac{\sqrt{2}\,\bar{a}}{\overline{b'}}$\\
    \hline 
  \end{tabular}
  \caption{Formulas for the initial forms of triples $(\ell,p,p')$ lifting $\Lambda_l$ and $\Lambda_r$.\label{tab:liftingMult4}}
\end{table}

Our next result shows that no bitangent lift $(\ell,p,p')$ of  $\Lambda_l$ or $\Lambda_r$ has $\overline{p}=\overline{p'}$ under the genericity conditions from~\autoref{rem:genericity}. This rules out the overlooked case  in the proof of~\cite[Theorem 4.1]{len.mar:20}. Even though this situation can be discarded using the classical count of bitangents to smooth plane quartics, we provide an independent combinatorial proof.
\begin{theorem}\label{thm:liftingMult4} Assume $V(q)$ has no hyperflexes. Fix a  bitangent triple  $(\ell,p,p')$  where $\Trop\, \ell$ is either $\Lambda_l$ or $\Lambda_r$. Then,  $\overline{p} \neq \overline{p'}$.
\end{theorem}
\begin{proof} 
  We argue by contradiction. We assume $q\in R[x,y]\smallsetminus \mathfrak{M}R[x,y]$ and we let $v=(0,0)$ be the vertex of the tropical line. The proof for both $v=v_l$ or $v_r$ is the same, so we argue for both vertices simultaneously. 
  By 
  {Lemmas}~\ref{lm:vl_NecConditions} and~\ref{lm:vr_NecConditions} we have a unique choice for $(\overline{m}, \bar{n})$ that is compatible with $\overline{p}=\overline{p'}$. Replacing these values in~\eqref{eq:leftVertexEqn} and \eqref{eq:rightVertexEqn} yields 
  \[ (\overline{m},\bar{n},\overline{p}) =
  \begin{cases}
    (\frac{\bar{b}}{4\,\bar{c}}, \frac{4\,\bar{a}\,\bar{c}^2}{\bar{b}^3}, \frac{\bar{b}^4}{16\,\bar{a}\bar{c}^3}, -\frac{\bar{b}}{2\,\bar{c}}) & \text{ for }\Lambda_l,\\
    (\frac{\bar{a}\,(7\pm 4\sqrt{2}\,i)}{9\,\overline{b'}},\frac{\overline{b'}^3(7 \mp 4\sqrt{2}\,i)}{108\,\bar{a}^2\,\bar{c}},
    \frac{4\,(43 \pm 13\sqrt{2}i)}{27\,\bar{a}^3\overline{b'}^4\bar{c}},
-\frac{(4 \pm \sqrt{2}i)}{3\,\bar{a}\,\overline{b'}})  & \text{ for }\Lambda_r.
  \end{cases}
  \]
We write $p=(r,s)$ and $p'=(r',s')$. Since $p\neq p '$ both lie in $\ell$, we have $r\neq r'$ and $s\neq s'$. We fix $N\gg \max\{\val(r-r'), \val(s-s')\}$ and re-embed $V(q)$ and $V(\ell)$  via the ideals
  \begin{equation*}\label{eq:idealsMult4}
      I := \langle q, x'-(x-r-\varepsilon_1) ,y'-(y-s-\varepsilon_2)\rangle  \text{ and }
      I_\ell := \langle \ell, x'-(x-r-\varepsilon_1) ,y'-(y-s-\varepsilon_2)\rangle,
 \end{equation*}
  where $\varepsilon_1, \varepsilon_2$ are two suitable parameters in $\KK$ of valuation $N$. This transformation corresponds to performing tropical modifications of $\RR^2$ along $\max\{X,0\}$ and $\max\{Y,0\}$.

  We write $\tilde{q}(x',y') := q(x'+r+\varepsilon_1,y'+s +\varepsilon_2)$ and $\tilde{\ell}(x',y'):= y' + m' + n\,x'$ where $m':=(m + s + \varepsilon_2 + n(r +\varepsilon_1))$. The two tangency points become $p_1:= (-\varepsilon_1, -\varepsilon_2)$, $p'_1 := (r'-r-\varepsilon_1, s'-s-\varepsilon_2)$. Note that $\val(m')>\val(m)$,  $\Trop\,p_1 = (-N,-N)$ and  $\Trop\,p'_1 = (\val(r'-r), \val(s'-s))$ both lie in  $(\RR_{<0})^2$. Our choice  $\varepsilon_1, \varepsilon_2$ is determined by two conditions: $\tilde{q}(0,0)\neq 0$ and $p_1, p'_1\in (\KK^*)^2$. Both can be attained for generic  $\varepsilon_1, \varepsilon_2\in R\smallsetminus \mathfrak{M}$ of valuation $N$.

  \autoref{fig:mult4} shows the relevant pieces of the Newton subdivisions of $\tilde{q}$, $q$, and the tropical curves $\Gamma':=\Trop\,V(\tilde{q})$ and  $\Lambda':=\Trop\, V(\tilde{\ell})$.  We can certify this by confirming that the expected valuations of $x'$, $y'$ and $(y')^4$ for $v_l$ (respectively,  $x',x'y', y'$ and $(y')^4$ for $v_r$) are attained, whereas the valuation of the constant term is higher than expected. Furthermore, the initial forms of $\tilde{a}_{10}$ and $\tilde{a}_{01}$ in both cases depend only on $\overline{p}$ and the coefficients of  $q$.

By construction, the tropical tangency points $\Trop\,p_1$ and $\Trop\,p'_1$ of $\Lambda'$ are distinct and  belong to the slope-one  edge $e$ of $\Gamma'$ adjacent to $v$. We let $v'$ be the second endpoint of $e$. The location of the vertex of $\Lambda'$ is uncertain, but it must lie in $\RR_{\leq 0}(1,1) + \Trop \,p_1$. We analyze two cases, depending on the  positions of $v'$ and $\Trop\,p_1$.

First, assume both $\Trop\,p_1$ and $\Trop\,p_1'$ lie in the relative interior of $e$. Then, the local equations for $\tilde{q}, \tilde{\ell}$ and their Wronskian at both points would yield $\overline{n'} =\overline{\tilde{a}_{10}}/\overline{\tilde{a}_{01}}$, and $\overline{\varepsilon_2} + \overline{\varepsilon_1}\, \overline{\tilde{a}_{10}}/\overline{\tilde{a}_{01}}  =0$, contradicting our genericity assumptions on the parameters $\varepsilon_1,\varepsilon_2$.

Finally, assume $\Trop\,p_1=v'$ and  consider the fiber over $v'$ of the tropicalization map $\trop\colon V(\tilde{q}) \to \RR^2$. By~\cite[Corollary 4.2]{pay:09}, this fiber is Zariski dense in $V(\tilde{q})$. Furthermore, $\overline{\tilde{q}}_{v'}\in \resK[x',y']$ vanishes along  initial forms of all points in the fiber. Thus, we pick a point $u=(r'',s'')$ in the fiber with $\overline{r''}\neq \overline{-\varepsilon_1}$ and $\overline{s''}\neq \overline{-\varepsilon_2}$. We modify along $\max\{X',-N\}$ and $\max\{Y',-N\}$ followed by a re-embedding using $x'' = x' - r'' - \varepsilon_1'$, $y''=y'-s''-\varepsilon_2'$ for generic parameters $\varepsilon_1', \varepsilon_2'$ of valuation $>N$. This has a simple effect on the projection to the $(x'',y'')$-plane: it prolongs  the edge $e$  in the $(-1,-1)$ direction while leaving $\Trop\,p_1$ and $\Trop\,p'_1$ fixed. The result then follows from the earlier case analyzed above. 
\end{proof}

\begin{figure}[htb]
  \includegraphics[scale=0.5]{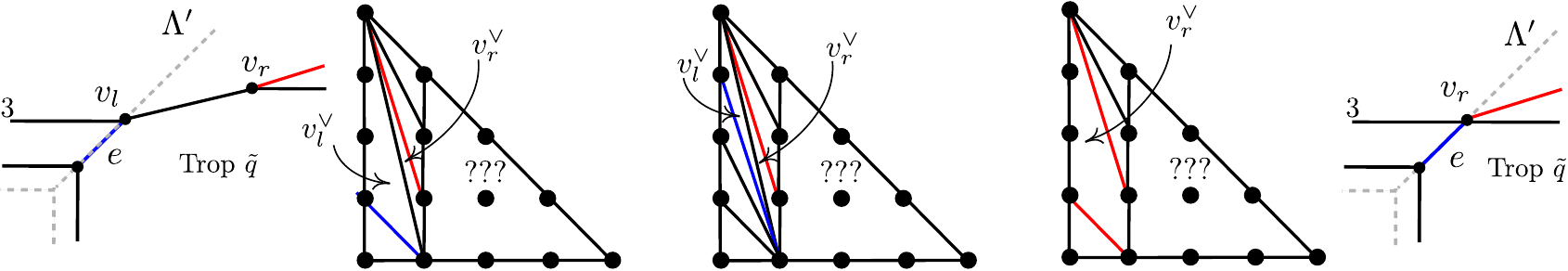}
  \caption{From left to right:  $\Trop\, V(\tilde{\ell})$, $\Trop\, V(\tilde{q})$ and Newton subdivision of $\tilde{q}$ when $\overline{p}=\overline{p'}$ and the tangency $\Trop\,p=\Trop\,p'$ is $v_l$ or $v_r$.\label{fig:mult4}}
\end{figure}

Finally, by combining~\autoref{pr:liftingFormulasMult4} and~\autoref{thm:liftingMult4} we conclude that the bitangents $\Lambda_l$ and $\Lambda_r$ both lift with  multiplicity one, as was asserted in~\cite[Theorem 4.1]{len.mar:20}.

\begin{corollary}\label{cor:mult1forvl-vr} The bitangent $\Lambda_l$ and $\Lambda_r$  have lifting multiplicity one.
\end{corollary}

The techniques discussed above can be used to study multiplicity four tangencies of types (4) and (3b). The first situation arises for members of a shape (II) bitangent class inside the bounded edge in the boundary of its two-cell. We show: 
\begin{proposition}\label{pr:mult4type4}
  Assume a bitangent line $\Lambda$ to $\Gamma$ has a local tangency of type (4) and multiplicity four. Then, $\Lambda$ does not lift to a bitangent triple for $V(q)$.
\end{proposition}
\begin{proof} We argue by contradiction and write $(\ell,p,p')$ for the bitangent triple associated to $\Lambda$. Up to $\Sn{3}$-symmetry, we may assume the bitangent class of $\Lambda$ has shape (II). As usual, we set $v=(0,0)$ to be the vertex of $\Lambda$ and fix $q\in R[x,y]\smallsetminus \mathfrak{M}R[x,y]$. We let $e$ be the edge of $\Gamma$ containing $v$ in its relative interior. As~\autoref{fig:NP} shows,  $e^{\vee}$ has endpoints $(0,0)$ and $(1,3)$. The local system at $(0,0)$ is defined by the vanish of  the equations 
  \[
  \overline{q_v} = \bar{a} + \bar{b} \bar{x}\bar{y}^3, \quad \overline{\ell}\!:= \bar{y} + \overline{m} +\bar{n}\, \bar{x}  \;\text{ and }\; \overline{W_v}=
  -3\bar{b}\bar{n}\bar{x}\bar{y}^2 + \bar{b}\bar{y}^3.
  \]
  Substituting the value for $\bar{y}$ obtained from   $\overline{\ell}=0$ into the other two equations yields an ideal $I\subset \resK[\overline{m}^{\pm}, \bar{n}^{\pm}][\bar{x}^{\pm}]$ with two generators with a multiplicity two solution at   $\bar{x}$. Algebraic manipulations as those used in the proof of~\autoref{lm:vl_NecConditions} produce a non-zero linear polynomial in $I$, namely
  $h=(48\bar{a}\bar{b}^2\overline{m}\bar{n}^5)\bar{x} + (12\bar{a}\bar{b}^2\overline{m}^2\bar{n}^2)$. This cannot happen. 
  \end{proof}

To conclude, we discuss tangencies of type (3b) of multiplicity four outside shape (C). Assume $\Lambda$ is a bitangent line with vertex $v$ and $\Lambda\cap \Gamma$ is a bounded horizontal edge $e$ with rightmost vertex $v$ with stable intersection multiplicity four. This situation arises for the valency two vertices of weight zero of bitangent classes with shape (D), (L), (L') or (O) through (S). Either by chip-firing on the stable intersection between $\Gamma$ and $\Lambda$ or by analyzing the tangencies for members in a neighborhood of $\Lambda$ we can conclude that $\Lambda$ has two multiplicity two tangencies, namely,the midpoint and the rightmost vertex   of $e$.
As~\autoref{fig:NP} shows, the  local equation for $q(x,y)$ at $v$ becomes
$q_v(x,y)= ax +b xy + c x^2y^2$.
Our final result discusses possible lifts of $\Lambda$.

\begin{proposition}\label{pr:mult4type3b} Let $\Lambda$ be a bitangent triple to  $\Gamma$ that realizes a local tangency of type (3b) at $v$ with two tangencies: one at $v$ and one in the relative interior of the adjacent bounded horizontal edge. Then, $\Lambda$ cannot be lifted to a bitangent triple.
  \end{proposition}
\begin{proof} We argue by contradiction, and fix a bitangent lift $(\ell,p,p')$ of $\Lambda$ with $\Trop\,p=v$. We set $\overline{p}=(\bar{x},\bar{y})$. As usual, we assume $q\in R[x,y]\smallsetminus \mathfrak{M}R[x,y]$ and $v=(0,0)$. Thus, $\val(a)=\val(b)=\val(c)=\val(m)=\val(n)=0$. Using the local equations for the tangency along the relative interior of the horizontal bounded edge we conclude that $\overline{m} = \bar{a}/\bar{b}$. An algebraic manipulation of the local equations at $v$ (in the variables $\bar{x},\bar{y}$ and $\bar{n}$)  using the same techniques as those in the proof of~\autoref{lm:vl_NecConditions} implies $\bar{n} = \bar{a}^2\bar{c}/\bar{b}^3$. Replacing this value in the class of $q_v$ and $W_v$ in  $\resK[\bar{x}^{\pm}, \bar{y}^{\pm}]$ yields an inconsistent system of equations for $\bar{x}$, namely
  $\bar{a}\,\bar{c}\,\bar{x} + 2\,\bar{b}^2 = 2\,\bar{a}\,\bar{c}\,\bar{x} + 3\,\bar{b}^2 = 0.$ Thus, $\Lambda$ cannot be lifted over $\KK$.
  \end{proof}

\normalsize

  \vspace{2ex}
  

  \noindent
\textbf{\small{Authors' addresses:}}
\smallskip
\

\noindent
\small{M.A.\ Cueto,  Mathematics Department, The Ohio State University, 231 W 18th Ave, Columbus, OH 43210, USA.
\\
\noindent \emph{Email address:} \texttt{cueto.5@osu.edu}}
\vspace{2ex}

\noindent
\small{H.\ Markwig, Eberhard Karls Universit\"at T\"ubingen, Fachbereich Mathematik, Auf der Morgenstelle 10, 72108 T\"ubingen, Germany.
  \\
  \noindent \emph{Email address:} \texttt{hannah@math.uni-tuebingen.de}}

\end{document}